\documentclass[12pt]{amsart}
   \newtheorem{thm}{Theorem}[section]
  \newtheorem{prop}[thm]{Proposition}
  \newtheorem{lem}[thm]{Lemma}
  \newtheorem{cor}[thm]{Corollary}
  \theoremstyle{definition}
 \theoremstyle{remark}
 \newtheorem{rmk}[thm]{Remark}
  
 \newtheorem{examp}[thm]{Example}
 
 \newtheorem{claim}[thm]{Claim}

 \usepackage{amsfonts}
 \usepackage{amsmath}
 \usepackage{calligra}
 \usepackage{amssymb}
 \usepackage{verbatim}
 \usepackage[colorlinks]{hyperref}
 \usepackage[T1]{fontenc}
 \usepackage{mathrsfs}
\usepackage{color}
 \usepackage[utf8]{inputenc}
 \usepackage[makeroom]{cancel}
 \usepackage{tikz-cd}
 \input xy
 \xyoption{all}
 
   \usepackage{mathtools}

 \newcommand{\e}{\mathfrak{e}}
 
 \newcommand{\g}{\mathfrak{g}}
 \newcommand{\h}{\mathfrak{h}}
 \renewcommand{\k}{{\mathfrak k}}
 \renewcommand{\l}{\mathfrak l}
 
 \newcommand{\p}{\mathfrak p}
 \newcommand{\q}{\mathfrak q}
 \newcommand{\s}{\mathfrak s}
 \newcommand{\z}{\mathfrak z}
 \renewcommand{\t}{{\mathfrak t}}
 \renewcommand{\u}{{\mathfrak u}}
 
 \newcommand{\C}{{\mathbb C}}

 \newcommand{\R}{{\mathbb R}}
 \newcommand{\Z}{{\mathbb Z}}

 \numberwithin{equation}{section}
 \DeclareFontFamily{U}{mathx}{\hyphenchar\font45}
 \DeclareFontShape{U}{mathx}{m}{n}{
       <5> <6> <7> <8> <9> <10>
       <10.95> <12> <14.4> <17.28> <20.74> <24.88>
       mathx10
       }{}
 \DeclareSymbolFont{mathx}{U}{mathx}{m}{n}
 \DeclareFontSubstitution{U}{mathx}{m}{n}
 \DeclareMathAccent{\widecheck}{0}{mathx}{"71}
 \DeclareMathAccent{\wideparen}{0}{mathx}{"75}

\begin{document}
 \title{ Symmetry breaking differential operators and Discrete Series}
\author{  Bent {\O}rsted,   Jorge A.  Vargas}
\thanks{Partially supported by Aarhus University  (Denmark),  CONICET
(Argentina)  }
\date{\today }
 \address{ Mathematics Department, Aarhus University, Denmark; FAMAF-CIEM, Ciudad Universitaria, 5000 C\'ordoba, Argentine}
\email{orsted@imf.au.dk,  jorge.a.vargas@unc.edu.ar}

\begin{abstract}
For a semisimple Lie group $G$ satisfying the
equal rank condition, the most basic family of unitary
irreducible representations is the Discrete Series found
by Harish-Chandra. In this paper, we study the structure of symmetry breaking operators  for Discrete Series when restricted to a subgroup
$H$ of the same type by combining classical results
with recent work of T. Kobayashi, Nakahama and Pevzner. This we do by using reproducing
kernels for the representations and our previous duality principle in order to find some explicit
details on the nature of the differential operators representing symmetry breaking operators,
in particular to what extent  they are given by differentiations in normal directions to the $H$-orbit in $G/K$.
\end{abstract}
\maketitle
\markboth{{\O}rsted- Vargas}{Diff. operators and Discrete Series}

\tableofcontents
\section{Introduction}

 Differential operators play several key roles in representation theory of Lie groups and in the study of
symmetries. In this paper we are motivated by and interested in the setting of holomorphic vector bundles
over complex homogeneous spaces $X$ and differential operators acting on their sections. Here $X$
could be a Hermitian symmetric space, thought of as a bounded symmetric domain in an $n$-dimensional
complex linear space, so that $X = G/K$ for a semisimple Lie group $G$. We also consider a closed subgroup
$H$ of $G$ of the same type and its orbit $Y$ through the origin in $X$, so that we have the pair of symmetric domains
$$Y = H/L \subset X = G/K. $$
We wish to find differential operators $D$ acting on (sections over) $X$ that exhibit symmetries with respect to $H$,
namely in the sense that composing with restriction to $Y$ yields an $H$-equivariant linear operator, a
so-called symmetry breaking operator. Note here the key result Proposition~\ref{prop:diffarehomo} to the effect that all differential
symmetry breaking operators are in a precise sense homogeneous.

Based on our previous work on branching theory we study the nature in detail of such differential operators $D$,
in particular to what extent they can be thought of as purely tangential  to $Y$, normal to $Y$, or a mixture.
It turns out that for low-order operators the detailed structure, using root systems, of the domains can
be effectively applied to such questions. A key notion is here is for the case of $H$ being  a symmetric subgroup
of $G$ where we consider a certain natural dual symmetric subgroup $H_0$.

Some of the basic results are of an algebraic nature, perhaps of independent interest, and also giving insight
about the geometry of the pair of symmetric domains.

 More precisely, in this paper we continue our study of the program of T. Kobayashi on symmetry breaking; our aim is to
make explicit the construction of differential operators leading to symmetry breaking in the
setting of discrete series representations. We focus mainly on holomorphic discrete series representations,
as already here there are many structural facts about their realization available; we combine some geometric
results about the imbedding of one bounded symmetric domain into another one, such as coordinates that
are partly normal and partly tangential, and also consider the algebraic relations between the respective root systems.
The main results are allowing finding in particular purely normal differential symmetry breaking operators.

This paper consists of several sections, in the   section 4 we elaborate on the algebraic counterpart of symmetry breaking operators represented by differential operators, in    section 3 we apply the results obtained in   sections 4 to the actual study of the symmetry breaking operators. In section 2, we present  background material.

An aim of this paper   is to continue the initial work of Kobayashi-Pevzner \cite{KP1}\cite{KP2}\cite{KP3}  to understand the
structure of symmetry breaking operators in the realm of
      $H$-admissible Discrete Series. To describe our results, we present some considerations on the structure of symmetry breaking operators. To begin with, we recall that Discrete Series for $G$ (resp. $H$) admits a realization as space of  square integrable smooth sections $H^2(G,\tau)$ of homogeneous vector bundles   $G\times_\tau W \rightarrow G/K$ (resp. $ H\times_\sigma Z \rightarrow H/L$) attached to the corresponding lowest $K$-type  (resp. $L$-type). We fix   a symmetry breaking operator $S:H^2(G,\tau)\mapsto H^2(H,\sigma)$, represented by the kernel $K_S :G\times H \rightarrow Hom_\C( W, Z)$ and recall  the subspace $Z_S:=Image(Z\ni z \mapsto K_{S^*}(e, \cdot) (z)=K_S(\cdot, e)^\star z \in H^2(G,\tau))$, $Z_S$ is a $L$-irreducible subspace \ref{sub:zs}, \cite{OV3}. Whence, either $Z_S \cap \mathcal U(\h_0)W=\{0\}$ or $Z_S \subset \mathcal U(\h_0)W$. In \cite{OV3}, under the hypothesis of the restriction of the representation $H^2(G,\tau)$ is $H$-admissible, it is shown that in the latter case $S$ is represented by a normal derivative differential operator, whereas in the former case $S$ is represented by a differential operator that never will be a normal derivative differential operator. To be more precise, we write

    $Hom_H(V,V_\sigma) =\{ S: Z_S \subset \mathcal U(\h_0)W\}\cup \{ S: Z_S \cap \mathcal U(\h_0)W=\{0\}\}$.

  After ignoring the zero operator,  this is a disjoint union,  the first subset is the one that contains the symmetry breaking operators that are represented by normal derivatives operators and the second subset is its complement. To continue, we recall the subspace   $\mathcal L_{W,H}^c \subset H^2(G,\tau)_{K-fin}$ determinate by the linear span of the subspaces $Z_S$ where $(\sigma, Z)$ runs over the set of lowest $L$-types of the irreducible factors of the restriction of $H^2(G,\tau)$ to $H$, and, $S$ runs over the totality of symmetry breaking operators for $H^2(G,\tau)$.

   Roughly speaking,  the first subset in the decomposition of above consists of the lowest $L$-type subspaces contained in     $ \mathcal L_{W,H}^c\cap \mathcal U(\h_0)W  $ and measures the "quantity" of symmetry breaking operators   represented by normal derivatives operators, whereas, $ \mathcal L_{W,H}^c \backslash (\mathcal L_{W,H}^c\cap \mathcal U(\h_0)W)$,  "measures"  the totality  of "non normal" derivative symmetry breaking operators. Note here formula \ref{eq:sasdfo} which links a symmetry breaking operator to a differential operator representing it and  \ref{sub:tech}, \ref{sub:disjoint}, Proposition~\ref{prop:everynorfortriple}.

  Obviously,  if  $H^2(H,\sigma)$ has multiplicity one in $H^2(G,\tau)$, the symmetry breaking operators in $Hom_H(V_\tau^G,V_\sigma^H)$ are either normal derivatives differential operators or not.\cite{KP1}, \cite{KP2} has constructed examples of both possibilities. For a scalar holomorphic  representation $\tau$ and $(\g,\h)$ a real rank one pair,  in \cite[Theorem 5.3]{KP2} it is shown: Either every   nontrivial symmetry breaking operators is represented by normal derivative operators or none is represented by normal derivative operator's. We can complement  their result, for either scalar or non scalar holomorphic representations in Proposition~\ref{prop:equal},  Proposition~\ref{prop:sonn-1} and  Proposition~\ref{prop:nu1nonzero} Table B.

 Before to continue, we recall a result in \cite{OV2}: If some nonzero symmetry breaking operator from $H^2(G,\tau)$ is represented by a differential operator, then the restriction   of $H^2(G,\tau)$ to $H$ is discretely decomposable. Besides, if every symmetry breaking operator is represented via differential operators, then, the restriction is $H$-admissible.  We will consider  an answer to the following:  if   $H^2(H,\sigma)$ has multiplicity  larger than one  in $H^2(G,\tau)$, the symmetry breaking operators attached to this subrepresentation are all  normal derivatives differential operators or all are not or a mix. In the following paragraphs we outline  our partial answer to this question for the case of holomorphic representations.

We always have $ W \subset \mathcal L_{W,H}^c \cap \mathcal U(\h_0)W $. In Proposition~\ref{prop:allsbarenormal} we show a necessary and sufficient condition for the statement: every symmetry breaking operator is normal.   In \ref{sub:degree}  we define the order of a symmetry breaking operator, and we show in Proposition~\ref{prop:everynorfortriple} the fact: For a given nonscalar lowest $K$-type,   if every symmetry breaking operator of a fixed order bigger or equal to two is normal derivative, then the triple condition $[[\p_\h^+, \p_{\h_0}^-], \p_{\h_0}^-]=\{0\}$ holds. As a partial converse, we show that the validity of triple bracket equal to zero,  allows us to construct nonzero normal derivative operators. For this, we apply Engel's Theorem to a convenient nilpotent Lie subalgebra of $\k_\C$ and a non scalar representation $\tau$, see Proposition~\ref{prop:existencenormal2}.
For the treatment of first order operators we do not have  to consider  the triple bracket condition, in Proposition~\ref{prop:nonormal}, we show: whenever,    every first order symmetry breaking operator is normal, then, generally the lowest $K$-type is one dimensional with a few exceptions. In order to produce first order normal operators, once again, we appeal to Engel's Theorem, except  for a few very simple examples, we are not able to complete a picture on the first order normal operators, however, see Proposition~\ref{prop:first}. We present a conjecture on the structure of first order normal operators   in \ref{sub:w0irred} paragraph b).

    A consequence of   example \ref{exa:example1}, \ref{exa:example2} is that the following inclusions  might be proper,
  $ \{0\}\subseteq   \mathcal L_{W,H}^c \cap \mathcal V^{(1)} \cap \mathcal U(\h_0)W \subseteq  \mathcal L_{W,H}^c\cap \mathcal V^{(1)}  $. Therefore, we have examples where the subset of first order normal differential operators is a proper subset of the set of first order symmetry breaking operators.   However, in general, we have not been able to elucidate on the nature of the inclusions $\{0\} \subseteq Hom_H(V_\sigma^H, V_\tau^G)_{normal\,diff}\subseteq Hom_H(V_\sigma^H, V_\tau^G)_{diff}=Hom_H(V_\sigma^H, V_\tau^G)$. A result on the properness of the first inclusion is shown in Proposition~\ref{prop:nu1nonzero}.  We finally point to the recent authoritative paper on admissibility \cite{Os}, indeed we hope to extend our
results from the holomorphic to a more general setting of admissible representations.

\section{Background and notations}\label{sec:notation} This paper, essentially deals with holomorphic Discrete Series representations, the aim of this section is to set up the hypothesis and notation needed in order to state and show the results. To begin with $G$ (resp $H$) denotes a {\it simple},  connected, noncompact, matrix Lie group (resp. $H$ is a connected reductive group) so that $H$ is a closed subgroup of $G$. We fix maximal compact subgroups $K,L$ of $G$, $H$  and maximal torus $T\subset K$, $U\subset L$,  satisfying $K\cap H=L, T\cap H=U$. Henceforth, we assume both $G/K$, (resp.$ H/L$) carries a $G$-invariant (resp. $H$-invariant) holomorphic structure equivalent to respective bounded symmetric domains $\mathcal D_G$,   $\mathcal D_H$. Thus, $T$ as well $U$ are compact Cartan subgroups of respectively $G,H$. Furthermore, we assume
$$ The \, inclusion \, H/L \hookrightarrow G/K \, is \,a \,holomorphic\, map.$$
Our last   assumption is: $(G,H)$ is a {\it symmetric pair}, that is, there exists an involution $\sigma$ of $G$ so that the connected component, at the identity, of the fix points of $\sigma$ is $H$. The maximal compact subgroup $K$ gives rise to a Cartan involution $\theta$ so that $K$ is the set of fix points of $\theta$,  and $\theta$ commutes with $\sigma$. The Lie algebra of a Lie group is  denoted by the corresponding lower case German letter.  To avoid notation, the complexification of  the Lie algebra of a Lie group is  also denoted by the corresponding German letter without any subscript or sometimes with the subscript $\C$.  The same criteria will be applied to a vector space.     $V^\star $ denotes the dual space to a vector space $V.$ The Cartan decomposition associated to the Cartan involution  is denoted by $\mathfrak g=\mathfrak k +\mathfrak p_\g.$   We have similar decomposition for $\h=\mathfrak l +\p_\h=\k \cap \h + \p \cap \h$. The involution $\sigma$ yields the decomposition $\g=\h+\q$. Here, $\q=\{ x \in \g : \sigma(X)=-X\}$. Since $\sigma \theta$ is again an involution of $G$, we consider the subalgebra $\h_0 :=\{X \in \g: \sigma \theta(X)=X  \}= \l +\p_\g \cap \h_0=\l +\p_{\h_0}=\k \cap \h + \q \cap \p_\g$. $(\g, \h_0)$ is called the associated pair to $(\g, \h)$. The hypothesis $(G,K)$, as well as $(H,L)$, is a holomorphic pair, and the hypothesis $H/L \rightarrow G/K$ is a holomorphic map   yields:  the splitting  into $Ad(K)$-invariant subspaces $\p_\g= \p_\g^+ \oplus \p_\g^-$  (resp  into $Ad(L)$-invariant subspaces $\p_\h= \p_\h^+ \oplus \p_\h^-$, $\p_{\h_0}= \p_{\h_0}^+ \oplus \p_{\h_0}^-$) so that
$\p_\cdot^-$ $ =\overline{\p_\cdot^+}    $, (here $ x \rightarrow \bar{x}$ denotes de conjugation of $\g_\C$ with respect to the real form $\g$ and $\cdot$ represents the algebras $\g, \h, \h_0$). Besides, the Lie bracket $[\p_\cdot^+, \p_\cdot^+]=\{0\}$ and $\p_\g^+= \p_\h^+ \oplus \p_{\h_0}^+$. Our hypothesis of $G$ simple and holomorphicity yields the center of $\k$, $\z_\k$, is contained in $\l$, for a proof \cite[3.6, Table 3.6.1]{Kob9}.

Let $\Phi(\mathfrak g,\mathfrak t)=\Phi_\g $  denote the root system attached to the Cartan subalgebra $\mathfrak t.$ Hence, $\Phi(\mathfrak g,\mathfrak t) =\Phi_c(\mathfrak g, \mathfrak t) \cup \Phi_n (\mathfrak g, \mathfrak t) =\Phi_c \cup \Phi_n(\g)=  \Phi_c \cup \Phi_n$ splits up as the union the set of compact roots and the set of noncompact roots. We have similar root decomposition for $\h, \h_0$. Our hypothesis implies the existence of a system of positive roots $\Psi=\Psi_\g  \subset \Phi(\g,\t) $ so that $\p_\g^+$ is equal to the sum of root spaces corresponding to each root in $\Psi \cap \Phi_n(\g)$. Similarly, we have a system of positive roots $\Psi_\h$ (resp. $\Psi_{\h_0}$) in $\Phi(\h,\u)$ (resp. in $\Phi(\h_0, \u)$). When $U=T$, our hypotheses force $\Psi_\h =\Psi_\g \cap \Phi(\h,\t)$, and $\Psi_{\h_0} =\Psi_\g \cap \Phi(\h_0,\t)$. For $U\not= T$, latter on, we will write a description of $\Psi_\h, \Psi_{\h_0}$ from $\Psi_\g$.

We have the $Ad(L)$-invariant decomposition $\k_\C  =\l_\C \oplus  (\q_\C \cap \k_\C)$, whence, we may consider the set of weights $\Psi(\q_\C \cap \k_\C),\u) \subset \u_\C^\star $   of $\u$ in $(\q_\C \cap \k_\C)$, it readily follows that any of this weight  is the restriction of a root in $\Phi(\g,\t)$. Thus, $\Psi_\g$ determinate a subset $\Psi(\q_\C \cap \k_\C),\u) \subset \u_\C^\star $  of those weights that are restriction of roots in $\Psi_\g$, we denote by $(\q_\C \cap \k_\C)^\pm$ the sum of the weights spaces for weights in  $\pm \Psi(\q_\C \cap \k_\C),\u)$    and we have the decomposition $\q_\C \cap \k_\C= (\q_\C \cap \k_\C)^+\oplus (\q_\C \cap \k_\C)^-$. When $\u=\t$ each weight is just a root in $\Phi(\g, \t)$ and each $(\q_\C \cap \k_\C)^\pm$ is a sum of root spaces.  Even though $(K,L)$ is a symmetric pair,   it   follows from the classification of the pairs $(G,H)$ \cite{Kob9}\cite{KOadv} that $(\q_\C \cap \k_\C)^\pm$  is  not always $Ad(L)$-invariant, whence,  $(K,L)$ is not always  a hermitian pair. Our hypotheses and, in particular,  $\g$ is a simple Lie algebra, imply the center of $\k$, denoted by $\z_\k$ is contained in $\l$ \cite[Section 3.6]{Kob9}. Hence,  $(\q_\C \cap \k_\C)$ is a subset of the semisimple factor $\k_{ss}$ of $\k$, whence,  we conclude that when $\g$ is a simple Lie algebra we have both $(\q_\C \cap \k_\C)^\pm$ is an abelian  Lie algebra consisting of nilpotent elements of $\k_{ss}$. This fact,   let as apply Engel's theorem.

\subsubsection{Holomorphic Discrete Series}\label{sub:holds} A realization for the holomorphic Discrete Series representation of a group $G$ is in space of holomorphic functions from $\mathcal D_G$ into its lowest $K$-type $(\tau, W)$. To be more precise, in the following,    we recall the standard description
of a symmetric complex domain realization of the holomorphic Discrete Series.   (reference for details are \cite{FOO}, \cite{JV}, \cite[XII.5]{Neeb})

 $G$ a semisimple Lie group and a maximal compact subgroup $K$ of $G$.

 $\bigl(G,H=(G^\sigma)_0\bigr)$ is a symmetric pair.

 $H_0:=(G^{\sigma \theta})_0$ associated subgroup to $H$.

 $L=K\cap H$  a maximal compact subgroup of $H$.

 $H/L \rightarrow G/K$   a holomorphic embedding.

 $(L_\cdot^\tau ,V_\tau)=(L_\cdot^\tau ,V_\tau^G)$ is a holomorphic Discrete Series for $G$ of lowest $K$-type $(\tau,W)$, realized as holomorphic $W$-valued functions in the bounded symmetric domain $\mathcal D\cong G/K$.      The action, $L_\cdot^\tau $ of $\mathfrak g$,  is recalled bellow.

 $(L_\cdot^\sigma ,V_\sigma)=(L_\cdot^\sigma ,V_\sigma^H )$ is a holomorphic Discrete Series for $H$ of lowest $L$-type $(\sigma, Z)$, realized as holomorphic $Z$-valued functions in the bounded symmetric domain $\mathcal D_\mathfrak h \cong H/L$. We sometimes refer to these simply as holomorphic representations.

Next, we recall the definition of {\it holomorphic embedding}. First of all, the symmetric space $G/K$ ($G$ simple) admits a $G$-invariant holomorphic structure if and only if the center of $K$ is one dimensional torus.
Similarly, for $H/L$, $H$ simple. The natural inclusion $H/L\rightarrow G/K$ is a {\it holomorphic embedding} if both symmetric spaces admits an ($H$-invariant) $G$-invariant complex structure and
the natural inclusion is holomorphic. The inclusion is holomorphic is equivalent to: the holomorphic tangent bundle of $H/L$ is a subset of the holomorphic tangent bundle for $G/K$. The   inclusion $H/L\rightarrow G/K$ being holomorphic, and $\g$ a simple Lie algebra, implies the center of $K$ is contained in the center of $L$. For a proof \cite[Table 2.1.3]{Kob9} and references therein.

We recall   $ T$ is a maximal torus of $K$ so that $U=T\cap H$ is a maximal torus of $L$.  $o :=$the coset $eK$. The above defined holomorphic system  $\Psi$   in $\Phi(\mathfrak g,\mathfrak t)$, $\Psi_\h, \Psi_{\h_0}$ holomorphic systems in $\Phi(\mathfrak h,\mathfrak u), \Phi(\mathfrak h_0,\mathfrak u)$   so that
 $\sum_{\beta \in \Psi^n} \mathfrak g_\beta:=\mathfrak p^+=\mathfrak p_\mathfrak g^+= \mathfrak p_{\mathfrak h}^+ + \mathfrak p_{\mathfrak h_0}^+$ are isomorphic to the respective   holomorphic tangent spaces at $o$. Since $\k_\C +\p^+$ is a parabolic subgroup of the complexification for $G$, we have the decomposition $G\subset P_+K_\mathbb C P_{-}$, and we have the Harish-Chandra realization as bounded symmetric domain  $G/K \cong \mathcal D\subset \mathfrak p^+, H/L \cong \mathcal D_\mathfrak h \subset \mathfrak p_\mathfrak h^+$, $ H_0/L \cong \mathcal D_{\mathfrak h_0} \subset \mathfrak p_{\mathfrak h_0}^+$. For this, we recall the triangular  decomposition

 \phantom{xxxx} $P_+K_\mathbb C P_{-} \ni x=exp (x_+) x_0 exp (x_-), x_\pm \in \mathfrak p^\pm, x_0 \in K_\mathbb C$.\\ Then, the Harish-Chandra embeddings are:  $G/K \ni xK \mapsto  x_+\in \mathcal D\subset \mathfrak p^+$,   $H/L \ni hL   \mapsto  h_+\in \mathcal D_\mathfrak h \subset \mathcal D \subset  \mathfrak p^+$, $H_0/L \ni hL   \mapsto  h_+\in \mathcal D_{\mathfrak h_0} \subset \mathcal D $.

   The cocycle $c_\tau$ and the reproducing kernel $K_\tau^c$ that defines the Hilbert structure and the representation on  $V_\tau$ are:  for $g\in G, z,w \in \mathcal D$,
 \begin{equation} \label{eq:ctauktauhol} c_\tau(g,z)=\tau ((g exp z)_0),\,\, K_\tau^c (w,z)=\tau  (((exp  \bar w)^{-1}exp z)_0)^{-1}. \end{equation}
Here, we take the conjugation with respect to the real form $\mathfrak g$.

 We consider the space $L_\tau^2 (\mathcal D,W):=L_{c_\tau}^2(\mathcal D,W)$ and the subspace
 \begin{multline} \label{eq:vtau} V_\tau :=V_\tau^G= \{ f\in  \mathcal O(\mathcal D,W) : \\ \int_\mathcal D (K_\tau^c (w,w) f(w), f(w))_W dm_{G/K} (w) <\infty \}.  \end{multline}
 From now on, we assume the space $V_\tau^G$ is nonzero.

The action $L_g^\tau :=\pi_{c_\tau}(g), g\in G$ in $\mathcal O(\mathcal D,W)$ is defined    by the formula

\phantom{xxxx}$L_g^\tau(f)(z)=\tau ( (g^{-1}exp z)_0)^{-1} f( (g^{-1}exp z)_+), g\in G, z\in \mathfrak p^+$.

 Then, \cite{Neeb}, it can be shown that $(L_\cdot^\tau, V_\tau)$ is a unitary, irreducible representation for $G$. Due that any matrix coefficient of $(L_\cdot^\tau, V_\tau)$ is square integrable for the Haar measure on $G$, henceforth, we refer to:

\smallskip
 \phantom{xxxxxxx}$(L_\cdot^\tau, V_\tau)$ is {\it a holomorphic Discrete Series} for $G$.

\smallskip
 We recall    the subspace $(V_\tau)_{K-fin}$ of $K$-finite vectors in $V_\tau$ is \begin{equation}\label{eq:1} (V_\tau)_{K-fin}=V_{K-fin}=\mathcal P(\mathfrak p^+,W)  \cong \mathcal U(\mathfrak g)\otimes_{\mathcal U(\mathfrak k +\mathfrak p^+)} W \cong S(\mathfrak p^-)\otimes W \end{equation}

 Here, we have identified $W$ with the subspace of constant functions.

 The isomorphisms in \ref{eq:1} are $L_D^\tau (w)\leftarrow [D\otimes w] \leftarrow D\otimes w$.
The first isomorphism is a $\g$-map, whereas the second is a $K$-isomorphism.
 Here, we use that $\mathfrak p^\pm$ is an abelian Lie algebra, and hence, the symmetrization map may be thought as the identity map. \ref{eq:1} reflects that the Harish-Chandra module of a holomorphic Discrete Series representation is equivalent to a Verma module determinate by the parabolic subalgebra $\k_\C +\p_\g^+$, the inducing representation is the lowest $K$-type representation extended by zero to the nilpotent $\p_\g^+$.

  \subsubsection{}\label{sub:jv} In  \cite{JV} \cite{FOO}  they have computed the action of $\mathfrak g$ (resp.$K$) in $\mathcal O(\mathcal D,W)$. For this we define $(\delta(x)f)(v)=\lim_{t\rightarrow 0} \frac{ f(v+tx)-f(v)}{t}$, then, for $p\in \mathcal O(\mathcal D,W)$,

 and for $x\in \mathfrak k, L_x^\tau(p)(v)=\tau (x)(p(v))-(\delta([x,v]) p)(v)$,

 and for $x \in \mathfrak p^+, L_x^\tau(p)(v)=-(\delta(x) p)(v)$,

 and for $x \in \mathfrak p^-, L_x^\tau(p)(v)=\tau([x,v])(p(v))-\frac12(\delta([[x,v],v]) p)(v)$,

 and for $k \in K, L_k^\tau (p)(v)=\tau(k)(p(k^{-1}v))$.

We set $V_\tau^{\mathfrak p_\g^+}=\{p \in V_{K-fin}: L_x^\tau(p)=\delta(x)p=0, \forall x \in \mathfrak p_\mathfrak g^+ \}$. Similarly, for $\p_\h^+$  we define $V_\sigma^{\p_\h^+}$, $V_\tau^{\mathfrak p_{\h}^+ }$.
 It readily follows that \begin{equation*}V_\tau^{\mathfrak p^+}=W , \,\,\, V_\tau^{\mathfrak p_\mathfrak h^+}=\{p \in V_{K-fin}: \delta(x)p=0, \forall x \in \mathfrak p_\mathfrak h^+ \}=\mathcal P(\mathfrak p_{\mathfrak h_0}^+,W).\end{equation*}

 \smallskip
 The hypothesis, the inclusion $H/L \rightarrow G/K$ is a holomorphic embedding and $V_\tau$ is a holomorphic Discrete Series  representation  yields that the restriction of $(L_\cdot^\tau , V_\tau)$ to $H$ is an $H$-admissible representation and the  totality of irreducible factors are holomorphic Discrete Series for $H$ for a reference \cite{Kob9}. Therefore, after we write $$res_H(V_\tau)=\oplus_{\mu \in Spec(res_H(V_\tau))} V_\tau[V_\mu^H],$$  the subspace $\mathcal L_{W,H}^c$  defined in \cite{Va}\cite{OV3},  equal to the sum   of the lowest $L$-type's of the totality of irreducible $H$-factors of $res_H(V_\tau)$, is equal to $\mathcal P(\mathfrak p_{\mathfrak h_0}^+,W)$. That is,

 \begin{equation}\label{eq:Lwh} \begin{split} \mathcal L_{W,H}^c   \ &  =\oplus_{\mu \in Spec(res_H(V_\tau))} V_\tau[V_\mu^H][V_{\mu^H +\rho_n^H}^L] \\ & =\{p \in V_{K-fin}: L_x^\tau(p)=\delta(x)p=0, \forall x \in \mathfrak p_{\mathfrak h}^+ \}=\mathcal P(\mathfrak p_{\mathfrak h_0}^+,W). \end{split} \end{equation}

 Via the Killing form, $b$, $\mathfrak p_{\mathfrak h_0}^-$ is in duality with $\mathfrak p_{\mathfrak h_0}^+$, which provides an isomorphism between $\mathcal P(\mathfrak p_{\mathfrak h_0}^+,W)$ and $S(\mathfrak p_{\mathfrak h_0}^-)\otimes W$. That is, the inverse  map to

 \phantom{xx} $\mathfrak p_{\mathfrak h_0}^- \otimes W  \ni Y\otimes w \mapsto (\mathfrak p_{\mathfrak h_0}^+ \ni X \stackrel{p_Y}{\longrightarrow} b(X,Y)w)\in \mathcal P(\mathfrak p_{\mathfrak h_0}^+,W),$

  extends to a   $L$-equivariant isomorphism

 \begin{equation*} D_0 :    \mathcal P(\mathfrak p_{\mathfrak h_0}^+,W) \rightarrow S(\mathfrak p_{\mathfrak h_0}^-)\otimes W .\end{equation*} The action of $L$ in $S(\mathfrak p_{\mathfrak h_0}^-)\otimes W$ is   tensor product   action.

 We recall $W$ is identified with the constant functions, the equality $\mathfrak h_0=\mathfrak p_{\mathfrak h_0}^- +\mathfrak l + \mathfrak p_{\mathfrak h_0}^+$ and that $V_\tau$ is a holomorphic Discrete Series  representation. Then, as in \cite{Va}\cite{OV3} we consider the subspace  \begin{equation*} \begin{split} \mathcal U(\mathfrak h_0)W  & :=\{L_D^\tau  w, D\in \mathcal U(\mathfrak h_0), w\in W\} \\ & =\{L_D^\tau  w, D\in \mathcal U(\mathfrak p_{\mathfrak h_0}^-), w\in W\}\end{split} \end{equation*} and  the map $D_1$ defined by
 $$\mathcal U(\mathfrak p_{\mathfrak h_0}^-)\otimes W \ni D\otimes w \stackrel{D_1}{\longmapsto} L_D^\tau  w \in \mathcal U(\mathfrak h_0)W.$$

 \subsubsection{}\label{sub:D1iso}Then,   $D_1$ is a $L$-equivariant isomorphism, as it follows from \ref{eq:1}.

 Finally, we recall $\mathfrak p_{\mathfrak h_0}^-$ is an abelian Lie algebra, hence, the symmetrization map from $S(\mathfrak p_{\mathfrak h_0}^-)$ onto $\mathcal U(\mathfrak p_{\mathfrak h_0}^-)$ may be thought of as an identification. Thus, composition isomorphisms,  we have given a   proof of the $L$-isomorphism \cite{Va} \cite[Proposition 1]{OV3} \cite{KP2} in  the holomorphic setting. \begin{equation} \mathcal L_{W,H}^c \stackrel{D}\cong \mathcal U(\h_0)W \end{equation}

In the sequel, the resulting $\g$-isomorphism (resp.   $K$-isomorphism,  $\h_0$-isomorphism, $L$-isomorphism) will be useful.
\begin{equation}  V_{K-fin}\ni L_D^\tau (w) \leftarrow   [D\otimes w]\in  \mathcal U(\g) \otimes_{\mathcal U(\k_\C +\p_\g^+)} W   \end{equation}
\begin{equation}     \mathcal U(\g) \otimes_{\mathcal U(\k_\C +\p_\g^+)} W \ni [D\otimes w]\leftarrow  D\otimes w \in S(\p_{\g}^-)\otimes W   \end{equation}
\begin{equation}  \mathcal U(\h_0)W \ni L_D^\tau (w) \leftarrow [D\otimes w]\in  \mathcal U(\h_0) \otimes_{\mathcal U(\k_\C +\p_\g^+)} W   \end{equation}
\begin{equation}      \mathcal U(\h_0) \otimes_{\mathcal U(\k_\C +\p_\g^+)} W \ni [D\otimes w] \leftarrow D\otimes w \in  S(\p_{\h_0^-})\otimes W   \end{equation}

 \section{Symmetry breaking operators represented via differential operators} \label{do}

In this section we analyze in detail the structure of symmetry breaking operators between admissible holomorphic Discrete Series representation, in particular, we will work out several facts on the nature of the involved differential operators when we describe    symmetry breaking operators by means of differential operators. This section is a continuation of ground work of Kobayashi-Pevzner \cite{KP1}\cite{KP2}, Nakahama \cite{Na} as well as latter work of   \cite{OV2}\cite{OV3}\cite{OV4} and references therein.

T. Kobayashi has coined every $H$-intertwining, continuous, linear map from a representation of $G$ into any of a discrete $H$-factor   by the term {\it Symmetry Breaking operator}. To follow, we recall definitions and facts required in order to describe a connection between symmetry breaking operators and differential operators.

 For $X\in \mathcal U(\mathfrak g)$, let $R_X$ denote  infinitesimal right derivative by $X$.  Let $(G,H)$ be an {\it arbitrary reductive pair}, $ K, L:=K\cap H$ respective maximal compact subgroups,  $(\tau, W)$ is (resp.    $(\sigma,Z)$ ) irreducible representations of $K,L $. We assume there exists  Discrete Series representations for $G, H$ and we realize them as respective spaces  $H^2(G,\tau), H^2(H,\sigma)$ of smooth sections of respective bundles $G\times_\tau W \rightarrow G/K$, $H\times_\sigma Z \rightarrow H/L$. Then,
 a definition of   differential operator is based on the fact (for a proof, \cite[chap V]{Wal}): A linear map $ D :\Gamma^\infty(G\times_\tau W) \rightarrow  \Gamma^\infty(G\times_\sigma Z)$ is a {\it differential operator} if and only if there exists finitely many smooth functions $c_\alpha :G\rightarrow Hom_\mathbb C(W,Z)$ so that $D(f)=\sum c_\alpha R_{X_1^{\alpha_1}\cdots X_n^{\alpha_n}}(f)$ or a similar expression by means of left derivatives. Such a $D$ is invariant by left translations by $G$ if and only if, for every $\alpha$,  $c_\alpha$ is a constant function and $\sum_\alpha c_a X_1^{\alpha_1}\cdots X_n^{\alpha_n} \in (Hom_\mathbb C(W,Z)\otimes \mathcal U(\mathfrak g))^L$.

 As in \cite{KP1}\cite[4.0.1]{OV3}    $S : H^2(G,\tau) \rightarrow H^2(H,\sigma)$ is {\it represented by a  differential operator}  if there exists a $G$-invariant differential operator   $D:   \Gamma^\infty(G\times_\tau W) \rightarrow     \Gamma^\infty(G\times_\sigma Z)$ so that $ \forall \,  f \in H^2(G,\tau)$ we have $  S(f)= res (D(f))$. Here, $res$ is the restriction map $ res: \Gamma^\infty(G\times_\sigma Z) \rightarrow \Gamma^\infty(H\times_\sigma Z)$. According to Theorem \cite[Theorem 2.5]{OV4}, in order to show that $S$ is the restriction of a differential operator,  it suffices to show $S$ is the restriction of a global differential operator.

  The Hilbert spaces $H^2(G,\tau), H^2(H,\sigma)$ have the property that its elements are smooth functions, thus, point evaluation is defined in each space, since each space is the $L^2$-kernel of an elliptic differential property, we have that point evaluation is a continuous linear map. This fact, allows to define the kernel of a symmetry breaking operator as follows: We fix $z \in Z, h \in H$, then the linear map $H^2(G,\tau)\ni f \mapsto (Sf(h),z)_Z \in \C$ is continuous, so there exists $K_S(\cdot, h)^\star \in H^2(G,\tau)$ so that $(Sf(h),z)_Z=\int_G (f(x), K_S(x, h)^\star (z))_W dx \, \forall f$. Thus, $S(f)(h)=\int_G K_S(x,h) f(x) dx$. For more details see \cite{OV2}\cite{OV4}\cite{Kob9}.

   We also recall that in \cite[Lemma 4.2]{OV2} it is shown that $S$ is the restriction of a differential operator if and only if $K_S(\cdot,e)^\star z$ is a $K$-finite vector for each $z\in Z$. Actually, we always have $K_S(\cdot,e)^\star z \in V_\tau^G[V_\sigma^H][Z]$, the isotypic component of $(\sigma,Z)$ in the isotypic component $V_\tau^G[V_\sigma^H]$ \cite{OV4}.

  \subsection{Symmetry breaking operators and  normal derivatives} For this subsection $(G,H)$ is a symmetric pair and $V_\tau^G$ is a   square integrable representation. Our aim is to generalize a result  in   \cite[Theorem 5.1]{Na}. In \cite{KP2} are considered symmetry breaking operators expressed by means of normal derivatives, they obtain results for holomorphic embedding of a rank one symmetric pairs.   As before, $H_0=G^{\sigma \theta}$ is the associated subgroup. We recall $\mathfrak h\cap \mathfrak p$ is orthogonal to $\mathfrak h_0\cap \mathfrak p$ and that $\mathfrak h\cap \mathfrak p\equiv T_{eL}(H/L)$, $\mathfrak h_0\cap \mathfrak p\equiv T_{eL}(H_0/L)$. Hence, for $X \in \mathfrak h_0\cap \mathfrak p$, more generally for $X \in \mathcal U(\mathfrak h_0)$, we say $R_X$ is a normal derivative  to $H/L$ differential operator. For short, {\it normal derivative}.  This allows us to write: a symmetry breaking operator $S$ is represented by a {\it normal derivative differential operator}, if there exists $\sum_\alpha c_\alpha X^{\alpha_1}\cdots X^{\alpha_s}\in (Hom_\C(W,Z)\otimes \mathcal U(\h_0))^L$ so the associated differential operator agrees with $S $ on the subspace $V_\tau^G$. Other ingredient necessary for the next Proposition are the subspaces $\mathcal L_{W,H}^c$ (\ref{eq:Lwh}) and $\mathcal U(\mathfrak h_0)W$. The latter subspace is contained in the subspace of $K$-finite vectors. Since, owing to our hypothesis, $res_H(V_\tau^G)$ is $H$-admissible $\mathcal L_{W,H}^c$ is contained in the subspace of $K$-finite vectors \cite[Prop 1.6]{Kobdd3}. However, it might not be equal to $\mathcal U(\mathfrak h_0)W$ as we will verify. The next Proposition and its converse, dealt with consequences of the equality $\mathcal L_{W,H}^c = \mathcal U(\mathfrak h_0)W$. We would like to recall that due to the definition of $\mathcal L_{W,H}^c $, we always have the equality of isotypic components $\mathcal L_{W,H}^c [Z]=  V_\tau^G[V_\sigma^H][Z]$.
\begin{prop}\cite[Proposition 5]{OV3} We assume $(G,H)$ is a symmetric pair.  We also assume there exists a irreducible representation $(\sigma, Z)$ of $L$ so that $V_\sigma^H$ is a irreducible factor of $res_H(V_\tau^G)$ and  $ V_\tau^G[V_\sigma^H][Z]=\mathcal L_{W,H}^c[Z]= \mathcal U(\mathfrak h_0)(W)[Z]=L_{\mathcal U(\mathfrak h_0) }^\tau( W)[Z]$. Then,   any symmetry breaking operator from $V_\tau^G$ into $V_\sigma^H $ is represented by a normal derivative  differential operator. Conversely.  If every element
 in $Hom_H(V_\tau^G, V_\sigma^H) $ has a expression as differential operator by means of "normal derivatives",  then, the equality  $\mathcal L_{W,H}^c[Z]= V_\tau^G[V_\sigma^H][Z]=\mathcal U(\mathfrak h_0)W[Z]$ holds.
\end{prop}
\subsubsection{Examples}\label{sub:zs} We present four examples, for two of the examples we obtain quite precise results, whereas the second and fourth example shows the difficulty for computing the subspace of first order normal derivative operators inside the space of first order symmetry breaking operators. We begin with a few general considerations and then move to develop the examples. We fix  a symmetry breaking operator $S:V_\tau^G \mapsto V_\sigma^H$, represented by the kernel $K_S :G\times H \rightarrow Hom_\mathbb C( W, Z)$ and we recall  the subspace $Z_S:=Image(Z\ni z \mapsto K_{S^*}(e, \cdot) (z)=K_S(\cdot, e)^\star z \in V_\tau^G [V_\sigma^H][Z]))$ is a $L$-irreducible subspace contained in $\mathcal L_{W,H}^c$ \cite[Proposition 2.1]{OV3}. Thus, either $Z_S \cap \mathcal U(\mathfrak h_0)W=\{0\}$ or $Z_S \subset \mathcal U(\mathfrak h_0)W$. Our hypothesis is that $(L_\cdot^\tau, V_\tau^G)$ is an $H$-admissible representation, hence,  every symmetry breaking operator is represented by a differential operator  \cite{KP1}, \ref{eq:sasdfo},   \cite[Proposition 4.4]{OV2}. We recall in \cite[Proposition 6.1]{OV3}, it is shown that in the case  $Z_S \subset \mathcal U(\mathfrak h_0)W$,  $S$ is represented by a normal derivative differential operator, see \ref{sub:tech}, whereas in the   case $Z_S \cap \mathcal U(\mathfrak h_0)W=\{0\}$,  $S$ is represented by a differential operator   that never will be a normal derivative differential operator.

 The first order symmetry breaking operators $S$ are those $S$ such that $Z_S \subset \mathcal V^{(1)}$ (see \ref{sub:vn}, \ref{sub:disjoint}). More precisely, are those $S \in Hom_H(V_\tau^G, V_\sigma^H)$ so that

 $Z_S=Im(Z\ni z \mapsto  K_S(\cdot, e)^\star z \in V_\tau^G[V_\sigma^H][Z])
\subseteq \mathcal L_{W,H}^c \cap \mathcal V^{(1)}$

 \phantom{xxxxxxxxxxxx} $=\{ Y:=\sum_{\gamma \in \Psi_n(\g)} X_{-\gamma} \otimes w_\gamma : L_X^\tau(Y)=0 \, \forall X \, \in \p_\h^+ \}  $.

  Whereas, the first order normal derivative symmetry breaking operators are those $S$ so that

 $Z_S=Im(Z\ni z \mapsto  K_S(\cdot, e)^\star z \in V_\tau^G[V_\sigma^H][Z]) \subseteq \mathcal L_{W,H}^c \cap \mathcal U(\h_0)W \cap  \mathcal V^{(1)}$

\phantom{xxxxxxxxxxxx} $=\{ Y:=\sum_{\gamma \in \Psi_n(\h_0)} X_{-\gamma} \otimes w_\gamma : L_X^\tau(Y)=0 \, \forall \, X \in \p_\h^+ \}  $.

  Here, $X_\gamma$ is a root vector for the root $\gamma$ and $w_\gamma \in W$.
For the actual computation for the first two pairs  (see    \ref{exa:example1},\ref{exa:example2}), for the third and fourth pair see \ref{eq:so2n2n-1}

 The pairs we consider are:\\ a)($\g=\mathfrak{su}(n,1),   \h =\mathfrak s(\mathfrak u( 1 )\oplus \mathfrak{u}(n-1,1))$. In this case for $\tau$ a scalar representation we have the equality $\mathcal L_{W,H}^c \cap \mathcal U(\h_0)W\cap \mathcal V^{(1)}=\mathcal L_{W,H}^c \cap \mathcal V^{(1)}$, For $\tau$ a nonscalar representation we  obtain $\mathcal L_{W,H}^c \cap \mathcal U(\h_0)W\cap \mathcal V^{(1)}$ is an irreducible representation for $L$ and a proper subspace of $\mathcal L_{W,H}^c \cap \mathcal V^{(1)}$. We compute the dimension of the operators in $\dim Hom_H(V_\tau, V_\sigma)$ minus the dimension of the subspace of normal derivative operators is equal to $\dim V_{\lambda +\rho_n}^K -\dim V_{w_L(W_K(\lambda +\rho_n-\rho_c))+\rho_L}^L$.

b)$(\g=\mathfrak{su}(n,1),   \h =\mathfrak s(\mathfrak u( n-1)\oplus \mathfrak{u}( 1,1))$. Once again, for $\tau$ a scalar representation we have every first order symmetry breaking operator is normal derivative. For a nonscalar $\tau$ we have the subspace of first order normal symmetry operators is a nontrivial proper subspace of $Hom_H(V_\tau, V_\sigma) $. We only have a conjecture how large is the subspace of normal derivative operators. see \ref{sub:w0irred},\ref{exa:example2}

c)$(\g=\mathfrak{so}(2,2n+1),   \h =\mathfrak{so}(2,2n)+\mathfrak o( 1 ))$.
For $\tau$ non scalar representation we have there is no first order normal derivative operator, this follows from $\mathcal V^{(1)}\cap \mathcal L_{W, \mathfrak{so}(2,2n)+\mathfrak o( 1 )}^c \cap \mathcal U(\mathfrak{so}(2,1)+\mathfrak {so}(2n))W=\{0\}$. \ref{eq:so2n21},\ref{prop:sonn-1}.

d)$(\g=\mathfrak{so}(2,2n+1),   \h =\mathfrak{so}(2,1)+\mathfrak {so}(2n))$. This is the associate pair to the pair in c). In this case there exists nonzero first order normal derivative operator, owing to  $\mathcal V^{(1)}\cap \mathcal L_{W, \mathfrak{so}(2,1)+\mathfrak {so}(2n)  }^c \cap \mathcal U(\mathfrak{so}(2,2n)+\mathfrak o( 1 ) )W\not=\{0\}$. See Table A,  \ref{prop:sonn-1}.

  \subsubsection{Technique to represent $S$ as a differential operator}\label{sub:tech} A technique to find a differential operator that represents $S$ knowing its kernel $K_S$ has been developed in \cite{OV3}\cite{OV4} is as follows: we fix a orthonormal basis $\{z_p \}_{1\leq p \leq\dim  Z}$ for $Z$. The hypothesis $res_H(V)$ is admissible gives  $V_\tau^G[V_\sigma^H][Z] \subset L_ {\mathcal U(\mathfrak g)}(V_\tau^G[W])[Z])$ \cite[Prop. 1.6]{Kobdd3}, since  $K_S(\cdot,e)^\star z \in V_\tau^G[V_\sigma^H][Z], \forall z \in Z$, we obtain for     each  $p$, there exists   $D_{p,i}\in \mathcal U(\mathfrak g)$) and $w_{p,i} \in W$ so that $K_S(\cdot,e)^\star z_p=\sum_i L_{D_{p,i}}^\tau K_\tau(\cdot,e)^\star w_{p,i}$. Here, the sum over $i$ is a finite sum, and, depends on $p$. Then,  for any $f \in V_\tau^G$    \begin{equation}\label{eq:sasdfo} S(f)(h)= \sum_{1\leq p \leq  \dim  Z} \sum_i ( R_{\check D_{p,i}^\star}f(h),w_{p,i})_W \,\, z_p. \end{equation} Since  $D_{p,i} \in \mathcal U(\mathfrak g)$)  such a expression of $S $ is a representation of $S$ in terms of  in terms of differential operators.

In our case, $V_\tau^G$ is a $H$-admissible holomorphic Discrete Series representation, whence $(V_\tau^G)_{K-fin}= \mathcal U(\p_\g^-)W$, and, then, we may choose  each $D_{p,i} \in \mathcal U(\p_\g^-) $.  Since, $\check D_{p,i}^\star \in \mathcal U(\p_\g^+)$  and in \cite[IX,\S 5]{L} we find a proof of: right infinitesimal derivative in $\mathcal D_G \equiv G/K$ with respect to a root vector $E_\alpha$ in $\p_\g^+$ is essentially the partial derivative $\partial_{z_\alpha}$, we have outlined $S$ is represented by a holomorphic differential operator, a result of \cite{KP1}. In case $K_S(\cdot,e)^\star z \in V_\tau^G[V_\sigma^H][Z]  \in      L_ {\mathcal U(\mathfrak h_0)}^\tau(V_\tau^G[W])[Z]$, we further obtain each $D_{p,i} \in \mathcal U(\mathfrak p_{\h_0}^-)$, in consequence,  an expression of $S $   in terms of normal derivatives. Thus,  the equality $\mathcal L_{W,H}^c = \mathcal U(\mathfrak h_0)W$ implies: {\it every symmetry breaking operator into an arbitrary irreducible  $H$-factor of the restriction   of the holomorphic Discrete Series $V_\tau^G$, to $H$, is represented by a normal differential operator}.

To follow, we resume a consequence of the statements in Section~\ref{sec:algebraic}. Subsequently to  the statement of each Proposition we explicit the meaning of such a statement  and the relation with Section~\ref{sec:algebraic}.

\begin{prop}\label{prop:allsbarenormal} We keep the assumption $\g$ is a simple Lie algebra and $H/L\rightarrow G/K$ is a holomorphic immersion, then

a)When $ \k_{ss}$ is a simple Lie algebra,   $\mathcal L_{W,H}^c = \mathcal U(\mathfrak h_0)W$ if and only if $[[\p_\h^+,\p_{\h_0}^-],\p_{\h_0}^-]=\{0\}$ and $(\tau, W)$ is a one dimensional representation.

b) When $ \k_{ss}$ is not a simple Lie algebra,   $\mathcal L_{W,H}^c = \mathcal U(\mathfrak h_0)W$ if and only if $[[\p_\h^+,\p_{\h_0}^-],\p_{\h_0}^-]=\{0\}$ and $(\tau, W)$ is so that one of the simple factors of  $ \k_{ss}$ is contained in $Ker(\tau)$.
\end{prop}
  In \cite[4.6]{Va}, see \ref{sub:triplepairszero},  we find the list of symmetric   pairs $(\g,\h)$ that satisfies the condition $[[ \p_{\h_0}^-,\p_\h^+ ],\p_{\h_0}^-]=\overline{[[ \p_{\h_0}^+, \p_\h^- ],
\p_{\h_0}^+]}=\{0\}$, they  are:

  a) Pairs with $ \k_{ss}$ is   a simple Lie algebra: $ (\mathfrak{su}(1,n), \mathfrak{su} (1,l) +\mathfrak{su} ( n-l)+\mathfrak u(1)) $; $(\mathfrak{so}(2m,2), \mathfrak u(m,1)) $;   $(\mathfrak{so}^\star (2n), \mathfrak u(1,n-1)) $; $(\mathfrak{so}^\star(2n),  \mathfrak{so}(2) +\mathfrak{so}^\star(2n-2)) $;
$(\mathfrak e_{6(-14)}, \mathfrak{so}(2,8)+\mathfrak{so}(2)).$

b) $ \k_{ss}$ is not a simple Lie algebra: $ m>1,n>1, (\mathfrak{su}(m,n), \mathfrak{su} (m,l) +\mathfrak{su} ( n-l)+\mathfrak u(1)).$

Therefore for the pairs listed in a) any symmetry breaking operator with domain an scalar holomorphic Discrete Series representation is represented by a normal derivative differential operator, whereas, for the pair listed in b) the same statement holds, and, besides  for representations   $(\tau, W)$ so that $Ker(\tau)$ contains at least one irreducible factor of $\k_{ss}\cong \mathfrak{su}(m)+\mathfrak{su}(m)$.

Among the pairs in a) b), the rank one pairs are:  $(\mathfrak{so}(2m,2), \mathfrak u(m,1)) $;  $(\mathfrak{so}^\star(2n),  \mathfrak{so}(2) +\mathfrak{so}^\star(2n-2)) $; $   (\mathfrak{su}(m,n), \mathfrak{su} (m,n-1) +\mathfrak{su} ( n-1)+\mathfrak u(1)).$ Thus, for the rank one pairs, the statement in Proposition~ \ref{prop:allsbarenormal} agrees with \cite[Theorem 5.3]{KP2}.

Up to now, we have analyzed a extreme case, that is, every symmetry breaking operator is represented via normal derivatives. To follow, we present some results of extreme cases and non extremal cases.  We begin with an example: (note that in this  example $\g$ is not a simple Lie algebra).  We have for
  $G=SL_2(\R)\times SL_2(\R), H=SL_2(\R)$, that, $\mathcal L_{W,H}^c \cap \mathcal U(\h_0)W =W+\mathcal V^{(1)}\cap \mathcal L_{W,H}^c \cap \mathcal U(\h_0)W $.
Thus, every symmetry breaking operator represented by a differential operator of order
  zero or one is represented by a normal derivative differential operator, meanwhile, a symmetry breaking operator represented by a differential operator of order bigger than one is not represented  by a normal derivative differential operator.

  On order to continue, we point out that  the representation of a symmetry breaking operator by a differential operator is not unique because if $D=\sum_\alpha c_\alpha R_{X_1^{\alpha_1}\cdots X_{2q}^{\alpha_{2q}}})\in (Hom_\C (W,Z)\otimes \mathcal U(\g))^L$ represents $S$, also does $D + (\Omega_\g -c)^k$ for $c$ equal to the value of the central infinitesimal character of $V_\tau^G$ in $\Omega_\g$, and any $k\geq 1$, this is due  to any element of $V_\tau^G$ is a smooth function either en $G$ or in $\mathcal D_G$. However, in \cite[Proposition 2.12]{OV4} we have shown,
 \begin{prop} \label{prop:uniqness} We assume $H/L \rightarrow G/K$ is a holomorphic embedding and both $V_\tau^G, V_\sigma^H$ are holomorphic Discrete Series representations. Let $S : V_\tau^G \rightarrow V_\sigma^H$ be a symmetry breaking operator. Then, there exists an unique linear function $Z\ni z\mapsto D_z:=\sum_i D_z^i \otimes w_i \in \mathcal U(\mathfrak p_\g^-)\otimes W$ so that $\forall z \in Z, \,K_S^c(\cdot,e)^\star z = \sum_i L_{D_z^i}^\tau (K_\tau^c(\cdot, e)^\star w_i)$.
 \end{prop}

Here, $w_i, i=1\cdots , \dim W$ is an ordered linear basis for $W$.
\subsubsection{Degree of a symmetry breaking operator}\label{sub:degree} We now proceed as in \ref{sub:tech} and express  $S$ as a differential operator. Therefore, we may and we  define the {\it degree of S} as  the minimum $k$ so that $D_z \in \sum_{0\leq j \leq k} \mathcal V^{(j)}$. For the definition of $\mathcal V^{(j)}$ see \ref{sub:vn}. In  \ref{sub:disjoint} we find that the  isotypic component $V_\tau^G[V_\sigma^H][Z] =\mathcal L_{W,H}^c[Z]$ is contained in some $\mathcal V^{(k)}$ for each irreducible factor $V_\sigma^H$ of $res_H(V_\tau^G)$, of course, $k$ depends on $V_\sigma^H$. Thus, the inclusion  $V_\tau^G[V_\sigma^H][Z] \subset \mathcal V^{(k)}$ for some $k$, yields all the symmetry breaking operators onto the same Discrete Series $V_\sigma^H$ are represented by homogeneous differential operators and they share  the same degree.

The {\it zero degree symmetry breaking operators} are constructed as follows: we fix a decomposition in irreducible representations for $L$ of $res_L(W)=\sum_j (\sigma_j, Z_j)$, let $P_j$ denote the projector onto $Z_j$ along $\sum_{i\not= j}Z_i$.  Let $r$ denote restriction of functions  from $G$ (resp. $\mathcal D_G$) to functions on $H$ (resp. $\mathcal D_H$) Then, $P_j \circ r$ maps $ V_\tau^G$ onto $V_{\sigma_j}^H$ and is a zero order operator. In this case $K_S(\cdot,e)^\star z =P_j(K_\tau(\cdot,e)^\star z) \in \mathcal V^{(0)}$, that is, it is a homogeneous operator.

{\it First order symmetry breaking operators} are the one's so that

 $K_S(\cdot,e)^\star z \in \mathcal V^{(0)}+ \mathcal V^{(1)}, \forall z \in Z$. Now in \cite[Proposition 1]{OV3}, see \ref{sub:disjoint}, it is shown that the $L$-modules $res_L(\mathcal V^{(0)})=res_L(W)$, $res_L(\mathcal V^{(1)})$ are disjoint, that is, they do not have a $L$-type in common, together with the fact that the map $z \mapsto K_S(\cdot,e)^\star z$ is a $L$-map, let us obtain
 $K_S(\cdot,e)^\star z \in   \mathcal V^{(1)}, \forall z \in Z$. That is, the {\it first order symmetry breaking operators} are "homogeneous" differential operators. We explicit that formula \ref{eq:sasdfo} together with the comment after the same, let us conclude that in holomorphic coordinates, $\sum_{\beta \in \Psi_n^\h} z_\beta E_\beta + \sum_{\gamma \in \Psi_n^{\h_0}} w_\gamma E_\gamma \in \p_\g^+$, a first order symmetry breaking operator has the shape \begin{equation}\label{eq:firstorder}  res_{\mathcal D_H}(\sum_{\beta \in \Psi_n(\h)} c_\beta(z,w)  \partial_{z_\beta} + \sum_{\gamma \in \Psi_n(\h_0)} d_\gamma(z,w) \partial_{w_\gamma} ) . \end{equation} Here, $c_\beta, d_\gamma $ are in  $\mathcal O(\mathcal D_G, Hom_\C( W,Z))$. The normal derivative operator is when  that $c_\beta $ vanishes for every $\beta$. \cite{KP1} has shown the deeper result under our hypothesis and holomorphic setting that every symmetry breaking operator is represented by a constant coefficient holomorphic operator.

 In \ref{sec:algebraic},  \ref{sub:7}   we have analyzed the inclusion  of  the subspace of first order normal derivative differential operators in the space of first order symmetry breaking operators, that is, we studied the intersection   $\mathcal U(\h_0)W
\cap \mathcal V^{(1)}\cap   \mathcal L_{W,H}^c$. To follow, we describe the results. We split up the presentation according to $(G,H)$ is so that $\k_{ss}$ is a simple Lie algebra, that is, $\g\ncong \mathfrak{su}(m,n), m>1,n>1$.   Proposition~\ref{prop:first} address the case  $\k_{ss}$ is the sum of two simple Lie algebras.
\begin{prop}\label{prop:nonormal}Our hypotheses
  are: $\g$ is a simple Lie algebra, $(G,H)$ a symmetric pair, $(G,H_0)$ associated pair, $H/L \rightarrow G/L$ is a holomorphic immersion, $(\tau, W)$(resp. $(\sigma, Z)$) lowest $K$-type (resp. lowest $L$-type) of holomorphic Discrete Series $V_\tau^G$ (resp. $V_\sigma^H$).

a) When $\k_{ss}$ is a simple Lie algebra (equivalently $\g\ncong \mathfrak{su}(m,n), n>1,m>1$) we have that every first order symmetry breaking operator is represented via normal derivatives if and only if $(\tau, W)$ is a one dimensional representation.

b) When  $\k_{ss}$ is a simple Lie algebra, $ [[\p_h^+, \p_{h_0}^-],\p_{h_0}^-]=\{0\}$ and $\dim W>1$. Then,   the subspace of first order symmetry breaking operator represented via normal derivatives is a nonzero and  proper subspace of the space of first order symmetry breaking operators.

c) When  $\k_{ss}$ is a simple Lie algebra, $ [[\p_h^+, \p_{h_0}^-],\p_{h_0}^-]\not=\{0\}$,  $\dim W>1$,  Table B shows   the pairs  so that there is no nonzero first order normal derivative symmetry breaking operator.

d) For the pair $(\mathfrak{sp}(n+1, \R), \mathfrak{sp}(n, \R)+\mathfrak{sp}(1, \R))$ in \cite[Theorem 5.3]{KP2} it is shown there is no normal derivative symmetry breaking operator of positive order.
\end{prop}
A proof of the Proposition follows from Proposition~\ref{prop:nu1nonzero} Table A and Proposition~\ref{sub:7}. We recall that among  the admissible holomorphic pairs $(G,H)$, and  $\g$ is simple,   the unique   Lie algebra so that $\k_{ss}$   is a direct sum of two simple ideals is $\g \cong \mathfrak{su}(m,n), m>1,n>1$, in this case $\k_{ss} \cong \mathfrak{su}(m)+\mathfrak{su}(n)$.    These follows from   the list  \cite[Table 3]{OV4}\cite[Table 2]{KOadv}. To complement the previous Proposition
we  present, for the totality of  symmetric pairs $(\mathfrak{su}(m,n),\h), m>1,  n>1 $, some results  on first order normal symmetry breaking operators.
\begin{prop}\label{prop:first} The hypothesis is: $\g \cong \mathfrak{su}(m,n), m>1,n>1$, $(G,H)$ a symmetric pair, $(G,H_0)$ associated pair, $H/L \rightarrow G/L$ is a holomorphic immersion, $(\tau, W)$(resp. $(\sigma, Z)$) lowest $K$-type (resp. lowest $L$-type) of holomorphic Discrete Series $V_\tau^G$ (resp. $V_\sigma^H$).

d1)  Every first order symmetry breaking operator from \\ \phantom{xxxccccc}$V_\tau^{SU(m,n)} $ into $ V_\sigma^{SU(p,q)\times SU(m-p,n-q)},\,\,1\leq p <m, 1\leq q <n$ \\ is represented via a normal derivative differential operator if and only if  $(\tau,W)$ is a scalar representation.

d2)   Every first order symmetry breaking operator from $V_{\tau}^{SU(m,n)} $ to $ V_{\sigma}^{SU(m,q)\times SU(n-q)}, \, 1\leq q <n$ is represented via a normal derivative differential operator if and only if  $(\tau,W)$ is a scalar representation or $Ker(\tau)$ contains one of the simple ideals of $k_{ss}$. Equivalently, if and only if $(\tau, W)$ restricted to   $\k_{ss}$ is not a faithful representation.

d3) When $(\tau, W)$  restricted to  $\k_{ss}\cong  \mathfrak{su}(m)+\mathfrak{su}(n)$ is a faithful representation, the subspace of first order symmetry breaking operators represented  via a normal differential operator from $V_\tau^{SU(m,n)}\rightarrow V_\sigma^{SU(p,q)\times SU(m-p,n-q)}$ is nonzero and a proper subspace of the space of  first order symmetry breaking  operators.

d4) Every first order symmetry breaking operator from $ V_\tau^{SU(n,n)} \rightarrow V_\sigma^{Sp(n,\R)}$, or, from  $  V_\tau^{SU(n,n)}  \rightarrow V_\sigma^{SO^\star(2n)}$, is represented via a normal derivative operator if and only if $(\tau,W)$ is a scalar representation.

d5) For    a  non scalar representation    $(\tau, W)$  of $\k\cong \mathfrak z_\k +\mathfrak{su}(n)+\mathfrak{su}(n)$, and  first order symmetry breaking operators from $ V_\tau^{SU(n,n)} \rightarrow V_\sigma^{Sp(n,\R)}$, or, from $  V_\tau^{SU(n,n)}  \rightarrow V_\sigma^{SO^\star(2n)}$, we have that   the subspace of first order symmetry breaking operators  represented by means of a normal differential operators is nonzero and a proper subspace of the space of the first order symmetry breaking  operators.
\end{prop}
The proof of the Proposition follows from Proposition~\ref{prop:nu1nonzero} Table A and Proposition~\ref{sub:7}.
\begin{cor} For any semisimple, not necessary simple, Lie algebra,  and $(\tau, W)$ a one dimensional representation for $\k$, we have that every first order symmetry breaking operator from $V_\tau^G$ is represented via a homogeneous normal derivative operator of degree one.
\end{cor}

\subsection{Higher order symmetry breaking operators} In this subsection we show that among the representations of symmetry breaking operator via a differential operators, we may choose a homogeneous differential operator.  For this, we first determinate  that the $L$-isotypic component  of the restriction of representation of $H$ in the $H$-isotypic component  $V_\tau^G[V_\sigma^H]$, is  homogeneous, in the sense that it lies on a subspace of a homogeneous component $\mathcal V^{(\cdot)}$. Next, we analyse some consequences.
\subsubsection{Isotypic components for $L$}\label{sub:disjoint} In order to analyze the $L$-types of $res_H(V_\tau^G)$, among the subspaces   we have to consider are:

$(V_\tau^G)_{K-fin}[(V_\sigma^H)_{L-fin}][Z] \subseteq (V_\tau^G)_{K-fin}[Z] \subseteq (V_\tau^G)^{\infty}][Z]\subseteq V_\tau^G[Z]$.

\noindent
In \cite{Kobdd3} it is shown that the $H$-admissibility hypothesis yields the second and third inclusion are actual  equalities. Moreover,   the subspace $(V_\tau^G)_{K-fin}[(V_\sigma^H)_{L-fin}][Z]$ is equal to $(V_\tau^G)_{K-fin}[(V_\sigma^H)_{L-fin}] \cap V_\tau^G[Z]$ and is the subspace   generated by the lowest $L$-type subspace of each $H$-factor of $V_\tau^G$ isomorphic to $V_\sigma^H$.
   To follow, we show the subspace $(V_\tau^G)_{K-fin}[(V_\sigma^H)_{L-fin}][Z]$ is contained in  a homogeneous subspace: \begin{lem} We assume $\g$ is a simple Lie algebra and   $H$-admissibility of $(L_\cdot^\tau , V_\tau^G))$. Then, the isotopic component $(V_\tau^G)_{K-fin}[(V_\sigma^H)_{L-fin}][Z]$, determined by the lowest $L$-types  of the representation of $(\h,L)$  in the  isotypic component $(V_\tau^G)_{K-fin}[(V_\sigma^H)_{L-fin}]$   is contained in a homogeneous subspace $\mathcal V^{(k)} =L_{ S^k(\p_\g^- )}^\tau ( W) .$ \end{lem}

 \begin{proof} The Lemma follows after we show that for $k \not= t$,  $S^k(\p_\g^-)\otimes W$   and  $S^t(\p_\g^-)\otimes W$ have no irreducible $L$-factor in common. This happens, owing to the subspace of $K$-finite vectors in $V_\tau^G$, $(V_\tau^G)_{K-fin}$, is $K$-isomorphic to $S(\p_\g^-)\otimes W=\oplus_{k\geq 0} \, S^k(\p_\g^-)\otimes W$. A particular  $K$-isomorphism is $L_D^\tau (D\otimes w)\leftarrow D\otimes W \in S(\p_\g^-)\otimes W$ and in \ref{sub:vn} we have defined  $\mathcal V^{(k)}= L_{S^k(\p_\g^-)}^\tau(W)$. The action of $K$ in $S(\p_\g^-)\otimes W$ is
 the one that is obtained via the Adjoint representation of $K$ in $\p_\g^-$ tensor with the action $\tau$ of $K$ in $W$. Now, since $G/K$ and $\g$ is a simple Lie algebra the representation $K$ in $\p_\g^+$ is irreducible and the center $\z_\k$ of $\k$ is one dimensional. Shur's lemma gives  $\z_\k$ acts by one dimensional representation either on $\p_\g^-$ or $W$, they are respectively: $\z_\k \ni Y \mapsto \beta (Y), Y \mapsto \chi_\tau (Y)$. Thus, $\z_\k $ acts in $S^k(\p_\g^-)\otimes W$ by the representation $Y \mapsto k\beta (Y)+ \chi_\tau (Y)$. Whence, we conclude: the representations of $K$ in  $S^k(\p_\g^-)\otimes W$ and  $S^t(\p_\g^-)\otimes W$ for $t\not= k$ are {\it not   equivalent} due that respective representations of $\z_\k$ are not equivalent.

 To follow we consider a holomorphic immersion $H/L \rightarrow G/K$ and a holomorphic representation $V_\tau^G$ of $G$ so that its restriction to $H$ is admissible. We claim: {\it for each $L$-isotypic component $V_\tau^G [V_\sigma^H][Z]$, there exists $k$ so that $V_\tau^G [V_\sigma^H][Z]\subseteq S^k(\p_\g^-)\otimes W$}. Indeed, the hypothesis holomorphic immersion, $\g$ is a simple Lie algebra and $H$-admissibility of $res_H(V_\tau^G) $ implies $\z_\k \subset \l $. For a proof see \cite[Table 2.1.3]{Kob9}. Since we are considering the $L$-isotypic component $V_\tau^G [V_\sigma^H][Z]$ we have that center of $\l$, $\z_\l$, acts by  a {\it fixed} one dimensional representation   on $V_\tau^G [V_\sigma^H][Z]$. Since, $\z_\k \subseteq \z_\l$, the previous observations on the action of $\z_\k$ on $S^\cdot(\p_\g^-)\otimes W$ yields $V_\tau^G [V_\sigma^H][Z]\subseteq S^k(\p_\g^-)\otimes W$ for some $k$. \end{proof}

\subsubsection{A structural fact on symmetry breaking operators}
 We deduce that every symmetry breaking operator, under our hypothesis, admits a representation via a homogeneous differential operator
 \begin{prop}\label{prop:diffarehomo}For $\g$ a simple Lie algebra, every symmetry breaking linear operator from $V_\tau^G$ onto $V_\sigma^H$ is represented by a homogeneous differential operator.
 \end{prop}
 \begin{proof}In \ref{sub:disjoint} we have shown the inclusion $V_\tau^G [V_\sigma^H][Z]\subseteq S^k(\p_\g^-)\otimes W$. In \cite[Lemma 4.2]{OV2},\cite{OV4} we obtained that for symmetry breaking operator $ S: V_\tau^G \rightarrow V_\sigma^H$ the kernel $K_S$ that represents $S$ as integral operator satisfies $K_S(\cdot, e)^\star z \in V_\tau^G [V_\sigma^H][Z]$ for each $z\in Z$. Thus, $Z_S \subset V_\tau^G [V_\sigma^H][Z]$, next,   the technique in \ref{sub:tech} let us to construct  by means of  \ref{eq:sasdfo},\ref{prop:uniqness} a homogeneous differential operator that represents $S$.
 \end{proof}

\subsubsection{Sufficient condition for existence of normal derivative operators} In previous paragraphs we have analyzed the existence of first order normal derivative symmetry breaking operators. To follow, we consider higher order normal derivative operators.
\begin{prop}\label{prop:existencenormal2} We assume $\g$ is a simple Lie algebra, the triple bracket $[\p_\h^+, \p_{\h_0}^-],\p_{\h_0}^-] =\{0\}$ and $(\tau, W)$ is a nonscalar representation. Then, for every $n\geq 1$,  there exists an irreducible factor $(L_\cdot^H, V_\sigma^H)$ for the restriction   of $(L_\cdot^G,   V_\tau^G) $ to $H$, and a nonzero normal derivative operator  of order $n$ in $Hom_H( V_\tau^G, V_\sigma^H)$.
\end{prop}

In fact, under the above hypothesis in Proposition~\ref{prop:exisnormal}, Proposition~\ref{prop:nu1nonzero} we show that $\mathcal V^{(n)}\cap \mathcal  L_{W,H}^c \cap \mathcal U(\h_0)W \not= \{0\}$. By definition $\mathcal L_{W,H}^c=\cup_{\sigma \in \hat L: V_\sigma^H \hookrightarrow V_\tau^G} \cup_{S \in Hom_H( V_\tau^G, V_\sigma^H)}  Z_S$. Besides, $Z_S$ is an irreducible $L$-module, thus, for each $n$ there exists $\sigma$ and $S$ such that $Z_S \subset \mathcal V^{(n)}\cap \mathcal  L_{W,H}^c \cap \mathcal U(\h_0)W$. Whence, we have verified   Proposition~\ref{prop:existencenormal2}.

We would like to point out, that under the hypothesis of   Proposition~\ref{prop:existencenormal2}, not every symmetry breaking operator of order $n$ is normal derivative operator as Proposition~\ref{prop:nonormal} shows. Next, we analyze this matter. For this we recall that under our setting each symmetry breaking operator is represented by a differential operator.

The size of  quotient of the set of order $n$ symmetry breaking operators
\begin{equation*}  \cup_{\sigma \in \widehat L: V_\sigma^H \hookrightarrow V_\tau^G}\,\,    Hom_H( V_\tau^G, V_\sigma^H)_{order\, n}     \end{equation*}
\noindent
  divided by the subset of symmetry breaking operators represented via normal derivative operators \begin{equation*}  \cup_{\sigma \in \widehat L: V_\sigma^H \hookrightarrow V_\tau^G}   \,\,\, Hom_H( V_\tau^G, V_\sigma^H)_{normal\,order\, n}      \end{equation*} is measured  by

\begin{center}$\mathcal V^{(n)}\cap(\sum_{ (\sigma,Z) \in \widehat L: V_\sigma^H \hookrightarrow V_\tau^G),  \,    S \in Hom_H( V_\tau^G, V_\sigma^H)  }    \, Z_S)$\end{center} divided by

\begin{center}$\mathcal V^{(n)}\cap(\sum_{ \{ (\sigma,Z) \in \widehat L: V_\sigma^H \hookrightarrow V_\tau^G),  \, S \in Hom_H( V_\tau^G, V_\sigma^H) :\, Z_S \subset \mathcal U(\h_0)W   \} } Z_S).$ \end{center}

\noindent
The above quotient    is the quotient \\ \phantom{xxxxxxxxxxxxxxxx}$(\mathcal V^{(n)}\cap \mathcal L_{W,H}^c )/\mathcal V^{(n)}\cap (\mathcal L_{W,H}^c \cap \mathcal U(\h_0)W)$.

\noindent
In \ref{sub:w0irred} we show that under the hypothesis in Proposition~\ref{prop:existencenormal2}, the subspace $\mathcal V^{(n)}\cap (\mathcal L_{W,H}^c \cap \mathcal U(\h_0)W)$ contains the irreducible $L$-module of lowest weight equal to the lowest weight of $\tau$ and formulate a conjecture on an eventual structure for $\mathcal V^{(1)}\cap (\mathcal L_{W,H}^c \cap \mathcal U(\h_0)W)$.

\smallskip
The next Proposition complements Proposition~\ref{prop:existencenormal2}.
  The Proposition  establishes a   consequence of assuming that every symmetry breaking operator of a fixed order greater than one is represented by means of normal differential operators.
\begin{prop}\label{prop:everynorfortriple}We assume there exists $k\geq 2$ so that  \begin{multline*} \cup_{\sigma \in \widehat L: V_\sigma^H \hookrightarrow V_\tau^G}\,\,    Hom_H( V_\tau^G, V_\sigma^H)_{order\, k} \\ = \cup_{\sigma \in \widehat L: V_\sigma^H \hookrightarrow V_\tau^G}\,\,    Hom_H( V_\tau^G, V_\sigma^H)_{normal\, order\, k}.\end{multline*}
Then,  $[\p_\h^+, \p_{\h_0}^-] \subset Ker(\tau)$ and $[[\p_\h^+, \p_{\h_0}^-], \p_{\h_0}^-]=\{0\}$.
\end{prop}
\begin{rmk}a) The triples $(\g, \h, \h_0)$ so that $[[\p_\h^+, \p_{\h_0}^-], \p_{\h_0}^-]=\{0\}$ are listed in \ref{sub:triplepairszero}. \\ b) Whenever $\k_{ss}$ is a simple Lie algebra, and the triple bracket vanishes,  hence, $(\g,\h)$ is one of the: $(\mathfrak{so}(2m,2), \mathfrak u(m,1)),$   $(\mathfrak{so}^\star (2n), \mathfrak u(1,n-1)),$ $(\mathfrak{so}^\star(2n),  \mathfrak{so}(2) +\mathfrak{so}^\star(2n-2)),$
$(\mathfrak e_{6(-14)}, \mathfrak{so}(2,8)+\mathfrak{so}(2)).$   the inclusion $[\p_\h^+, \p_{\h_0}^-] \subset Ker(\tau)$  forces   $\tau$ to be a one dimensional representation. \\ c)  For $\g\cong \mathfrak{su}(m,n), m>1,n>1 $ and  $\h=\s(\mathfrak{su}(m,p)+\mathfrak{su}(n-p)+\u(1))$,  then,  $\tau$ is either a scalar representation or $\tau \cong \pi^{\z_\k}\otimes \pi^{\mathfrak{su}(m)}\otimes \pi_0^{\mathfrak{su}(n)}$.
\end{rmk}
\begin{proof} The duality Theorem \cite[Theorem 1]{OV3} \cite[Theorem 4.6]{OV4} yields $\dim \mathcal V^{(k)} \cap \mathcal L_{W,H}^c= \dim \mathcal V^{(k)}\cap \, \mathcal U(\h_0)W$. Our hypothesis is that every symmetry breaking operator $S$ of order $k$ is normal, whence,   for every symmetry breaking operator $S$ of order $k$ we have  $Z_S \subset \mathcal V^{(k)}\cap \, \mathcal U(\h_0)W$. Now, the fact that  subset of $V_\tau^G$ determinate for each lowest $L$-type for the $H$-isotypic component for $res_H(V_\tau^G$ is contained in a same  homogeneous  component (\ref{sub:disjoint}) together with Proposition~\ref{prop:uniqness}, yields $\mathcal V^{(k)} \cap \mathcal L_{W,H}^c=\sum_S Z_S$ where $S$ runs over the totality of symmetry breaking operators of order $k$. Thus,  $\mathcal V^{(k)} \cap \mathcal L_{W,H}^c\subset  \mathcal V^{(k)}\cap \, \mathcal U(\h_0)W$.  The equality of dimension forces they are equal, and hence, we are in the hypothesis of  Proposition~\ref{prop:orderkisnormal} (see next section),     whence, we obtain the conclusion. \end{proof}
\begin{rmk} The same proof as the one for the Proposition yields: If every first order symmetry breaking operator is represented via a normal differential operator, then,
 $[\p_\h^+, \p_{\h_0}^-] \subset Ker(\tau)$. Hence, when $\k_{ss}$ is a simple Lie algebra,  $\tau$ is a scalar representation, and, for    $\g\cong \mathfrak{su}(m,n), m>1,n>1 $, in Prop's \ref{prop:nonormal}, \ref{prop:first}   we present an explicit answer.\\
\end{rmk}
\begin{rmk}We consider $(\tau, W)$ so that every first order symmetry breaking operator is represented via normal operators (\ref{prop:nonormal}), then, it holds the equivalence: {\it Every symmetry breaking operator is represented by normal differential operator if and only if the triple bracket $[[\p_\h^+, \p_{\h_0}^-], \p_{\h_0}^-] $ vanishes}. The statement readily  follows from Proposition~\ref{prop:equal}.
\end{rmk}
\begin{rmk} In Proposition~\ref{prop:nu1nonzero} we carry out a case by case  analysis on the triples $(\g,\h,\h_0)$ of when there exist (resp. does not exist) first order normal differential symmetry breaking operators for a given   nonscalar representation $(\tau,W)$ . We would like to point out that both Table A,  Table B in \ref{prop:nu1nonzero} are non symmetric on   $\h, \h_0 $.
\end{rmk}
The next Proposition complements a result in  \cite[Theorem 5.3]{KP2}.
\begin{prop}\label{prop:sonn-1}For the triple $(\g=\mathfrak{so}(n,2), \h=\mathfrak{so}(n-1,2)+\mathfrak{so}(1), \h_0=\mathfrak{so}(n-1)+\mathfrak{so}(1,2))$ we have

a) For a scalar representation of $K$, a  symmetry breaking operator is a normal derivative operator if and only if the order of the symmetry breaking operator is zero or one.

b)  For a non scalar irreducible representation of $K$, a  symmetry breaking operator is a normal derivative operator if and only if the order of the symmetry breaking operator is zero.
\end{prop} The Proposition follows from Lemma~\ref{lem:sonn-12} and the observation:  a symmetry breaking operator $S$ is represented via normal derivative if and only if  $Z_S \subset \mathcal U(\h_0)W$.

The proof of the Proposition is carried out in \ref{lem:sonn-12}, \ref{sub:interequalceronosigma}, \ref{eq:so2n21}, \ref{eq:so2n2n-1}.

\subsubsection{Tensor product} To continue, we present a result on first order symmetry breaking operators for the tensor product of two holomorphic Discrete Series representations. A quite substantial difference from the case $\g$ simple Lie algebra we have dealt with, hinges on that the center of $\k $ is no longer a subset of the subalgebra $\l$. In this case $\z_\k=\z_\l \oplus   \z_\k \cap \q$ and both summands are one dimensional.   Compare the following statement with \ref{prop:nonormal}. For notation and hypothesis of the next Proposition see \ref{sub:tensorpro}.
\begin{prop}For the tensor product of two holomorphic Discrete Series representations of simple Lie groups restricted to the diagonal subgroup, the totality of first order normal derivative operators is equal to the subspace of first order symmetry breaking operators if and only if the lowest $K$-type is one dimensional and restricted to   $\k_\C\cap \q_\C$   is the trivial representation.
\end{prop} The proof of the Proposition is the statement and proof in  \ref{prop:nu1fortensor}.

In \ref{sub:w0irred} we have developed a way to construct first order normal derivative symmetry breaking operators when $\g$ is a simple Lie algebra, in \ref{rmk:tensor2}, we verify the technique does not apply to the case of tensor products.

  \section{Graded structure  of   $\mathcal U(\g)\otimes_{\mathcal U(\k_\C +\p_\g ^+)} W, \,\mathcal L_{W,H}^c, \, \mathcal U(\h_0)W$ }\label{sec:algebraic}

In this section we present the algebraic  proofs of some of the results in section on symmetry breaking operators. The hypothesis for the whole section are as before, except for   \ref{sub:tensorpro}, $\g$ is a simple Lie algebra, $(G,H)$ a symmetric pair, $H_0$ the associated subgroup of $G$, as well as,  $\k,\l, \Psi_\g, \Psi_\h$, $\Psi_{\h_0}$, $\Psi_n(\g), \Psi_n(\h), \Psi_n(\h_0)$, $\p_\g, \p_\h, \p_{\h_0}$, $H/L\rightarrow G/K$ is a holomorphic immersion. The underlying Harish-Chandra module of the holomorphic Discrete Series we are dealing with is isomorphic to the Verma module induced from the Harish-Chandra, Siegel parabolic subalgebra by the representation of $\k$ determinate by the lowest $K$-type of the holomorphic representation we are dealing with. Precisely, the Harish-Chandra module is isomorphic to the Verma module   $\mathcal U(\g)\otimes_{\mathcal U(\k_\C +\p_\g ^+)} W$. Here, $(\tau, W)$ is an irreducible representation of $(\k,K)$ extended by zero to $\p_\g ^+$. Henceforth, the equivalence class of $D\otimes w, D \in \mathcal U(\g), w \in W$ in  $\mathcal U(\g)\otimes_{\mathcal U(\k_\C +\p_\g ^+)} W$ is denoted by $[D\otimes w]$. Thus, $[DY\otimes w] =[D\otimes \tau (Y)w], Y \in \k_\C$ and $[DY\otimes w] =0$ for $Y\in \p_\g^+$. The action of $D\in \mathcal U(\g)$ on $\mathcal U(\g)\otimes_{\mathcal U(\k_\C +\p_\g ^+)} W$ is denoted by $L_D^\tau$.

     Due that the Lie algebra $\p_\g^-$  is a abelian Lie algebra, the symmetrization from $S(\p_\g^-)$ onto $\mathcal U(\p_\g^-)$ becomes an associative algebra isomorphism, thus, we may and will think of the isomorphism $S(\p_\g^-)\otimes W$ onto $\mathcal U(\p_\g^-)\otimes W$ as an equality.    We recall   the $K$-isomorphism, as well as $\mathcal U(\p_\g^-)$-map,   from $\mathcal U(\p_\g^-)\otimes W $ to $\mathcal U(\g)\otimes_{\mathcal U(\k +\p^+)} W $ defined by  $D\otimes w \mapsto [D\otimes w]$.
     Then,  by means of the action   of $\p_\g^-$,  the space $\mathcal U(\g)\otimes_{\mathcal U(\k_\C +\p_\g ^+)} W$  inherits    a  graduated vector space structure, where the $n^{th}$ subspace is:
\subsubsection{}\label{sub:vn}
     $\mathcal V^{(n)}:=lin.span_\C \{L_{x_1\cdots x_n}^\tau (w)=[x_1\cdots x_n \otimes w], x_j \in \p_\g^-, w\in W\}$. The isomorphism  $ \mathcal U(\g)\otimes_{\mathcal U(\k +\p^+)} W \cong (V_\tau^G)_{K-fin}$ and the duality determinate by the Killing form between $\p_\g^-$ and $\p_\g^+$ yields:   {\it The subspace  $\mathcal V^{(n)}$ is equal to the subspace of degree $n$ holomorphic homogeneous polynomials in $\p_\g^+$}. For a proof see \cite{JV}.

\subsection{Analysis of the inclusion $\mathcal U(\h_0)W\cap \mathcal V^{(k)}\subset \mathcal L_{W,H}^c $, for $k\geq 1$}
 We recall  our hypothesis of $H/L \rightarrow G/K$ is a holomorphic immersion and we are analyzing a holomorphic    $H$-admissible representation whose underlying Harish-Chandra module is identified with the Verma module $\mathcal U(\g)\otimes_{\mathcal U(\k +\p^+)} W $.
To begin with, we present technical results in order to describe the main facts.
     \subsubsection{The subspace $[[\p_{\h_0}^-, \p_\h^+],\p_{\h_0}^- ]$}\label{sub:triplebrack} Owing to our hypothesis, $\h_\C$ is equal to the fix point set of an involution $\sigma$ that commutes
      with the Cartan involution, it readily follows that $[[\p_{\h_0}^-, \p_\h^+],\p_{\h_0}^- ]\subset \p_\h$.

   \begin{claim}\label{prop:triplein}  Actually,   we have $[[\p_{\h_0}^-, \p_\h^+],\p_{\h_0}^- ]\subset \p_\h^-$.
   \end{claim}

      In fact,   let $\beta$ denote the noncompact simple root in $\Psi$, whenever $U=T$,  then a set of generators for $[[\p_{\h_0}^-, \p_\h^+],\p_{\h_0}^- ]$ is $\{[[E_{-\gamma}, E_{\alpha}],E_{-\epsilon}]\}$, for convenient noncompact roots  $\alpha, \gamma, \epsilon $ in $\Psi$,    each nonzero    triple bracket     implies   that $\alpha -\gamma -\epsilon $ is a root. The coefficient of $\beta$ in the sum is $-1$. Whence, $\alpha -\gamma -\epsilon $   is a negative root. Hence, a set of generators for $[[\p_{\h_0}^-, \p_\h^+],\p_{\h_0}^- ]$  is a subset of $\p_\h^-$  and we have verified the claim for $U=T$.

When $U\not= T$,    we have to consider $(\g, \h)=(\mathfrak{su}(n,n), \mathfrak{sp}(n,\mathbb R))$, $(\g, \h)=(\mathfrak{so}(2,2(n-1)), \mathfrak{so}(2,2k+1)\oplus \mathfrak{so}(2(n-1)-2k-1))$ and their respective associate pairs.
For both pairs,  the representation of $L$ in $\k/\l $ is absolutely irreducible, hence,  it  follows the equality $[\p_{\h_0}^-, \p_\h^+]=\q_\C \cap \k_\C$. Thus,   $[[\p_{\h_0}^-, \p_\h^+],\p_{\h_0}^- ]\not = 0$, hence, for the involved roots, we have $\alpha, \gamma, \epsilon$ in $\Psi_n $
so that $\alpha -\gamma -\epsilon $ is a root. Owing to the previous observations on $\alpha, \gamma, \epsilon, \beta$, the coefficient of $\beta$ in the sum is $-1$. Now, the verification follows as the case $U=T$.

\subsubsection{} \label{sub:triplepairszero} In \cite[4.6]{Va},  we find the list of symmetric   pairs $(\g,\h)$ that satisfies the condition $[[ \p_{\h_0}^-,\p_\h^+ ],\p_{\h_0}^-]=\overline{[[ \p_{\h_0}^+, \p_\h^- ],
\p_{\h_0}^+]}=\{0\}$  are:
\begin{center} $ (\mathfrak{su}(m,n), \mathfrak{su} (m,l) +\mathfrak{su} ( n-l)+\mathfrak u(1)),$ $(\mathfrak{so}(2m,2), \mathfrak u(m,1)),$ \\ $(\mathfrak{so}^\star (2n), \mathfrak u(1,n-1)),$ $(\mathfrak{so}^\star(2n),  \mathfrak{so}(2) +\mathfrak{so}^\star(2n-2)),$\\
$(\mathfrak e_{6(-14)}, \mathfrak{so}(2,8)+\mathfrak{so}(2)).$
 \end{center}

It is also shown,   $[[ \p_{\h_0}^-,\p_\h^+ ],\p_{\h_0}^-]=\{0\}$ implies $[[ \p_{\h}^-,\p_{\h_0}^+ ],\p_{\h}^-]=\{0\}$.
\subsubsection{Computing $[\p_{\h}^+, \p_{\h_0}^-]$ when $[[\p_{\h}^+, \p_{\h_0}^-], \p_{\h_0}^-]=\{0\} $ }\label{sub:bracketphh0}

For a  symmetric pair $(\g,\h)$, and,  $\h_0, \Psi$, $\Psi_n(\h)$, $\q\cap k,  \p_\h   $ as in Section~\ref{sec:notation},  we  have partially computed the subspace $[\p_{\h}^+, \p_{\h_0}^-]$ in different places, the aim of this paragraph is to complete such a calculation.

{\it We assume the triple bracket is equal to zero, that is,   $[[\p_{\h}^+, \p_{\h_0}^-], \p_{\h_0}^-]=\{0\} $, we claim $[\p_{\h}^+, \p_{\h_0}^-]$ is a nilpotent abelian Lie algebra, more precisely, we verify $[\p_{\h}^+, \p_{\h_0}^-]$ is  equal to one of  $(\q_\C \cap \k_\C)^\pm$}.

According to \ref{sub:triplepairszero}, we are left to consider the  pairs in a), b), c), d):

a) $(\mathfrak{su}(m,n), \h=\mathfrak{su}(m,l)+\mathfrak{su}(n-l)+\mathfrak u(1)) (1\leq l <n)$ dual pair $ (\mathfrak{su}(m,n),\h_0=\mathfrak{su}(m,n-l)+\mathfrak{su}(l) +\mathfrak u(1))$.

The computation has been carried out in \ref{sub:b)} paragraph b3) b4). We have $[\p_{\h}^+, \p_{\h_0}^-]=(\q_\C \cap \k_\C)^-$,  $[\p_{\h_0}^+, \p_{\h}^-]=(\q_\C \cap \k_\C)^+$.

b) $(\mathfrak{so}(2m,2), \mathfrak u(m,1))$. $\Psi_c=\{(e_i \pm e_j):  1 \leq i<j\leq m \}$, $\Psi_n=\{(\delta \pm e_i):  1 \leq i \leq m \}$. $\Psi_n(\h)=\{(\delta -e_i  ):  1 \leq i\leq m \}$, $\Psi_n(\h_0)=\{(\delta +e_i  ):  1 \leq i\leq m \}$, $\Psi_c(\l)=\{(e_i - e_j):  1 \leq i<j\leq m \}$, $\Phi((\q_\C \cap \k_\C)^+)= \{(e_i + e_j):  1 \leq i<j\leq m \}$. Here, $[\p_{\h}^+, \p_{\h_0}^-]=(\q_\C \cap \k_\C)^-$.

c) $(\mathfrak{so^\star}(2m), \mathfrak{so}(2)+ \mathfrak{so^\star}(2m-2))$ dual pair $(\mathfrak{so^\star}(2m),  \mathfrak u(1,m-1))$.
$\Psi_c=\{(e_i - e_j):  1 \leq i<j\leq m \}$, $\Psi_n=\{(e_i + e_j):  1 \leq i<j\leq m  \}$. $\Psi_n(\h)=\{(e_i +e_j  ): 2 \leq i\not= j \leq m \}$, $\Psi_n(\h_0)=\{(e_1 +e_i  ):  2 \leq i\leq m \}$, $\Psi_c(\l)=\{(e_i - e_j):  2 \leq i<j\leq m \}$, $\Phi((\q_\C \cap \k_\C)^+)= \{(e_1 - e_j):  2\leq  j\leq m \}$. Here, $[\p_{\h}^+, \p_{\h_0}^-]=(\q_\C \cap \k_\C)^-$, $[\p_{\h_0}^+, \p_{\h}^-]=(\q_\C \cap \k_\C)^+$.

d) $(\mathfrak e_{6(-14)}, \mathfrak{so}(8,2)+\mathfrak {so}(2)), $
    $\k =\mathfrak {so}(10)    +\mathfrak {so}(2)$, $  \h_0 \equiv \h\equiv \mathfrak {so}(2,8) +\mathfrak {so}(2)$. Next, we pin down the positive roots for $\g, \h, \h_0$.

  The Vogan diagram for a holomorphic system for $\mathfrak e_{6(-14)}$ has fundamental roots (as in Bourbaki) $\alpha_1, \dots, \alpha_5, \alpha_6, $  the noncompact simple root is $\alpha_1$, the root $\alpha_2$ is  adjacent to both $\alpha_4$ and to the opposite of the maximal root. Maximal root is $\alpha_1 +2\alpha_3 +3 \alpha_4+2 \alpha_5+\alpha_6 +2\alpha_2$.

\noindent
 Then $\Phi(\k,\t)= \{\pm \sum_j a_j \alpha_j \in \Phi(\mathfrak e_6,\mathfrak t) :    a_1=0\}$; $\dim G/K= 32$; $\Psi_n(\h ,\t)=\{  \sum_j a_j \alpha_j \in \Phi(\mathfrak e_6,\mathfrak t) :  a_1=1, a_6=0\}=\{\beta_9,\dots, \beta_{16} \}, \\ \Psi_n(\h_0,  \t)=\{  \sum_j a_j \alpha_j \in \Phi(\mathfrak e_6,\mathfrak t) :  a_1=1, a_6=1\}=\{\beta_1, \dots, \beta_{16} \}$ and $\Phi(\l, \t)=\{ \pm \sum_j a_j \alpha_j \in \Phi(\mathfrak e_6,\mathfrak t) :  a_1=a_6=0\}.$ $\q_\C \cap \k_\C \equiv \k/\l \equiv \mathfrak {so}(10)/ (\mathfrak{so}(2)+ \mathfrak {so}(8))$ is the direct  sum of the abelian algebras $(\q_\C \cap \k_\C)^\pm =\sum_{\{ \alpha =\sum_j a_j \alpha_j  : a_1=0, a_6=\pm 1\}} (\e_6)_\alpha$, $\dim (\q_\C \cap \k_\C)^\pm =8 $. A direct computation, based on $-\beta_j +\beta_s =-\alpha_6 -\dots$ for  $\beta_s \in \Psi_n(\h)$, $\beta_j \in \Psi_n(\h_0)$   yields $[\p_\h^+ , \p_{\h_0}^-] = (\q_\C \cap \k_\C)^-$ and, via conjugation,  $[\p_{\h_0}^+ , \p_{\h}^-] = (\q_\C \cap \k_\C)^+$. It readily follows that
$[[\p_\h^+ , \p_{\h_0}^-], \p_{\h_0}^-]=\{0\}$.

\subsubsection{Resume: the values of $[\p_{\h_0}^-, \p_\h^+]$}  For $\g$ a simple Lie algebra, and  the pairs so that the triple bracket $[[\p_\h^+ , \p_{\h_0}^-], \p_{\h_0}^-]$ is zero,  the subspace $[\p_{\h_0}^-, \p_\h^+]$ is one of the proper  subspace of $(\q_\C \cap \k_\C)^\pm$.  For any other pair $(\g,\h)$  we obtain $[\p_{\h_0}^-, \p_\h^+]=\q_\C \cap \k_\C$.
  For $ (\g, \h) \cong  (\mathfrak{su}(m,n), \mathfrak{su}(k,n)\oplus \mathfrak{su}(m-k)\oplus \mathfrak u(1)) $  (details in \ref{sub:7}) and   $1\leq k < m $, then $[\p_{\h_0}^-, \p_\h^+]$ is equal to   $(\q_\C \cap \k_\C)^+$.

\smallskip

To follow, we show a fact that will have consequences on the structure of the symmetry breaking operators.
\begin{prop}\label{prop:orderkisnormal}We assume there exist $k\geq 2$ so that $\mathcal U(\h_0)W \cap \mathcal V^{(k)}
   \subseteq \mathcal L_{W,H}^c$. Then,   $[\p_\h^+, \p_{\h_0}^-] \subseteq Ker(\tau)$ and
$[[\p_{\h_0}^-, \p_\h^+],\p_{\h_0}^- ]=\{0\}$.  When, $\mathcal U(\h_0)W \cap \mathcal V^{(1)}
   \subseteq \mathcal L_{W,H}^c$. Then,   $[\p_\h^+, \p_{\h_0}^-] \subseteq Ker(\tau)$.
\end{prop}
  The case $k=1$ is analyzed in detail in  Proposition~\ref{sub:7}. The converse statement for $k=1$ is analyzed bellow.

\begin{proof}The case $k=2 $ is analyzed in \ref{prop:equal}. To show the underlying computation we carry out in detail the case $k=3$.

We fix
$ X \in \p_\h^+, Y_j \in \p_{\h_0}^-    $ and  we compute

\begin{equation*} \begin{split} L_X^\tau(Y_1Y_2Y_3\otimes w)& =[XY_1Y_2Y_3 \otimes w] \\ & = [Y_1XY_2Y_3 + [X,Y_1]Y_2Y_3 \otimes w ]
\\ & = [ Y_1Y_2XY_3\otimes w + Y_1[X,Y_2]Y_3\otimes w \\ & \quad \quad +Y_2[X,Y_1]Y_3\otimes w
+[[X,Y_1],Y_2]]Y_3 \otimes w ]
 \\ & = [ Y_1Y_2Y_3X\otimes w + Y_1Y_2[X,Y_3]\otimes w   \\ & \quad +Y_1Y_3[X,Y_2]\otimes w+Y_1 [[X,Y_2],Y_3]] \otimes w ]
 \\ & \quad \quad +Y_2Y_3[X,Y_1]\otimes w+Y_2 [[X,Y_1],Y_3]] \otimes w ]
 \\ & \quad \quad \quad + Y_3 [[X,Y_1],Y_2]] \otimes w + [[[X,Y_1],Y_2],Y_3 ]\otimes w
 \\ & = [ Y_1Y_2Y_3\otimes \tau(X) w + Y_1Y_2 \otimes \tau([X,Y_3]) w   \\ & \quad +Y_1Y_3\otimes \tau([X,Y_2]) w+Y_1 [[X,Y_2],Y_3]] \otimes w ]
 \\ & \quad \quad +Y_2Y_3 \otimes \tau([X,Y_1]) w+Y_2 [[X,Y_1],Y_3]] \otimes w ]
 \\ & \quad \quad \quad + Y_3 [[X,Y_1],Y_2]] \otimes w + [[[X,Y_1],Y_2],Y_3 ]\otimes w]
 \\ & = [ Y_1Y_2Y_3\otimes \tau(X) w + Y_1Y_2 \otimes \tau([X,Y_3]) w   \\ & \quad +Y_1Y_3\otimes \tau([X,Y_2]) w+ Y_2Y_3 \otimes \tau([X,Y_1])  w ]
 \\ & \quad \quad +Y_1 [[X,Y_2],Y_3]]  \otimes  w+Y_2 [[X,Y_1],Y_3]] \otimes w ]
 \\ & \quad \quad \quad + Y_3 [[X,Y_1],Y_2]] \otimes w + [[[X,Y_1],Y_2],Y_3 ]\otimes w].
 \end{split}
\end{equation*}
Now,  $X \in \p_\h^+ \Rightarrow \tau(X)w=0$, owing to \ref{prop:triplein} we have $[[X,Y_1],Y_2] \in \p_\h^- \subset \p_\g^-$, by hypothesis $Y_3 \in  \p_{\h_0}^- \subset \p_\g^-$ and $\p_\g^-$ is abelian, thus, the last summand is zero.
We are left with
\begin{equation} \label{eq:triple}\begin{split} L_X^\tau[Y_1Y_2Y_3\otimes w]& =
[ Y_1Y_2 \otimes \tau([X,Y_3]) w   \\ & \quad +Y_1Y_3\otimes \tau([X,Y_2]) w+ Y_2Y_3 \otimes \tau([X,Y_1])   w
 \\ & \quad \quad +Y_1 [[X,Y_2],Y_3]]  \otimes  w+Y_2 [[X,Y_1],Y_3]] \otimes w
 \\ & \quad \quad \quad + Y_3 [[X,Y_1],Y_2]] \otimes w ].
 \end{split}
\end{equation}
We choose $Y_j$ to be a linearly independent set, and  we fix a linear basis $Z_k$ for $\p_\h^-$, then $[[X,Y_i],Y_j]]=\sum_k c_{ij}^k  Z_k $. Since, $\p_{\h_0}^-$ is a complementary
space to $\p_\h^-$, the vectors $Y_iY_j, Y_i Z_k$ are linearly independent in $S(\p_\g)$ and the above sum is a sum in $S(\p_\g^-)\otimes W$, we obtain that $L_X^\tau[Y_1Y_2Y_3\otimes w]=0$ implies $c_{ij}^k$ vanishes for all $i,j,k$ and $\tau([X,Y_j])=0$ for all $X,Y_j$. That is, for $X \in \p_\h^+$, $Y_j \in \p_{\h_0}^-$, $L_X^\tau[Y_1Y_2Y_3\otimes w] =0$ implies $[[X,Y_i],Y_j]]=0$ and $[X,Y_j] \in Ker(\tau)$. Therefore, $[\p_\h^+, \p_{\h_0}^-]\subseteq Ker(\tau)$ and $[[\p_\h^+, \p_{\h_0}^-], \p_{\h_0}^-]=\{0\}$.

 To follow,   we analyze the case $Y_1=Y_2 \not= Y_3$ and linearly independent, then,  the equality becomes
\begin{equation*} \begin{split} L_X^\tau[Y_1^2Y_3\otimes w]& =
[ Y_1^2 \otimes \tau([X,Y_3]) w     +2Y_1Y_3\otimes \tau([X,Y_2]) w
 \\ & \quad \quad +2Y_1 [[X,Y_1],Y_3]]  \otimes  w+ Y_3 [[X,Y_1],Y_1]] \otimes w ].
 \end{split}
\end{equation*}
For the case $Y_1=Y_2=Y_3$ the equality is
\begin{equation*} \begin{split} L_X^\tau[Y_1^3\otimes w]& =
[ 3Y_1^2 \otimes \tau([X,Y_1]) w      +3Y_1 [[X,Y_1],Y_1]]  \otimes  w  ].
 \end{split}
\end{equation*}
After we argue as in case the vectors $Y_j$ are linearly independent,  we obtain a proof of  the Proposition for $k=3$.

In general, the formula for any $r\geq 2$, $X \in \p_\h^+, Y_j \in \p_{\h_0}^-    $ is
\begin{equation*} \begin{split} L_X^\tau[Y_1\cdots Y_r\otimes w]& =\sum_{1\leq j \leq r} [ Y_1 \cdots Y_{j-1} Y_{j+1}\cdots Y_r \otimes \tau([X,Y_j])(w)
 \\ &  +\sum_{1 \leq i<j \leq r} Y_1 \cdots Y_{i-1} Y_{i+1}\cdots \\ & \quad \quad \quad \cdots Y_{j-1}Y_{j+1}\cdots Y_r [[X,Y_i],Y_j]]   \otimes  w  ].
 \end{split}
\end{equation*}

The formula for any $a_j \geq 2  $, $X \in \p_\h^+$, $ Y_j \in \p_{\h_0}^- , j=1\cdots d=\dim_\C \p_{\h_0}^-   $,
 \begin{multline} \label{eq:formulapiX}
 L_X^\tau[Y_1^{a_1}\cdots Y_d^{a_d}\otimes w] \\
                    =\sum_{1\leq j \leq d} a_j [   Y_1^{a_1} \cdots Y_{j-1}^{a_{j-1}}Y_j^{a_j -1} Y_{j +1}^{a_{j+1}}\cdots Y_d^{a_d} \otimes \tau([X,Y_j])w \\
+\sum_{1\leq t \leq d} \binom{a_t}{2}  Y_1^{a_1} \cdots Y_{t-1}^{a_{t-1}}Y_t^{a_t -2} Y_{t+1}^{a_{t+1}}\cdots Y_d^{a_d}[[X,Y_t],Y_t]]  \otimes w \\
+\sum_{1\leq t \leq d} \sum_{s:s=t+1\cdots d} a_t a_{s}
 Y_1^{a_1} \cdots Y_{t-1}^{a_{t-1}} \\  \,\,\,\,\, \times
Y_t^{a_t -1}  Y_{t+1}^{a_{t +1}}\cdots Y_{s-1}^{a_{s-1}}  Y_{s}^{a_s -1}\cdots Y_d^{a_d} [[X,Y_{t}],Y_s]]   \otimes  w  ].
\end{multline}
\begin{multline} \label{eq:formulapiX2} a\geq 2, \,\,
 L_X^\tau[Y^{a} \otimes w]
                    =   [  a  Y^{a-1}  \otimes \tau([X,Y])w \\
+\binom{a}{2}  Y^{a -2} [[X,Y],Y]  \otimes w].
\end{multline}
\begin{equation*} \begin{split} L_X^\tau(Y_1Y_2\otimes w)& =[XY_1Y_2 \otimes w] \\ & = [Y_1XY_2 + [X,Y_1]Y_2 \otimes w ]
\\ & = [ Y_1Y_2X\otimes w + Y_1[X,Y_2]\otimes w \\ & \quad \quad +Y_2[X,Y_1]\otimes w
+[[X,Y_1],Y_2] \otimes w ]
\\ & = [Y_1Y_2\otimes \tau(X)w + Y_1 \otimes \tau([X,Y_2])w \\ & \quad \quad + Y_2\otimes  \tau([X,Y_1])w +
[[X,Y_1],Y_2] \otimes w]=\\ & = [ Y_1 \otimes \tau([X,Y_2])w \\ & \quad \quad + Y_2\otimes  \tau([X,Y_1])w +
[[X,Y_1],Y_2] \otimes w]. \end{split}
\end{equation*}
For the case $k=1$, for $X\in \p_\h^+, Y \in \p_{\h_0}^-, w \in W$,   after we recall $\tau(\p^+)W=\{0\}$, we have,
\begin{equation*} \begin{split} L_X^\tau[Y\otimes w]& = [ YX \otimes w]+[[X,Y] \otimes w]
 \\ &   = [ Y \otimes \tau(X)w  + 1    \otimes \tau([X,Y]) w  ]
 \\ &   = [   1    \otimes \tau([X,Y]) w  ]
 \end{split}
\end{equation*}
Thus, the inclusion    $\mathcal U(\h_0)W \cap \mathcal V^{(1)}
   \subseteq  \mathcal L_{W,H}^c$ forces   $[\p_\h^+, \p_{\h_0}^-] \subseteq Ker(\tau)$.
\end{proof}
 To follow, we show a sufficient condition that assures the existence of nonzero normal differential operators for   any order $n$.
\begin{prop}\label{prop:exisnormal} We assume $\g$ is a símple Lie algebra,  the equality $[[\p_{\h_0}^-, \p_\h^+],\p_{\h_0}^- ]=\{0\}$ and  $(\tau,W)$ is  a non scalar irreducible representation, then, for every $n\geq 1$, $\mathcal V^{(n)}\cap \mathcal L_{W,H}^c \cap \mathcal U(\h_0)W \not=\{0\}$.  \end{prop}
We carry out a analysis of $\mathcal L_{W,H}^c \cap \mathcal U(\h_0)W $ for the general case in \ref{prop:nu1nonzero}, \ref{sub:7}. We note  for $n=1$, it may happens $\mathcal V^{(1)}\cap \mathcal L_{W,H}^c \cap \mathcal U(\h_0)W =\{0\}$    (see
Table B), For any of these examples we have    $[[\p_{\h_0}^-, \p_\h^+],\p_{\h_0}^- ]\not=\{0\} $.\\  We recall for $G=SL_2(\R)\times SL_2(\R), H=SL_2(\R)$, $\mathcal L_{W,H}^c \cap \mathcal U(\h_0)W =W+\mathcal V^{(1)}\cap \mathcal L_{W,H}^c \cap \mathcal U(\h_0)W $.

\begin{proof}
We fix a linear basis $\{Y_j\}$ for $\p_{\h_0}^-$, the hypothesis forces that in  the formula \ref{eq:formulapiX} for $L_X^\tau([Y_1^{a_1}\cdots Y_d^{a_d}\otimes w])$ the second and third summands vanishes. We apply Engel's Theorem. For this we notice, that due to hypothesis $\tau$ is not a scalar representation, and that \ref{sub:bracketphh0} shows $[\p_{\h_0}^-, \p_\h^+]=(\q_\C \cap \k_\C)^-$ is a nilpotent Lie algebra, we have the family of operators $\tau([X,Y]), X\in \p_\h^+,  Y \in \p_{\h_0}^-$ in $W$ is nilpotent. Hence,  the  linear subspace
$W_0:=\{w \in W: \tau([X,Y_j])w=0, \forall X\in \p_\h^+, \forall Y \in \p_{\h_0}^-\}$ is nonzero. Therefore, for $w\in W_0$, \ref{eq:formulapiX} implies   $L_X^\tau([Y_1^{a_1}\cdots Y_d^{a_d}\otimes w])$ is equal zero and hence $[Y_1^{a_1}\cdots Y_d^{a_d}\otimes w] \in \mathcal V^{(n)}\cap \mathcal L_{W,H}^c \cap \mathcal U(\h_0)W$ for all $a_1,\cdots, a_d$ such that $\sum_i a_i=n$. That is, $S^n(\p_{\h_0}^- )\otimes W_0 $ is contained in $\mathcal V^{(n)}\cap \mathcal L_{W,H}^c \cap \mathcal U(\h_0)W$. \end{proof}

\subsubsection{Note}\label{sub:w0irred} a) Under the hypothesis in Proposition~\ref{prop:exisnormal}, $W_0$ is an irreducible representation of $L$. Indeed, we are dealing with holomorphic pairs $(K,L)$, actually $\p_\k^+= \p_\l^+ +(\q_\C \cap \k_\C)^+$,  hence, we would like to point out that {\it a $L$-lowest weight vector in $W_0$, is by the definition of $W_0$, a  $\k$-lowest weight   of $W$}, thus equal to a scalar multiple of the $\k$-lowest  vector of $W$. Thus, $W_0$ is $L$-irreducible and its lowest weight is equal to the lowest  weight of $W$.

b) Our
hypothesis is: triple bracket equal to zero and $\tau$ a nonscalar representation.

We fix a non zero $Y\in \p_{\h_0}^-$, let $W_Y :=\{ w \in W : \tau[X,Y](w)=0 \, \forall X \in \p_{\h}^+ \}$. Owing to our hypothesis the subspace, $[\p_{\h_0}^-, \p_\h^+]= (\q_\C \cap \k_\C)^-$ is an abelian nilpotent Lie algebra, hence, Engels Theorem yields $W_Y$ is a nonzero linear subspace.
Next, we consider $D=[Y\cdots Y \otimes w]$, (we multiply $Y$ $n$-times),  $w \in W_Y $. Then, for $X\in \p_{\h}^+$
  we apply \ref{eq:formulapiX2} and obtain $L_X^\tau(D)=0$. Thus, $0\not= D \in \mathcal V^{(n)}\cap \mathcal L_{W,H}^c \cap \mathcal U(\h_0)W_Y $. More generally, let $Y_1, \cdots, Y_n \in \p_{\h_0}^-$. We define for $ w \in W_{Y_1}\cap \cdots \cap W_{Y_n}$,  $D:=[Y_1 \cdots Y_n \otimes w]$. Then, formula \ref{eq:formulapiX} yields $D\in \mathcal V^{(n)}\cap \mathcal L_{W,H}^c \cap \mathcal U(\h_0)W$.

We would like to show: We fix $Y_j$    linear basis for $\p_{\h_0 }^-$. If $ \sum_j [Y_j \otimes w_j]\in \mathcal V^{(1)}\cap \mathcal L_{W,H}^c \cap \mathcal U(\h_0)W$, that is,   $L_X^\tau(\sum_j [Y_j \otimes w_j]) =0 \forall X \in \p_\h^+$, then $\forall\,j. \, w_j \in W_{Y_j}$, equivalently, for each $j$,  $\tau([X,Y_j])(w_j)=0$  $  \forall X \in \p_\h^+$. This, would imply the linear span of $\cup_j W_{Y_j}$ is $\mathcal V^{(1)}\cap \mathcal L_{W,H}^c \cap \mathcal U(\h_0)W$.

\subsection{ Some pairs $(\g,\h)$ so that   for every    not one dimensional representation $(\tau ,W)$,  the intersection   $\mathcal V^{(1)}\cap \mathcal U(\h_0)W \cap     \mathcal L_{W,H}^c$ is  a proper subspace of  $\mathcal V^{(1)}\cap\mathcal L_{W,H}^c$}
\phantom{xxxxxxxxxxxxxxxxxxxxxxxxxxxxxxxxxxxxxxxxxxxxxxxxxxxxxxxxxxxx}

 Whence, for these pairs we find non trivial "normal derivatives"  as well as non trivial "no normal derivatives" symmetry breaking operators. for details Section~\ref{do}.

 The pairs we consider are: $\g=\mathfrak{su}(n,1),   \h =\mathfrak s(\mathfrak u( 1 )\oplus \mathfrak{u}(n-1,1))$;

\noindent
$\g=\mathfrak{su}(n,1),   \h =\mathfrak s(\mathfrak u( n-1)\oplus \mathfrak{u}( 1,1))$.

\smallskip

 For any of the two pairs we have $[[\p_\h^+, \p_{\h_0}^-], \p_\h^+]=\{0 \} $, thus, for any {\it one dimensional representation} $(\tau,W)$ we have $\mathcal U(\h_0)W =  \mathcal L_{W,H}^c$, in consequence, for such a $\tau$,  every symmetry breaking operator is represented by normal derivative's differential operator. However, for non scalar representation $\tau$ the results are quite different.

\smallskip

Next, we analyze  a general $\tau$.

\smallskip
\begin{examp}\label{exa:example1} We  begin considering $\g=\mathfrak{su}(n,1),   \h =\mathfrak s(\mathfrak u( 1)\oplus \mathfrak{u}(n-1,1))$,   $\h_0 =\mathfrak{s}( \mathfrak{u}(1,1) \oplus \mathfrak u(n-1))$,
$\k=\mathfrak s(\mathfrak u(1)+\mathfrak u(n))$,
 $\p_\g^+=\sum_{1\leq j \leq n}\C E_{\epsilon_j -\delta}$, $\,\,\, \p_{\h_0}^+= \C E_{\epsilon_1 -\delta}$,  $\p_\h^+=\sum_{2\leq j\leq n}   \C E_{\epsilon_j -\delta}$, $[\  \p_{\h}^+,p_{\h_0}^-]= \sum_{2\leq j \leq n}\C E_{-( \epsilon_1 -\epsilon_j) }$.

 $\,\, \mathcal U(\h_0)W= \{\sum_{0\leq k \leq N} a_k (L_{E_{-(\epsilon_1 -\delta)}}^\tau)^k(w) , w \in W, a_k \in \C, \, N\geq 0\}$

 $=\{ \sum_{0\leq k \leq N} a_k [(E_{-(\epsilon_1 -\delta)}^k \otimes w] , w \in W, a_k \in \C, \, N\geq 0\}$. \\ Thus, $\,\, \mathcal U(\h_0)W \cap \mathcal V^{(1)}= \{ [E_{-(\epsilon_1 -\delta)} \otimes w] , w \in W \}$. \\ We recall $\mathcal L_{W,H}^c=\{ D \in \mathcal U(\g) \otimes_{\mathcal U(\k_\C +\p_\g^+)} W: L_X^\tau (D)=0 \,\, \forall X \in \p_\h^+ \}$. Now,  for $j\geq 2$, we have  $ L_{E_{\epsilon_j -\delta}}^\tau [E_{-(\epsilon_1 -\delta)} \otimes w] =\tau([E_{\epsilon _j -\delta}, E_{-(\epsilon_1 -\delta)}](w)=\tau(E_{\epsilon_j -\epsilon_1})(w) $, and since $E_{ \epsilon_j -\delta}, j\geq 2$ is a linear basis for $\p_\h^+$, we have

 $\mathcal U(\h_0)W\cap \mathcal L_{W,H}^c\cap \mathcal V^{(1)} \\ \phantom{xxxxxxxxxxxxxxx}=\{ [E_{\epsilon_1-\delta} \otimes w]  : w\in \cap_{2\leq i\leq n} Ker(\tau(E_{\epsilon_i -\epsilon_1}))\}$.

To follow,  we show $\mathcal U(\h_0)W\cap \mathcal L_{W,H}^c\cap \mathcal V^{(1)}$  is an irreducible representation of $L$. The proof follows the proof of \ref{sub:w0irred}.

\begin{lem} \label{prop:engel} We assume
   $\tau$ is not one dimensional representation of $\k\equiv \mathfrak{u}(n)$. Then, the subspace $ \cap_{2\leq i\leq n} Ker(\tau(E_{\epsilon_i -\epsilon_1})) $ is a proper subspace of $W$. Moreover,    $ \cap_{2\leq i\leq n} Ker(\tau(E_{\epsilon_i -\epsilon_1})) $ is the $L$-irreducible representation of lowest weight vector equal to the $\k$-lowest weight vector for $W$.
\end{lem}

\begin{proof}
    We have $\tau_{\vert_{\mathfrak{su}(n)}}$ is a faithful representation, thus, for every $j$,  $\tau(E_{\epsilon_j -\epsilon_1})$ is not the zero operator and algebraic groups gives it is a nilpotent matrix. Also, the linear span of $E_{\epsilon_j -\epsilon_1}, j\geq 2$ is a abelian Lie algebra. Thus, Engel's theorem yields the subspace $ \cap_{2\leq i\leq n} Ker(\tau(E_{\epsilon_i -\epsilon_1})) $ is a proper subspace of $W$.

 Next, we show $ \cap_{2\leq i\leq n} Ker(\tau(E_{\epsilon_i -\epsilon_1})) $ is the $L$-irreducible representation of lowest weight vector equal to the $\k$-lowest weight vector. In fact, $-\Psi_c=\{ \epsilon_i -\epsilon_1, i\geq 2\}\cup \{ \epsilon_r -\epsilon_s,   r >s \geq 2 \}$. Thus, the $\k$-lowest weight vector of $\tau$ belongs to  $ \cap_{2\leq i\leq n} Ker(\tau(E_{\epsilon_i -\epsilon_1})) $ and the description of $-\Psi_c$ gives  any $L$-lowest weight vector in  $ \cap_{2\leq i\leq n} Ker(\tau(E_{\epsilon_i -\epsilon_1})) $ is a $\k$-lowest weight vector. Therefore,  $ \cap_{2\leq i\leq n} Ker(\tau(E_{\epsilon_i -\epsilon_1})) $ contains, up to a constant, a unique $L$-lowest weight vector and we have  $ \cap_{2\leq i\leq n} Ker(\tau(E_{\epsilon_i -\epsilon_1})) $ is a irreducible $L_{ss}$-module. The linear operator $\tau(H_{\epsilon_1 -\delta})$ commutes with $\l$, and the fact that the restriction from $U(n)$ to $U(n-1)$   is multiplicity free, we have that $\tau(H_{\epsilon_1 -\delta})$ leaves invariant the subspace  $ \cap_{2\leq i\leq n} Ker(\tau(E_{\epsilon_i -\epsilon_1})) $. Thus, $\tau(H_{\epsilon_1 -\delta})$ acts by constant times the identity in the intersection. Since, the lowest $\k$-weight vector $w_0 \not= 0$ lies in   $ \cap_{2\leq i\leq n} Ker(\tau(E_{\epsilon_i -\epsilon_1})) $, we have the constant is $(w_K(\lambda +\rho_n-\rho_c), \epsilon_1 -\delta)>0$, whence,  we obtain the intersection is a irreducible $L$-module.  \end{proof} Whence,  $\mathcal U(\h_0)W\cap  (\mathcal L_{W,H}^c  \cap \mathcal V^{(1)}$  is a proper subspace of both $ \mathcal U(\h_0)W\cap \mathcal V^{(1)}$, $  \mathcal L_{W,H}^c  \cap \mathcal V^{(1)}$,  and the  description of

 $\mathcal U(\h_0)W\cap \mathcal L_{W,H}^c\cap \mathcal V^{(1)}  \equiv V_{w_L(w_K (\lambda +\rho_n -\rho_c^K))+\rho_c^L}^L$.

 We also point out that $\mathcal L_{W,H}^c\cap \mathcal V^{(1)} =\{ [E_{-(\epsilon_1 -\delta)}\otimes w ]: w\in W\}\equiv res_L(\tau)$, and  $\mathcal U(\h_0)W\cap \mathcal V^{(1)} =\{L_{E_{-\epsilon_1 +\delta}}^\tau (w): w\in W\}\equiv res_L(\tau)$. A equivalence of
 $\mathcal L_{W,H}^c\cap \mathcal V^{(1)}$ with $\mathcal U(\h_0)W\cap \mathcal V^{(1)}$ is given by the map $D$ in \ref{sub:D1iso}\cite{OV4}, that in this case becomes $ [E_{-(\epsilon_1 -\delta)}\otimes w ] \mapsto L_{E_{-\epsilon_1 +\delta}}^\tau (w), w \in W$.  Finally,   $\dim (\mathcal L_{W,H}^c\cap \mathcal V^{(1)})/(\mathcal U(\h_0)W\cap \mathcal L_{W,H}^c\cap \mathcal V^{(1)})   =\dim(\p_{\h_0}^+ \otimes W)- \dim V_{w_L(w_K (\lambda +\rho_n -\rho_c^K))+\rho_c^L}^L   = \dim V_{\lambda +\rho_n  }^K -  \dim V_{w_L(w_K (\lambda +\rho_n -\rho_c^K))+\rho_c^L}^L$.
This concludes the first example \end{examp}.

\begin{examp}\label{exa:example2} Next, we analyze   $\g=\mathfrak{su}(n,1), \h =\mathfrak{s}( \mathfrak{u}(1,1) \oplus \mathfrak u(n-1))  $,   $\h_0 =\mathfrak s(\mathfrak u( 1)\oplus \mathfrak{u}(n-1,1))  $, $\p_\g^+=\sum_{1\leq j \leq n}\C E_{\epsilon_j -\delta}$, $\,\,\, \p_{\h}^+= \C E_{\epsilon_1 -\delta}$,

  $\p_{\h_0}^+ =\sum_{2\leq j\leq n}   \C E_{\epsilon_j -\delta}$.  $[p_{\h}^+, p_{\h_0}^- ]=\sum_{j\geq 2} \C E_{\epsilon_1 -\epsilon_j}$.  \ref{sub:vn} implies \\
 $\,\, \mathcal U(\h_0)W\cap \mathcal V^{(1)}= \{\sum_{2\leq t \leq n} c_t L_{ E_{-\epsilon_t +\delta}}^\tau(w), w \in W , c_t \in \C \}$

  $=\{ \sum_{2\leq t \leq n} [E_{-\epsilon_t +\delta}\otimes w_t ], w_t \in W   \}$.

In this case,  $\mathcal L_{W,H}^c=\{ D \in \mathcal U(\g) \otimes_{\mathcal U(\k_\C +\p_\g^+)} W: L_{E_{\epsilon_1 -\delta}}^\tau (D)=0  \}$. Now,  for $j\geq 2$, we have  $ L_{E_{\epsilon_1 -\delta}}^\tau [E_{-(\epsilon_t -\delta)} \otimes w_t] =\tau([E_{\epsilon _1 -\delta}, E_{-(\epsilon_t -\delta)}])(w_t)=\tau(E_{\epsilon_1 -\epsilon_t)})(w_t) $, and   we have

 $\mathcal U(\h_0)W\cap \mathcal L_{W,H}^c\cap \mathcal V^{(1)} \\ \phantom{xxxxxxxxxx}=\{ \sum_{2\leq t \leq n}[E_{\epsilon_t-\delta} \otimes w_t]  :   \sum_{2\leq i\leq n} \tau(E_{\epsilon_1 -\epsilon_t}))(w_t) =0\}$.

Hence,
 a way  to choose   $w_t  $, is to  consider the span  of the  matrices $E_{\epsilon_1-\epsilon_t}, t=2, \dots,n$, it is a abelian subalgebra consisting of nilpotent matrices. Since $\tau$ is not scalar representation, as in Example 1, we apply Engel's Theorem to $\tau(E_{\epsilon_1-\epsilon_t}), t=2, \dots,n$ and obtain a proper, $L$-invariant, subspace $W_0$ of $W$ so that $\tau(E_{\epsilon_1-\epsilon_t})(w)=0, t=2, \dots,n$ if and only $w \in W_0$.  It readily follows that $W_0$
 is a isotypic component of $res_{U(n-1)}(\tau)$ of lowest weight equal to the one of $\tau$. Thus, $\{\sum_{2\leq t \leq n}[E_{\epsilon_t-\delta} \otimes w_t]  : w_t \in W_0 \}$ is contained in $\mathcal U(\h_0)W\cap \mathcal L_{W,H}^c\cap \mathcal V^{(1)}$.  When we   proceed as in \ref{exa:example1}, we obtain that $\sum_t \, [\C E_{\delta-\epsilon_t } \otimes W_{E_{\epsilon_1-\epsilon_t }}]$ is contained in $\mathcal U(\h_0)W\cap \mathcal L_{W,H}^c\cap \mathcal V^{(1)}$. Here, $ W_{E_{\epsilon_1-\epsilon_t }}=ker(\tau(E_{\epsilon_1-\epsilon_t }))$

 We would hope that  the linear span of   $(Ad\otimes \tau) (L)\sum_t \, [\C E_{\delta-\epsilon_t } \otimes W_{E_{\epsilon_1-\epsilon_t }}]$ is  equal to $\mathcal U(\h_0)W\cap \mathcal L_{W,H}^c\cap \mathcal V^{(1)}$.\end{examp}

 \subsection{Analysis of the equality $  \mathcal V^{(1)}\cap \mathcal L_{W,H}^c =  \mathcal V^{(1)}\cap \mathcal U(\h_0)W  $ } \smallskip
 A way to analyze the equality $\mathcal U(\h_0)W=\mathcal L_{W,H}^c$ is to understand the equality  $  \mathcal V^{(1)}\cap \mathcal L_{W,H}^c = \mathcal V^{(1)}\cap \mathcal U(\h_0)W  $ and the triple bracket  $[[\p_{\h_0}^+, \p_{\h}^-],\p_{\h_0}^+]$, two results on the topic are Proposition~\ref{sub:7}, Proposition~\ref{prop:equal}. A consequence of the analysis of $  \mathcal V^{(1)}\cap \mathcal L_{W,H}^c $ is that shields light on the structure of the symmetry breaking operators represented by  first order differential operators.
   We recall $\mathcal L_{W,H}^c=\{ D \in \mathcal U(\g) \otimes_{\mathcal U(\k+\p^+)} W : L_X^\tau(D)=0 \, \forall X \in \p_\h^+\}$.  Before we state the next result we would like to point out that in \cite[Lemma 3]{Va} it is shown that the equality $\mathcal L_{W,H}^c  =   \mathcal U(\h_0)W  $ is {\em equivalent} to the equality $\mathcal L_{W,H_0}^c  =   \mathcal U(\h)W  $. This results also follows from the proof of the next Proposition as we will verify.

  \begin{prop}\label{sub:7} We recall  our hypothesis   $H/L \rightarrow G/K$ is a holomorphic immersion, hence,  $(L_\cdot^\tau ,\mathcal O(\mathcal D,W)\cap L_\tau^2(G,W))$ is $H$-admissible holomorphic representation. Then,

  a) For a simple Lie algebra $\g$ so that the semisimple factor of $\k$ is a simple Lie algebra, equivalently,  $\g\ncong \mathfrak{su}(m,n) , n\geq 2, m\geq 2$,  the following equivalence holds:

\medskip

  $  \mathcal V^{(1)}\cap \mathcal L_{W,H}^c  =  \mathcal V^{(1)}\cap \mathcal U(\h_0)W  $ if and only if $(\tau,W)$ is a one dimensional representation.

\smallskip

 b) For $\g= \mathfrak{su}(m,n),  m\geq 2, n\geq 2 $, we write $\k=\mathfrak z_K +\mathfrak{su}(m)+\mathfrak{su}(n)$, and $\tau =\pi^{\mathfrak z_K} \otimes \pi^{\mathfrak{su}(m) } \otimes \pi^{\mathfrak{su}(n)}$  as a tensor product of irreducible representations. We denote by $\pi_0^S$  the trivial irreducible representation of the group $S$.

We have four cases:

 b1) For $\g= \mathfrak{su}(n,n)$, $\h=\mathfrak{sp}(n,\R)$, $\h_0=\mathfrak{so}^\star (2n) $; or $\h=\mathfrak{so}^\star (2n) $ $\h_0=\mathfrak{sp}(n,\R)$ the equivalence in a) holds.

 b2) When $\g= \mathfrak{su}(m,n)$, $\h=\mathfrak{su}(k,l)+\mathfrak{su}(m-k,n-l)+\mathfrak u(1)$, $1\leq k <m, 1\leq l <n$ the equivalence in a) holds.

b3) If $\g= \mathfrak{su}(m,n)$, $\h=\mathfrak{su}(m,l)+\mathfrak{su}(n-l)+\mathfrak u(1),  1\leq l <n$, we obtain:

 The equality $  \mathcal V^{(1)}\cap \mathcal L_{W,H}^c  =  \mathcal V^{(1)}\cap \mathcal U(\h_0)W  $ is true if and only if $\tau = \pi^{\mathfrak z_K} \otimes \pi_0^{\mathfrak{su}(m) } \otimes \pi^{\mathfrak{su}(n)}$.

 b4)  For $\g= \mathfrak{su}(m,n)$, $\h=\mathfrak{su}(l)+ \mathfrak{su}(m-l,n)+ \mathfrak u(1)$, $1\leq l <m$, the equality  $  \mathcal V^{(1)}\cap \mathcal L_{W,H}^c  =  \mathcal V^{(1)}\cap \mathcal U(\h_0)W  $ holds if and only if $\tau= \pi^{\mathfrak z_K} \otimes \pi^{\mathfrak{su}(m) } \otimes \pi_0^{\mathfrak{su}(n)}$.

\end{prop}

 \begin{proof} To begin with, we recall a Theorem of Harish-Chandra that shows: a $K$-type of a Discrete Series representation never is the trivial representation of $K$. We point out that the lowest $K$-type of a holomorphic Discrete Series restricted to the center of $K$ is always different to the trivial representation. In fact, let $\lambda$ denotes the Harish-Chandra parameter of $V_\tau$, then the highest weight of $(\tau, W)$ is $\lambda+\rho_n -\rho_K$. Since, the system of positive roots defined by $\lambda$ is holomorphic, we have that $\rho_n$ is orthogonal to every compact root and that dual of the center of $Lie(K)$  is spanned by $i\rho_n$. Thus $(\lambda +\rho_n -\rho_K, \rho_n)= (\lambda +\rho_n, \rho_n)>0$. Hence, we may conclude that $Ker(\tau):=Ker( \dot{\tau}) $ is always an ideal contained in $\k_{ss}$. When $Ker(\tau)$ is equal to
  the semisimple factor of $\k$ we call $L_.^\tau$, {\it a scalar Discrete Series}.  We analyze the five statements at the same time till  we split up and consider each particular case.   \ref{sub:vn} yields $ \mathcal V^{(1)}\cap \mathcal U(\h_0)W =lin.span\{ L_a^\tau (w): a\in \p_{\h_0}^- , w \in W\}$. The definition of $\mathcal L_{W,H}^c$ gives  $  \mathcal V^{(1)}\cap \mathcal L_{W,H}^c   =\{ \sum_j L_{b_j}^\tau (w_j): L_x^\tau(  \sum_j L_{b_j}^\tau (w_j))=0, \, \text{with}\, b_j \in \p^- , w_j \in W, \forall x \in \p_{\h}^+ \}= \mathcal V^{(1)}\cap \mathcal P(\p_{\h_0}^+,W)$.
Since $\mathcal V^{(1)}$ is isomorphic to $\p^- \otimes W$ via the map $L_x^\tau(w)\leftarrowtail x\otimes w$, we have $\dim \mathcal V^{(1)}\cap \mathcal U(\h_0)W = \dim \p_{\h_0}^+ \dim W $, also $\dim  \mathcal V^{(1)}\cap \mathcal L_{W,H}^c   = \dim \p_{\h_0}^+ \dim W$.

 The equality of sets   $  \mathcal V^{(1)}\cap \mathcal L_{W,H}^c   = \mathcal V^{(1)}\cap \mathcal U(\h_0)W $ implies $\forall y \in \p_{\h}^+, \forall x \in \p_{\h_0}^-, \forall w\in W$ the equality $ L_y^\tau(L_x^\tau w)=0$. Since, $L_x^\tau L_y^\tau w -L_y^\tau L_x^\tau w =L_{[x,y]}^\tau w =\tau([x,y])w$, we obtain
 $0=-L_x^\tau (L_y^\tau w)+\tau([y,x])w$. Also, for $y\in \p^+, w \in W$ we have $L_y^\tau (w)=0$. Thus, the subspace $[\p_\h^+,\p_{\h_0}^-]$ is contained in kernel of $\tau$. We note  that the involution $\sigma$ acts by minus one in  $[\p_\h^+,\p_{\h_0}^-]$, and by hypothesis the involution $\sigma$ acts by plus one on $\z_\k$,  $\k=\z_\k +\k_{ss}$.  Thus, the ideal of $\k$  spanned by $[\p_\h^+,\p_{\h_0}^-]$  is contained in the kernel of $\tau$ and  in $\k_{ss}$.

 Next, we verify $[\p_\h^+,\p_{\h_0}^-]$ is always a nonzero subspace. For this, we recall a result valid for any symmetric riemannian pair $(\g,\k)$, that is,  $\k=[\p,\p]$, (this is so, owing to $[\p,\p]+\p$ is a ideal in $\g$). Our hypothesis yields  $\p^+=\p_\h^+ +\p_{\h_0}^+, [\p^+,\p^+]=0$, hence $\k=[\p_\h^+,\p_\h^-]+[\p_{\h_0}^+,\p_{\h_0}^-]+[\p_{\h_0}^+,\p_{\h}^-]
 +[\p_{\h_0}^-,\p_{\h}^+]$.  The first two summands are  contained in $\l$, whereas the last two summands are contained in $\q_\C \cap \k_\C$, hence, we have $\l=[\p_\h^+,\p_\h^-]+[\p_{\h_0}^+,\p_{\h_0}^-]$, $\q_\C \cap \k_\C=[\p_{\h_0}^+,\p_{\h}^-]
+[\p_{\h_0}^-,\p_{\h}^+]$. Also,  $\overline{[\p_{\h_0}^+,\p_{\h}^-]}
 =[\p_{\h_0}^-,\p_{\h}^+]$. Whence, if $[\p_\h^+,\p_{\h_0}^-]$ were the zero subspace, we would have $\k=\l$ and $G=H$, an absurd. Putting together,  we have: \\ if  $ \mathcal V^{(1)}\cap \mathcal L_{W,H}^c =  \mathcal V^{(1)}\cap \mathcal U(\h_0)W$,  then  $\{0\}\subsetneq [\p_\h^+,\p_{\h_0}^-]
\subseteq Ker(\tau) \subseteq \k_{ss}$ and we conclude

\noindent
     {\it whenever   $ \mathcal V^{(1)}\cap \mathcal L_{W,H}^c  =  \mathcal V^{(1)}\cap \mathcal  U(\h_0)W$  and the semisimple factor of $\k$ is a simple Lie algebra, we have $ \tau $ is one dimensional representation.}

\smallskip
 After a computation based on  \cite[Table 3]{Va} \cite[Table 3 ]{OV3} or \cite[Table 3]{KOadv} and $\mathfrak{so}(2,4)\equiv \mathfrak{su}(2,2) $ we find: for every $\g\ncong \mathfrak{su}(m,n), m\geq 2, n\geq 2$, as well as for  $\g=\mathfrak{su}(p,1), p\geq 2$, the semisimple factor of  $\k$ is a simple Lie algebra, therefore,  the direct implication in a) follows.

   For the converse statement in a),   in notation of the previous paragraph we have for $ y\in \p_{\h}^+,   x \in \p_{\h_0}^-,   w\in W$ that  $L_y^\tau (L_x^\tau w)=L_x^\tau(L_y^\tau w)+\tau([y,x])w$, now $L_y^\tau (w)=0$ and $[y,x]\in \k_{ss}$,  and since either $\tau$ is one dimensional representation or $[\p_{\h}^+,  \p_{\h_0}^-] \subset Ker(\tau)$, we have $\tau([y,x])=0$, hence, we obtain $ \mathcal V^{(1)}\cap \mathcal U(\h_0)W \subset   \mathcal V^{(1)}\cap \mathcal L_{W,H}^c$. Since both spaces are equidimensional  the   equality and  a) follows.


\medskip

 We are left to  verify $b1), b2), b3), b4)$.  Except for $m=1$ or $n=1$,    $\k_{ss}$ is the sum of two simple ideals. In \ref{sub:b)} we compute that in case $b1), b2)$,  $[\p_\h^+,\p_{\h_0}^-]$  nontrivially intersects both simple ideals and that in cases $b3), b4)$ $[\p_\h^+,\p_{\h_0}^-]$ is contained in one simple ideal. Thus, in case $b1), b2)$ the ideal spanned by $[\p_\h^+,\p_{\h_0}^-]$ is equal to $\k_{ss}$, whereas in cases $ b3), b4)$ the ideal spanned by  $[\p_\h^+,\p_{\h_0}^-]$ is equal to one simple factor of $\k$. Therefore, the verification of $ b1),b2)$ is the same as the one we have presented to show $a)$.

 For $b3), b4)$. If the equality $ \mathcal V^{(1)}\cap \mathcal L_{W,H}^c  =  \mathcal V^{(1)}\cap \mathcal  U(\h_0)W$  holds, then, as in the proof of a),  $Ker(\tau)$ contains the ideal spanned by $[\p_\h^+,\p_{\h_0}^-]$, hence, $Ker(\tau)$ is either equal to one simple factor of $\k_{ss}$ or equal to $\k_{ss}$, thus, the shape of the  structure of  $\tau$ is as claimed in $b3), b4)$. In \ref{sub:b)} we compute the ideal that contains $[\p_\h^+,\p_{\h_0}^-]$ and the precise result follows.  The converse statement readily follows. \end{proof}

 \subsubsection{Computation for b)}\label{sub:b)} Here,  $\g=\mathfrak{su}(m,n), m\geq 2, n\geq 2$.
 We compute that in cases $b1), b2)$   $[p_\h^+,p_{\h_0}^-]$   intersects both factors $\mathfrak{su}(m)_\mathbb C, \mathfrak{su}(n)_\mathbb C$. In case $b3)$,  $[p_\h^+,p_{\h_0}^-]$  is    contained in  $  \mathfrak{su}(n)_\mathbb C$. In case $b4)$,  $[p_\h^+,p_{\h_0}^-]$  is   contained in  $  \mathfrak{su}(m)_\mathbb C$.

  We fix as Cartan subalgebra   $\t$ of $\mathfrak {su}(m,n)$ the set of diagonal matrices in $\mathfrak {su}(m,n).$   For certain orthogonal basis $\epsilon_1, \dots ,\epsilon_m, \delta_1, \dots, \delta_n$ of the dual vector space to the subspace of diagonal matrices in $\mathfrak{gl}(m+n, \mathbb C),$ we may, and will choose  as compact positive roots $ \{ \epsilon_r - \epsilon_s, \delta_p -\delta_q, 1 \leq  r < s \leq m, 1 \leq p < q \leq n \},$ the set of noncompact positive  roots is $ \Psi_n= \{  (\epsilon_r - \delta_q) \}.$ We may and will fix root vectors $\{E_\alpha : \alpha \in \Phi(\g,\t)\}$ such that $[E_\alpha, E_\beta]=E_{\alpha +\beta}$.

  We have, $\h=\mathfrak{su}(k,l)+  \mathfrak{su}(m-k,n-l)+\z_\l$, $\h_0= \mathfrak{su}(k,n-l)+  \mathfrak{su}(m-k,l) +\z_\l$.

  $\Psi_n \cap \Phi(\h,\t)=\{ \epsilon_i -\delta_j, \epsilon_a -\delta_b, 1\leq i\leq k, 1\leq j \leq l, k+1 \leq a \leq m, l+1 \leq b \leq n \}$.

  $\Psi_n \cap \Phi(\h_0,\t)=\{ \epsilon_p -\delta_q, \epsilon_r -\delta_s, 1\leq p\leq k,   l+1 \leq q \leq n , k+1 \leq r \leq m,  1 \leq s \leq l \}$.

 We compute $[E_{\beta_1} ,E_{-\beta_2}]=E_{\beta_1 -\beta_2},$ with $\beta_1 \in \Psi_n \cap \Phi(\h,\t), \beta_2 \in \Psi_n \cap \Phi(\h_0,\t)$

$$ (\epsilon_i -\delta_j)- \left\{ \begin{array}{lllll}
 ( \epsilon_p -\delta_q)= -\delta_j +\delta_q & \text{if}\, i=p, &      [E_{\beta_1} ,E_{-\beta_2}]=E_{-\delta_j +\delta_q} \\
 (\epsilon_r -\delta_s)=-\epsilon_i +\epsilon_r & \text{if}\, j=s, &    [E_{\beta_1} ,E_{-\beta_2}]= E_{-\epsilon_i +\epsilon_r}
\end{array}
\right. $$
Hence, $[p_{\mathfrak{su}(k,l)}^+,p_{{\mathfrak{su}(k,n-l)}}^-] =
\sum_{1\leq j\leq l, l+1\leq q \leq n} \C E_{-\delta_j +\delta_q} \subset  \mathfrak{su}(n)_\C  $,

$[p_{\mathfrak{su}(k,l)}^+,p_{{\mathfrak{su}(m-k,l)}}^-] =
\sum_{1\leq i\leq k, k+1\leq r \leq m} \C E_{-\epsilon_i +\epsilon_r} \subset
\mathfrak{su}(m)_\C $.

Similarly,  $[p_{\mathfrak{su}(m-k,n-l)}^+,p_{{\mathfrak{su}(k,n-l)}}^-] =\sum \C E_{-\epsilon_b +\epsilon_a} \subset  \mathfrak{su}(m)_\C  $,

$[p_{\mathfrak{su}(m-k,n-l)}^+,p_{{\mathfrak{su}(m-k,l)}}^-] =\sum \C E_{-\delta_b +\delta_s} \subset  \mathfrak{su}(n)_\C $.

\medskip
 For $b3)$, $\h=   \mathfrak{su}(m,l)+ \mathfrak{su}(n-l)+\z_\l$ $\h_0= \mathfrak{su}(m,n-l)+  \mathfrak{su}(l)+\z_\l$.   $\Psi_n \cap \Phi(\h,\t)=\{ \epsilon_i -\delta_j,  1\leq i\leq m, 1\leq j \leq l  \}$.

  $\Psi_n \cap \Phi(\h_0,\t)=\{ \epsilon_p -\delta_q,     1 \leq p \leq m , l+1 \leq q \leq n \}$.

 We compute $[E_{\beta_1} ,E_{-\beta_2}]=E_{\beta_1 -\beta_2},$ with $\beta_1 \in \Psi_n \cap \Phi(\h,\t), \beta_2 \in \Psi_n \cap \Phi(\h_0,\t)$

$ (\epsilon_i -\delta_j)-
( \epsilon_p -\delta_q)= -\delta_j +\delta_q, \,\, \text{if}\, i=p,
[E_{\beta_1} ,E_{-\beta_2}]=E_{-\delta_j +\delta_q}\in \mathfrak{su}(n)_\C  $.
Thus, $[\p_\h^+ , \p_{\h_0}^- ] \subset \mathfrak{su}(n)_\C $,  equal to $\cap (\q_\C \cap \k_\C)^-$,
and $Ker(\tau)$ contains
the ideal $\mathfrak{su}(n)_\C$.

\medskip

 For $b4)$, $\h=  \mathfrak{su}(l)+ \mathfrak{su}(m-l,n)+  \z_\l$ $\h_0= \mathfrak{su}(l,n)+  \mathfrak{su}(m-l)+\z_\l$

$\Psi_n \cap \Phi(\h,\t)=\{ \epsilon_i -\delta_j, l+1\leq i\leq m, 1\leq j \leq n,  \}$.

$\Psi_n \cap \Phi(\h_0,\t)=\{ \epsilon_p -\delta_q,  1\leq p\leq l,   1 \leq q \leq n   \}$.

 We compute $[E_{\beta_1} ,E_{-\beta_2}]=E_{\beta_1 -\beta_2},$ with $\beta_1 \in \Psi_n \cap \Phi(\h,\t), \beta_2 \in \Psi_n \cap \Phi(\h_0,\t)$

$  (\epsilon_i -\delta_j)-
( \epsilon_p -\delta_q)=
\epsilon_i -\epsilon_p \,\, \text{if}\, j=q,     [E_{\beta_1} ,E_{-\beta_2}]=
E_{\epsilon_i -\epsilon_p}\in \mathfrak{su}(m)_\mathbb C$. Thus,
$[\p_\h^+ , \p_{\h_0}^- ] \subset \mathfrak{su}(m)_\C  $,  equal to $(\q_\C \cap \k_\C)^-$,
and  $Ker(\tau)$ contains
the ideal $\mathfrak{su}(m)_\C$.
\medskip

\subsubsection{}\label{sub:forb1}  For $b1)$. The simple roots for one of the holomorphic system of positive roots $\Psi$ in  $\Phi(\mathfrak{su}(n,n), \t)$ are $\alpha_1, \dots, \alpha_{n-1}, \alpha_n, \alpha_{n+1},\dots, \alpha_{2n-1}$, here $\alpha_1, \alpha_{2n-1}$ are the roots that correspond to the end points of the Dynkin diagram, $\alpha_n$ is a noncompact root and $\alpha_i, i\not= n$ is a compact root. The involutive outer automorphism $\sigma$ of $\mathfrak{su}(n,n)$ that we consider  is so that $\sigma(\alpha_k)=\alpha_{2n-k}, k=1,\dots, 2n-1 $. Thus, $\sigma(\alpha_n)=\alpha_n$.

  Moreover, there exists a set   of root vectors $E_\alpha$ so that for all  $\alpha \in \Phi(\mathfrak{su}(n,n), \t),  \sigma(E_\alpha)=E_{\sigma(\alpha)}$.

   Let $\u_\mathbb C=\t_\mathbb C^\sigma$. Whence,  $\h_\C:=\mathfrak{sp}(2n,\R)_\C =\u_\mathbb C +\sum_{\{\alpha :\alpha\not= \sigma(\alpha)\}} \C (E_\alpha +E_{\sigma(\alpha)})+ \sum_{\{\alpha :\alpha= \sigma(\alpha)\} } \C E_\alpha$ and $\h_{0\C}:=\mathfrak{so}^\star(2n)_\C=\u_\mathbb C +\sum_{\{\alpha :\alpha\not= \sigma(\alpha)\}} \C (E_\alpha -E_{\sigma(\alpha)})$. The noncompact positive roots $\Psi_n$ are $\beta_{ij} :=\alpha_i +\dots +\alpha_n +\dots +\alpha_j$ with $1\leq i \leq n, n\leq j \leq 2n-1$. Thus, $\Psi_n{(\h, \u)}=\{\frac12(\alpha +\sigma(\alpha)): \alpha \not=\sigma(\alpha), \alpha \in \Psi_n\} \cup \{\alpha: \alpha = \sigma(\alpha), \alpha \in \Psi_n\}$, $\Psi_n{(\h_0,\u)}=\{\frac12(\alpha +\sigma(\alpha)): \alpha \not=\sigma(\alpha), \alpha \in  \Psi_n\}$. The positive roots for each simple factor of $\k_{ss}$ are:  $\Psi(\mathfrak{su}(n), \t_1)=\{ \alpha_i +\cdots +\alpha_j : 1\leq i \leq j \leq n-1\}$, $\Psi(\mathfrak{su}(n), \t_2)=\{ \alpha_i +\cdots +\alpha_j : n+1\leq i \leq j \leq 2n-1\}$.  For further use, we write the simple roots for both $\Psi (\h, \u), \Psi (\h_0, \u)$. Respectively, they are: $\{ \frac12 (\alpha_1 +\sigma(\alpha_1)), \cdots, \frac12(\alpha_{n-1} +\sigma(\alpha_{n-1})),\alpha_n\}$,  $\{ \frac12 (\alpha_1 +\sigma(\alpha_1)), \cdots, \frac12(\alpha_{n-1} +\sigma(\alpha_{n-1})), \frac12(\alpha_{n-1} +\sigma(\alpha_{n-1})) +\alpha_n\}$.

 We compute for $\frac12 (\beta_{ij}+\sigma(\beta_{ij})) \in \Psi_n (\h, \u)$, $\frac12 (\beta_{ab}+\sigma(\beta_{ab})) \in   \Psi_n (\h_0, \u)$, $[E_{ \frac12 (\beta_{ij}+\sigma(\beta_{ij})) }, E_{-\frac12 (\beta_{ab}+\sigma(\beta_{ab}))}]$, for $i=a, j<b$ we obtain $\frac12 (\beta_{ij}+\sigma(\beta_{ij}))-\frac12 (\beta_{ib}+\sigma(\beta_{ib}))=\alpha_{j+1}+\cdots +\alpha_b \in \Psi(\mathfrak{su}(n), \t_2)$, whereas for $i=a $ and $j>b$ we obtain a negative root in $\Phi(\mathfrak{su}(n), \t_2)$, and for $j=b, i\leq a$ we obtain $\frac12 (\beta_{ij}+\sigma(\beta_{ij}))-\frac12 (\beta_{aj}+\sigma(\beta_{aj}))=\alpha_{i}+\cdots +\alpha_{a-1} \in \Psi(\mathfrak{su}(n), \t_1)$, meanwhile for $i>a$ we obtain a negative root in $\Phi(\mathfrak{su}(n), \t_1)$. Therefore, $[\p_\h^+,\p_{\h_0}^-]$ has a nonzero intersection with each simple factor for the complexification of $ \k_{ss}$ as well as   $[\p_\h^+,\p_{\h_0}^-]$ intersects nontrivially to both $(\q_\C \cap \k_\C)^\pm$. Whence,  $[\p_\h^+,\p_{\h_0}^-]=(\q_\C \cap \k_\C)$.  Another proof of the previous equality is: the representation of $\l$ in $\k_\C/\l_\C$ is irreducible, and  $[\p_\h^+,\p_{\h_0}^-]$ is a non zero $L$-invariant subspace. Thus, $[\p_\h^+,\p_{\h_0}^-]=(\q_\C \cap \k_\C)$. The same computation gives the same conclusion for $\g=\mathfrak{su}(n,n)$, $\h= \mathfrak{so}^\star(2n  )$, $\h_0= \mathfrak{sp} (n,\R)$.   Now,
we have concluded  the computation for Proposition~\ref{sub:7}.

\begin{cor} For a scalar holomorphic Discrete Series
 the equality $ \mathcal V^{(1)}\cap \mathcal U(\h_0)W =   \mathcal V^{(1)}\cap \mathcal L_{W,H}^c$
is always true.
\end{cor}
\begin{cor}For arbitrary $(\tau, W)$. The equality $ \mathcal V^{(1)}\cap \mathcal U(\h_0)W =
   \mathcal V^{(1)}\cap \mathcal L_{W,H}^c$ holds if and only if $[\p_\h^+,\p_{\h_0}^-]$ contains
$Ker(\tau)$.
\end{cor}
\begin{cor}The equality $\mathcal L_{W,H}^c  =   \mathcal U(\h_0)W  $ is {\em equivalent}
to the equality $\mathcal L_{W,H_0}^c  =   \mathcal U(\h)W  $.
\end{cor}

 Indeed, a inspection to the statement of the Proposition plus recalling that for $\g=\mathfrak{su}(m,n)$, $\h=\mathfrak{su}(k,l)+\mathfrak{su}(m-k,n-l)+\mathfrak u(1)$   we have $\h_0=\mathfrak{su}(k,n-l)+\mathfrak{su}(m-k,l)+\mathfrak u(1)$
yields the proof.

\smallskip

  \subsubsection{ } We continue to analyze when every symmetry breaking operator is a normal derivative differential operator.  In \cite[Proposition 3]{Va} it is shown for a {\it scalar holomorphic discrete series}

 \phantom{xxxxxxxxx}$\mathcal L_{W,H}^c =\mathcal U(\h_0)W$ if and only if $[[\p_{\h_0}^+, \p_{\h}^-],\p_{\h_0}^+]=\{0\}$.

 In \ref{sub:triplepairszero},  we find the list of symmetric   pairs $(\g,\h)$ that satisfies the condition $[[ \p_{\h_0}^-,\p_\h^+ ],\p_{\h_0}^-]=\overline{[[ \p_{\h_0}^+, \p_\h^- ],
\p_{\h_0}^+]}=\{0\}$.  Next, following the proof of \cite[Proposition 3]{Va}, we verify a necessary and sufficient condition so that after we assume every first order symmetry breaking operator is represented by a normal differential operator, then every symmetry breaking operator is represented by a normal differential operator.
\begin{prop}\label{prop:equal} Assume $(\tau, W)$ is so that the equality $  \mathcal V^{(1)}\cap \mathcal L_{W,H}^c  =
\mathcal V^{(1)}\cap \mathcal U(\h_0)W    $ holds, (see
Proposition~\ref{sub:7}).
  Then, \\
\phantom{xxxxxxxxx}   $\mathcal L_{W,H}^c =\mathcal U(\h_0)W $ if and only if
$[[ \p_{\h_0}^+,\p_\h^- ],\p_{\h_0}^+]=\{0\}$.
\end{prop}
\begin{proof}To begin with, under our hypothesis on $(\tau, W)$,  we show:

a)  If \, $[[\p_\h^+,
\mathfrak p_{\mathfrak h_0}^-], \p_{\h_0}^-]\not= \{0\},$ then
$\mathcal L_{W,H}^c     \not= \mathcal U(\mathfrak h_0) (W).$

b) If \, $[[\p_\h^+, \p_{\h_0}], \p_{\h_0}^-]=\{0\},$ then
$\mathcal L_{W,H}^c =\mathcal U(\mathfrak h_0) (W).$

In the course of the proof,  we will recall the details and notation
we are in need.
Since,  \ref{sub:D1iso}, \cite{OV2}\cite{OV3} $\mathcal L_{W.H}^c $ is $L $-isomorphic to $\mathcal U(\h_0)W $
    and both $L$-modules are $L$-admissible, the inclusion   $ \mathcal U(\h_0)W \subseteq  \mathcal L_{W,H}^c  $  implies the equality.  For this proof, we write
$\pi(y):=L_y^\tau, y\in \mathcal U(\g)$

We recall the subspace of $K$-finite vectors in our holomorphic Discrete Series  is equal to $\mathcal U(\g) \otimes_{\mathcal U(\t_\C \oplus \p^+)}   W =L_{(\mathcal U(\p^-)}^\tau(W)$
and is isomorphic to  $S(\p^-)\otimes W$.   $[D\otimes w]$ denotes the equivalence class in
$\mathcal U(\g) \otimes_{\mathcal U(\t_\C \oplus \p^+)}   W $ for $D\otimes w$.  Then
$\mathcal L_{W,H}^c= \{ v \in  \mathcal U(\g) \otimes_{\mathcal U(\t_\C \oplus
\p^+)} W : \pi(Y)v=0 \,\forall Y \, \in \, \p_\h^+ \}$. After we fix a ordered
basis $\{Y_1, \cdots, Y_q\}$, when $\u=\t$ of root vectors, in $\p_\g^+$
so that $\{Y_1, \cdots, Y_p\}$ belongs  to $\p^+ \cap \h_0 =\p_{\h_0}^+$ and
  $\{Y_{p+1}, \cdots, Y_q\}$ belongs to $\p^+ \cap \h =\p_{\h}^+$  and a basis $\{w_1,\cdots,w_c \}$
for $W$, we have that a linear basis over $\C$ for
$\mathcal U(\g) \otimes_{\mathcal U(\t_\C \oplus \p^+)} W $  is
    $\{ [Y_1^{a_1} \dots Y_q^{a_q} \otimes w_i] : i=1\dots c, a_j \in \Z_{\geq 0} \}  $
and a basis for $\mathcal U(\h_0)W $ is
    $\{[ Y_1^{a_1} \dots Y_p^{a_p} \otimes w_i] : i=1\dots c, a_j \in \Z_{\geq 0} \}  $.

To follow we show: \\
when     $[[\p_\h^+,
\mathfrak p_{\mathfrak h_0}^-], \p_{\h_0}^-]\not= \{0\},$ then
$ \mathcal U(\mathfrak h_0) (W) \cap \mathcal V^{(2)}$
is not contained in $\mathcal L_{W,H}^c $. More precisely, we construct a
second order element in $ \mathcal U(\mathfrak h_0) (W)$
and not in $\mathcal L_{W,H}^c $.

We   choose $X_{\h} \in  \p_{\h}^+, \, Y_1, Y_2, \in
  \p_{\h_0}^-$  ,
so that $[[X_{\h}, Y_{1}]  , Y_{2}] $  is nonzero. We fix nonzero $w \in W$, we have:
 \begin{equation*} \begin{split} \pi (X_{\h})[Y_1 Y_2 \otimes w]   &=  [X_{\h}Y_1Y_2 \otimes w]
  \\ & =
 [Y_1X_{\h} Y_2 + [Y_1,X]Y_2 \otimes w ]
\\ & = [ Y_1Y_2X_{\h} \otimes w + Y_1[X_{\h} ,Y_2]\otimes w \\ & \quad \quad +Y_2[Y_1,X_{\h}]\otimes w
+[[Y_1,X_{\h}],Y_2] \otimes w ]
\\ & = [Y_1Y_2\otimes \tau(X_{\h})w + Y_1 \otimes \tau([X_{\h},Y_2])w \\ &
\quad \quad + Y_2\otimes  \tau([Y_1,X_{\h}])w +
[[Y_1,X_{\h}],Y_2] \otimes w] \end{split}
\end{equation*}
  The first summand is zero because $X_{\h} \in \p^+, $
the second and third summand is zero because  our hypothesis is $Ker(\tau)$ contains
$[\p_\h^+ , \p_{\h_0}^-]$ and
  $[X_{\h}, Y_{1}]
\in [\p_\h^+ , \p_{\h_0}^-] $.
Now,  $ [[[X_{\h}, Y_{1}]  , Y_{2}]\otimes w] = \pi ([[X_{\h}, Y_{1}]  ,
Y_{2}]) ([1\otimes w])$
 and
             $[[[X_{\h}, Y_{1}]  , Y_{2}] \in \mathfrak   \p^- $, next, we recall
              the fact that,  for a nonzero $Y \in \p^-$, $\pi_\lambda (Y)$
is injective in $ \mathcal U(\g) \otimes_{\mathcal U(\t_\C \oplus \p^+)} W $, hence,
  $ \pi (X_{\h})[Y_{1}
Y_{2} \otimes w]$ is nonzero and we have shown the claim and we have
concluded the proof of a).

To follow, we show b).       We verify, by induction on $a_1 +\dots +a_p,$
the equality \begin{equation*}\pi_\lambda(X_{\h} ) [Y_{1}^{a_1} \cdots Y_{p}^{a_p}
\otimes w_i]=0 \,\,\text{for\,all}\,\,X_{\h} \in \h \cap \p^+. \end{equation*}
Since, $\pi$ is a  holomorphic representation of lowest $K$-type $(\tau,W)$
we have $\pi(X_{\h})[1\otimes w_i]=0, \forall \,\,X_\h
\in \p^+ $,  whence, the case  $a_1 +\dots +a_p =0$ is consequence of our hypothesis.

In general, for $ a_1= \dots =a_{j-1}=0, a_j \geq 1,
\text{and}\, 1\leq  j\leq p$ we have
\begin{equation*} \begin{split} \pi (X_{\h})[Y_{j}^{a_j}
\dots Y_{p}^{a_p} \otimes w_i] & = [X_{\h}Y_{j}^{a_j}\dots Y_{p}^{a_p} \otimes w_i]  \\ & =[Y_{j}X_{\h}Y_{j}^{a_j -1}Y_{j+1}^{a_{j+1}}
\dots Y_{p}^{a_p} \otimes w_i \\ & \quad \quad  + [X_{\h}, Y_{j}] Y_{j}^{a_j -1}Y_{j+1}^{a_{j+1}}
\dots Y_{\beta_p}^{a_p} \otimes w_i] \\   & =  \pi_\lambda (Y_{j})\pi_\lambda(X_{\h})
[Y_{j}^{a_j -1} Y_{j+1}^{a_{j+1}} \dots Y_{p}^{a_p} \otimes w_i]   \\ &  \quad \quad + [  Y_{j}^{a_j -1}Y_{j+1}^{a_{j+1}}
  \dots Y_{p}^{a_p} [X_{\h}, Y_{j}] \otimes w_i]\\   & =  \pi_\lambda (Y_{j})\pi_\lambda(X_{\h})
[Y_{j}^{a_j -1} Y_{j+1}^{a_{j+1}} \dots Y_{p}^{a_p} \otimes w_i]   \\ &   \quad \quad + [  Y_{j}^{a_j -1}Y_{j+1}^{a_{j+1}}
\dots Y_{\beta_p}^{a_p} \otimes \tau([X_\h,Y_j])w_i]  \\ & =0. \end{split}
\end{equation*}

The first summand is equal to zero owing to the inductive hypothesis,
the second summand is equal to zero owing to our  the hypothesis
states that for $1\leq j \leq p, 1 \leq k \leq p$, $[X_{\h}, Y_{j}]$ commutes with
$Y_{k}$   and
that $[X_{\h}, Y_{j}]$ belongs to $Ker(\tau).$
Thus, we have concluded the proof of b) and thereby shown
$\mathcal U(\h_0)W \subseteq \mathcal L_{W,H}^c$, the comment at the beginning of
this proof concludes the proof of the Proposition. \end{proof}

\subsubsection{Sufficient condition so that the subspace $\mathcal V^{(1)}\cap \mathcal L_{W,H}^c \cap \mathcal U(\h_0)W$ is non trivial}

The following Proposition has as a consequence, the existence of normal derivative symmetry breaking operators, as we will verify in Section~\ref{do}
\begin{prop}\label{prop:nu1nonzero} We assume $\g$ is simple Lie algebra and $(\tau, W)$ is a nonscalar   representation. Then, it holds $\mathcal V^{(1)}\cap \mathcal L_{W,H}^c \cap \mathcal U(\h_0)W) \not= \{0\}$ for the triples  in  Table A below, and in  Table B we list    triples  where it is not the case.
\end{prop}
\begin{proof} This is done by proving it
case by case,
finishing at Lemma 4.13. We follow quite closely many of the main steps in Proposition
\ref{prop:exisnormal}. To begin with, we choose nonzero root vectors   $E_{-\gamma},   \gamma \in \Psi_n(\h_0)$   and, $E_\beta, \beta \in  \Psi_n(\h)$.
Whence,  $\p_{\h }^+=\sum_{\beta \in \Psi_n(\h) } \C E_{ \beta}$, then, \ref{sub:vn} let us write

\noindent
   $\mathcal V^{(1)}\cap \mathcal U(\h_0)W =\{\sum_{\gamma \in \Psi_n(\h_0)}   L_{E_{-\gamma}}^\tau w_{\gamma}= \sum_{\gamma \in \Psi_n(\h_0)}   E_{-\gamma}\otimes w_{\gamma}  :        w_{\gamma} \in W \}$.

For $X\in \p_{\h}^+, Y \in \p_{\h_0}^-, w \in W$,  we have $L_X^\tau ([Y\otimes w])=[XY\otimes w] =\tau([X,Y])w$, whence

$\mathcal L_{W,H}^c \cap \mathcal V^{(1)}\cap  \mathcal U(\h_0)W =\{\sum_{\gamma \in \Psi_n(\h_0)}   L_{E_{-\gamma}}^\tau  ( w_{\gamma})=\sum_{\gamma \in \Psi_n(\h_0)}  [E_{-\gamma} \otimes w ]: \sum_{\gamma \in \Psi_n(\h_0)}   \tau[X, E_{-\gamma}] ( w_{\gamma})=0 \, \forall X\in \p_{\h}^+\}. $

Since, the vectors $\{ E_\beta : \beta \in  \Psi_n(\h)\}$ span $\p_\h^+ $, we obtain\\
\phantom{xxxxxxx} $\sum_{\gamma \in \Psi_n(\h_0)}  E_{-\gamma} \otimes w_\gamma  \in   \mathcal L_{W,H}^c$ if and only if
$\forall \beta \in \Psi_n(\h)$\\ \phantom{xxxx}it holds $ \sum_{\gamma \in \Psi_n(\h_0)}   \tau[E_\beta, E_{-\gamma}] ( w_{\gamma})= \sum_{\gamma \in \Psi_n(\h_0)}   \tau (E_{\beta-\gamma}) ( w_{\gamma})=0.$

Now, in a {\it case by case checking}, using the hypothesis $\g$ is a simple Lie algebra,  see \ref{sub:comp1}, we construct a non empty subset $\Sigma \subset \Psi_n(\h_0)$ so that the set $\{ E_{\beta -\gamma} : \beta \in \Psi_n(\h), \gamma \in \Sigma \}$ spans a non   trivial nilpotent Lie subalgebra of $\q_\C \cap \k_\C$. Since, $\g$ is a simple Lie algebra and $H/L \rightarrow G/K$ is a holomorphic embedding  in Section 2 we find a proof that $\q_\C \cap \k_\C \subset \k_{ss}$. Since $\tau $ is a non scalar representation   algebraic representation, we apply
Engel's Theorem to the nilpotent algebra $\tau(\{ E_{\beta -\gamma} : \beta \in \Psi_n(\h), \gamma \in \Sigma \})$. Thus, we obtain nonzero vector $w_+\in W $ so that $\tau(E_{\beta -\gamma})(w_+)=0 \, \forall \, \beta \in \Psi_n(\h), \gamma \in \Sigma $. Next, we fix $w_\gamma =w_+, for  \, \gamma \in \Sigma$, and, $w_\gamma =0 $ for $\gamma \in \Psi_n(\h_0), \gamma \notin \Sigma$. To continue, we verify $Y:=\sum_{\gamma \in \Psi_n(\h_0)} [E_{-\gamma} \otimes w_\gamma]$ belongs to $\mathcal L_{W,H}^c \cap \mathcal V^{(1)}\cap  \mathcal U(\h_0)W$. For $\beta \in \Psi_n(\h)$, our choice gives $L_{E_\beta}^\tau(Y)= \sum_{\gamma \in \Sigma} \tau( E_{\beta -\gamma})(w_+) =0, $ and we have shown the claim. Thus, the proof of Proposition~\ref{prop:nu1nonzero} is concluded module the verification of Table A and Table B. For this, the necessary computation   is carried out in the following subsections.  \end{proof}

\begin{center}
Table A presents the triples $(\g, \h, \h_0)$ so that there exists a non empty $\Sigma \subset \Psi_n(\h_0)$  so that the subalgebra spanned by $\{E_{\beta -\gamma} : \beta \in \Psi_n(\h), \gamma \in \Sigma\}$ is  nilpotent.
\end{center}
\begin{center}
Table A
\begin{tabular}{|c| c | c|}
\hline   $\g$  &   $\h \, $   &  $\h_0 \,   $     \\
 \hline   $\mathfrak{su}(m,n), k\leq m, l \leq n$   &     $\mathfrak{su}(k,l)+\mathfrak{su}(m-k,n-l)+ \u(1)$ & $\mathfrak{su}(k,n-l)+\mathfrak{su}(m-k,l)+ \u(1)$    \\
  \hline   $\mathfrak{su}(n,n) $   &     $\mathfrak{sp}(n.\R) $ & $\mathfrak{so}^\star(2n) $    \\
   \hline   $\mathfrak{su}(n,n) $   &     $\mathfrak{so}^\star(2n)$ & $\mathfrak{sp}(n.\R) $    \\
\hline   $\mathfrak{so}(2,2n+1)$ $0\leq k<n$ & $\mathfrak{so}(2,2k+1)+ \mathfrak{so}(2n-2k)$ & $\mathfrak{so}(2,2n-2k)+ \mathfrak{so}(2k+1)$       \\
\hline   $\mathfrak{so}(2,2n+1)$ $1\leq k<n$ & $\mathfrak{so}(2,2k)+ \mathfrak{so}(2n-2k+1)$ & $\mathfrak{so}(2,2n+1-2k)+ \mathfrak{so}(2k)$       \\
\hline   $\mathfrak{sp}(n, \mathbb R)$  & $\mathfrak{u}(m,n-m)$ & $\mathfrak{sp}(m, \mathbb R)+ \mathfrak{sp}(n-m, \mathbb R)$       \\
 \hline   $\mathfrak{so}(2,2n)$ $1\leq k \leq n-1$  & $\mathfrak{so}(2,2k)+ \mathfrak{so}(2n-2k)$   & $\mathfrak{so}(2,2n-2k)+ \mathfrak{so}(2k)$      \\
 \hline   $\mathfrak{so}(2,2n)$   & $\mathfrak{u}(1,n) $   & $\mathfrak{u}(1,n) $      \\
\hline   $\mathfrak{so}^\star (2n)$  & $\mathfrak{u}(m,n-m)$ & $\mathfrak{so}^\star(2m)+ \mathfrak{so}^\star(2n-2m)$       \\
\hline   $\mathfrak{so}^\star (2n)$  & $ \mathfrak{so}^\star(2m)+ \mathfrak{so}^\star(2n-2m)$ & $\mathfrak{u}(m,n-m) $       \\
\hline $\e_{6(-14)}$ & $\mathfrak{so}(2,8)+\mathfrak{so}(2) $ & $\mathfrak{so}(2,8)+\mathfrak{so}(2)$ \\
\hline $\e_{6(-14)}$ & $\mathfrak{su}(2,4)+\mathfrak{su}(2) $ & $\mathfrak{su}(2,4)+\mathfrak{su}(2)$ \\
\hline $\e_{6(-14)}$ & $\mathfrak{so}^\star(10)+\mathfrak{so}(2) $ & $\mathfrak{su}(5,1)+\mathfrak{sl}(2, \mathbb R)$ \\
\hline $\e_{6(-14)}$ & $ \mathfrak{su}(5,1)+\mathfrak{sl}(2, \mathbb R) $ & $\mathfrak{so}^\star(10)+\mathfrak{so}(2) $ \\
\hline $\e_{7(-25)}$ $  $   & $\mathfrak{so}^\star(12)+\mathfrak{su}(2) $ & $\mathfrak{su}(6,2)$ \\
\hline $\e_{7(-25)}$ $ $   & $ \mathfrak{su}(6,2) $ & $\mathfrak{so}^\star(12)+\mathfrak{su}(2)  $ \\
\hline $\e_{7(-25)}$ $  $ & $\mathfrak{so}(2,10)+\mathfrak{sl}(2, \mathbb R)   $ & $\e_{6(-14)} + \mathfrak{so}(2)$ \\
\hline $\e_{7(-25)}$ $  $ & $\e_{6(-14)} + \mathfrak{so}(2)    $ & $\mathfrak{so}(2,10)+\mathfrak{sl}(2, \mathbb R) $ \\
\hline   $\mathfrak{su}(n,n)$  & $\mathfrak{so}^\star(2n)$ &$\mathfrak{sp}(n,\mathbb R)$       \\
\hline   $\mathfrak{su}(n,n)$  & $ \mathfrak{sp}(n,\mathbb R)$ &$\mathfrak{so}^\star(2n) $       \\
\hline   $\mathfrak{so}(2,2n)$ $0\leq k\leq n-2$ & $\mathfrak{so}(2,2k+1)+ \mathfrak{so}(2n-2k-1)$   & $\mathfrak{so}(2,2n-2k-1)+ \mathfrak{so}(2k+1)$      \\
\hline
\end{tabular} \\
\end{center}
\begin{center}
The following table present the triples $(\g, \h, \h_0)$, for which, there does NOT  exist  a non empty $\Sigma \subset \Psi_n(\h_0)$  such that the subalgebra spanned by $\{E_{\beta -\gamma} : \beta \in \Psi_n(\h), \gamma \in \Sigma\}$ is  nilpotent.
\end{center}
\begin{center}
Table B
\begin{tabular}{|c| c | c|}
\hline   $\g$  &   $\h \, $   &  $\h_0 \,   $     \\
\hline   $\mathfrak{so}(2,2n+1)$    & $\mathfrak{so}(2,2n)+ \mathfrak{so}(1)$ & $\mathfrak{so}(2,1)+ \mathfrak{so}(2n)\,\,\S$       \\
\hline   $\mathfrak{sp}(n, \mathbb R)$  & $ \mathfrak{sp}(m, \mathbb R)+ \mathfrak{sp}(n-m, \mathbb R)$ & $\mathfrak{u}(m,n-m) \,\,\,  \natural $    \\
\hline   $\mathfrak{so}(2,2n)$  & $\mathfrak{so}(2,2(n-1)+1)+ \mathfrak{so}(1  )$   & $\mathfrak{so}(2,1)+ \mathfrak{so}(2n-1)\,\,\S$      \\
\hline
\end{tabular} \\
The mark $\S$ points that for $k\geq 2$ always happens $\mathcal V^{(k)}\cap  \mathcal U(\h_0)W \cap    \mathcal L_{W,H}^c=\{0\}$ and for  $\tau$ not a scalar representations,          $\mathcal V^{(1)}\cap  \mathcal U(\h_0)W \cap    \mathcal L_{W,H}^c=\{0\}$.

The sign $\natural$ points that for a scalar $\tau$, and $m=1$ \cite{KP2} has shown: $k\geq 2$ forces $\mathcal V^{(k)}\cap  \mathcal U(\h_0)W \cap    \mathcal L_{W,H}^c=\{0\}$. We have not been able to achieve a conclusion for $m : n-2 \geq m>1$.
\end{center}
\subsubsection{Construction of $\Sigma$}\label{sub:comp1}When $[[ \p_{\h_0}^-,\p_\h^+ ],\p_{\h_0}^-]=\{0\}$, we have verified that either $[ \p_{\h_0}^+,\p_\h^- ]=(\q_\C \cap \k_\C)^+$ or $[ \p_{\h_0}^+,\p_\h^- ]=(\q_\C \cap \k_\C)^-$, thus, for this cases, we fix $\Sigma =\Psi_n(\h_0)$, which yields $\{E_{\beta -\gamma} : \beta \in \Psi_n(\h), \gamma \in \Sigma\}$ is contained in a nilpotent Lie algebra.   Thus, for all pairs $(\g,\h)$  satisfying $[[ \p_{\h_0}^-,\p_\h^+ ],\p_{\h_0}^-]=\{0\}$, carrying out a computation as in    subsubsection~\ref{exa:example1}, we obtain the same conclusion as in \ref{exa:example1}  . That is,  for all pairs $(\g,\h)$
and $\tau$ not a scalar representations,      the intersection   $\mathcal V^{(1)}\cap  \mathcal U(\h_0)W \cap    \mathcal L_{W,H}^c$ is nonzero, contains either the $L$-irreducible spanned by the $\k$-highest weight vector or  the $L$-irreducible spanned by the $\k$-lowest weight vector. Applying Proposition~\ref{sub:7} we obtain $\mathcal V^{(1)}\cap  \mathcal U(\h_0)W \cap    \mathcal L_{W,H}^c$ is a  proper subspace of  $\mathcal V^{(1)}\cap\mathcal L_{W,H}^c$.

For the next paragraphs we assume $U=T$.

Another easy case to verify the existence of $\Sigma$ is when $U=T$ and either the maximal root or the noncompact simple root for $\Psi$ belongs to  $\Psi_n(\h_0)$,
then, the choice of  $\Sigma$ equal to the singleton defined by  either  the maximal or the simple root   satisfies the requirement, due that, $\{E_{\beta -\gamma} :\beta \in \Psi_n(\h), \gamma \in \Sigma\}$ is contained in one of the nilpotent Lie algebra's $(\q_\C \cap \k_\C)^\pm$.

When, $[[ \p_{\h_0}^-,\p_\h^+ ],\p_{\h_0}^-]\not=\{0\}$,  $\h=\h^1 \oplus \mathfrak c $, here, $\h^1$ is a noncompact simple Lie algebras, $ \mathfrak c$ is a compact Lie algebra,   $\g $ not equivalent to $\mathfrak{su}(m,n)$, $\h_0=\h_0^1 \oplus \h_0^2 \oplus \mathfrak c $, here, $\h_0^j$ are noncompact simple Lie algebras, $ \mathfrak c$ is a compact Lie algebra. We may write the disjoint union $\Psi_n(\h_0)=\Psi_n(\h_0^1)\cup \Psi_n(\h_0^2)$. We check case by case that either $\Sigma = \Psi_n(\h_0^1)$ or $\Sigma = \Psi_n(\h_0^2)$ satisfies our requirement.

\medskip

\medskip

\smallskip
The following particular example is illustrative.

Notation as in \ref{sub:b)}. $\g=\mathfrak{su}(m,n), 1\leq k <m, 1\leq l <n, \, \text{or}\, 1\leq k <m, n=1, l=0, \, \h= \mathfrak{su}(k,l)+\mathfrak{su}(m-k,n-l) +\u(1), \h_0 =\mathfrak{su}(k,n-l)+\mathfrak{su}(m-k, l)+\mathfrak u(1)$. We have checked $[\p_\h^+, \p_{\mathfrak{su}(k,n-l)}^-]= (\q_\C \cap \k_\C)^+$. Thus, $\Sigma =\Psi_n(\mathfrak{su}(k,n-l))$ is a choice.

 $\g=\mathfrak{so}(2,2n+1), \h=\mathfrak{so}(2,2n-1)+\mathfrak{so}(2), \h_0=\mathfrak{so}(2,2)+\mathfrak{so}(2n-1)$.

$\Psi_c=\{\delta_i \pm \delta_j :  1 \leq i<j\leq n\}\cup \{\delta_1, \cdots, \delta_n \}$, $\Psi_n=\{\epsilon \pm \delta_i :  1 \leq i \leq n \}$. $\Psi_n(\h)=\{\epsilon, \epsilon \pm \delta_i  :  1 \leq i\leq n-1 \}$, $\Psi_n(\h_0)=\{ \epsilon -\delta_n, \epsilon +\delta_n  \}$, $\Psi_c(\l)=\{\delta_i \pm \delta_j :  1 \leq i<j\leq n-1 \}\cup \{\delta_1, \cdots, \delta_{n-1} \}$, $\Psi ((\q_\C \cap \k_\C)^+ )= \{\delta_i \pm \delta_n :  1 \leq i \leq n-1, \delta_n \}$. We have two choices for $\Sigma$, they are:
$\Sigma_1= \{ \epsilon -\delta_n\}$ or $\Sigma_2= \{\epsilon +\delta_n  \}$,  because the set  $\{E_{\beta -\gamma} : \beta \in \Psi_n(\h), \gamma \in \Sigma\}$, for the first election is $\{ E_{\delta_n}, E_{\pm \delta_j +\delta_n} \}$, and, for the second is $\{ E_{-\delta_n}, E_{\pm \delta_j -\delta_n} \}$,  and each of them span an abelian nilpotent Lie algebra. We exchange roles $\h:=\h_0, \h_0:=\h$, then, at least, we have three choices for $\Sigma$, they are: $\{\epsilon \}$, $\{\epsilon +\delta_1\}, $ $\{\epsilon -\delta_1 \} $ as it is easily verified.

$\g=\mathfrak{so}(2,2n+1),  \h=\mathfrak{so}(2,2k+1)+\mathfrak{so}(2n-2k), \h_0=\mathfrak{so}(2,2n-2k)+\mathfrak{so}(2k+1)$.

 $1 \leq k<n$.  $\Psi_n=\{  \epsilon \pm \delta_i :  1 \leq i \leq n \}\cup \{\epsilon \}$. $\Psi_n(\h)=\{ \epsilon \pm \delta_i  :  1 \leq i\leq k \}\cup \{\epsilon \}$, $\Psi_n(\h_0)=\{ \epsilon \pm \delta_j, k+1 \leq j \leq n  \} $. One  choice  for $\Sigma$ is the set
$  \{ \epsilon +\delta_n \}$,  because the set  $\{E_{\beta -\gamma} : \beta \in \Psi_n(\h), \gamma \in \Sigma\}= \{E_{\delta_n \pm \delta_i} : 1\leq i \leq k \}$ spans an abelian  nilpotent Lie algebra, as a direct computation yields.

For $k=0$, $\Psi_n(\h)=\{\epsilon \}  $, $\Psi_n(\h_0)= \{ \epsilon \pm \delta_i  :  1 \leq i\leq n \} $. $\epsilon-(\epsilon \pm \delta_i )=\mp \delta_i$. One choice for $\Sigma=\{ \epsilon -\delta_i, i=1\cdots n\}$ then the subalgebra spanned by $\{E_{\beta -\gamma} : \beta \in \Psi_n(\h), \gamma \in \Sigma\}= \{  E_{\delta_i}, i=1\cdots n\}$ is a nilpotent Lie algebra.

\smallskip

$\g=\mathfrak{so}(2,2n+1),  \h=\mathfrak{so}(2,2k)+\mathfrak{so}(2n+1-2k), \h_0=\mathfrak{so}(2,2n+1-2k)+\mathfrak{so}(2k)$.

 $ 1 \leq k<n$,  $\Psi_n=\{  \epsilon \pm \delta_i :  1 \leq i \leq n \}\cup \{\epsilon \}$. $\Psi_n(\h)=\{ \epsilon \pm \delta_i  :  1 \leq i\leq k \} $, $\Psi_n(\h_0)=\{ \epsilon \pm \delta_j, k+1 \leq j \leq n  \}\cup \{\epsilon \}$. One  choice  for $\Sigma$ is the set
$  \{ \epsilon +\delta_n \}$,  because the set  $\{E_{\beta -\gamma} : \beta \in \Psi_n(\h), \gamma \in \Sigma\}= \{E_{\delta_n \pm \delta_i} : 1\leq i \leq k \}$ spans an abelian  nilpotent Lie algebra, as a direct computation yields.

\subsubsection{}\label{sub:interequalceronosigma}For $k=n$, $\g=\mathfrak{so}(2,2n+1),  \h=\mathfrak{so}(2,2n)+\mathfrak{so}(1), \h_0=\mathfrak{so}(2, 1 )+\mathfrak{so}(2n)$. $\Psi_n(\h)=\{ \epsilon \pm \delta_i  :  1 \leq i\leq n \} $, $\Psi_n(\h_0)= \{\epsilon \}$. $\epsilon \pm \delta_i-\epsilon=\pm \delta_i$, the unique possible choice for $\Sigma$ is $\{\epsilon \}$, for which, the subalgebra spanned by $\{E_{\beta -\gamma} : \beta \in \Psi_n(\h), \gamma \in \Sigma\}=\{ E_{\pm \delta_i} \}$ is not nilpotent. To follow, we verify $\mathcal V^{(1)}\cap \mathcal L_{W,H}^c \cap \mathcal U(\h_0)W=  \mathcal V^{(1)}\cap \mathcal L_{W, \mathfrak{so}(2,2n)+\mathfrak{so}(1) }^c \cap \mathcal U(\mathfrak{so}(2,1)+\mathfrak{so}(2n))W =\{0\}$. In fact $\p_{\h_0}^- \otimes W=\{ E_{-\epsilon} \otimes w, w\in W \}$ and $L_{E_{\epsilon \pm \delta_i}} (E_{-\epsilon}\otimes w)=\tau([E_{\epsilon \pm \delta_i}, E_{-\epsilon}](w)=0, \forall i=1,\dots,n \}$ implies $\tau(E_{\pm \delta_i})(w)=0 \forall i=1,\dots,n$,  hence $\tau(E_{\delta_j \pm \delta_i})(w)=0 \forall i,j$. Thus, $w$ is lowest and highest weight for $(\tau, W)$, since $\tau$ is not a scalar representation, this forces $w=0$, and we have verified $ \mathcal V^{(1)}\cap \mathcal L_{W, \mathfrak{so}(2,2n)+\mathfrak{so}(1) }^c \cap \mathcal U(\mathfrak{so}(2,1)+\mathfrak{so}(2n))W =\{0\}$. Compare with Example \ref{exa:example1}.

\smallskip

$\g=\mathfrak{sp}(n, \R), 1\leq m <n,  \h=\mathfrak u(m,n-m), \h_0=\mathfrak{sp}(m, \R)+\mathfrak{sp}(n-m, \R) $. $\Psi_n(\h)=\{ \epsilon_a +\epsilon_b, 1\leq a \leq m, m+1 \leq b \leq n \}$, $\Psi_n(\h_0)=\{ 2\epsilon_i, \epsilon_i +\epsilon_j, 1\leq i,j \leq \ m  \}\cup \{ 2\epsilon_p, \epsilon_p +\epsilon_q, m+1\leq p,q \leq  n  \} $.

It readily follows: $[\p_{\mathfrak u(m,n-m)}^+, \p_{\mathfrak{sp}(m, \R) }^-]=\sum_{ 1\leq a\leq m, m+1\leq b \leq n} \C E_{-(\epsilon_a -\epsilon_b)}$.

 $[\p_{\mathfrak u(m,n-m)}^+, \p_{\mathfrak{sp}(n-m, \R) }^-]=\sum_{ 1\leq a\leq m, m+1\leq b \leq n} \C E_{(\epsilon_a -\epsilon_b)}$.

Whence, at least, we have two choices for $\Sigma$, either $\{ \epsilon_i +\epsilon_j, 1\leq i,j \leq \ m  \}$ or $\{ \epsilon_p +\epsilon_q, m+1\leq p,q \leq  n  \}$.

  We exchange roles $\g=\mathfrak{sp}(n, \R)$, $\h:=\mathfrak{sp}(m, \R)+\mathfrak{sp}(n-m, \R) $, $\h_0:=\mathfrak u(m,n-m)$.
Next, we verify there is no $\Sigma$ owing to $[\p_\h^+, \p_{\h_0}^-]$ contains $E_{\pm(\epsilon_i - \epsilon_p)}, i\leq m, m+1 \leq p$.
 We would like to have: \begin{equation} \mathcal U( \mathfrak u(m,n-m))W \cap \mathcal L_{W,\mathfrak{sp}(m, \R)+\mathfrak{sp}(n-m, \R)} \cap \mathcal V^{(1)}=\{0\} \end{equation}

  $\g=\mathfrak{sp}(n+1,\R)$, $\h=\mathfrak{sp}(n, \R)+\mathfrak{sp}(1, \R)$ $\h_0=\mathfrak u(n,1)$. $(\tau, W)$ is a scalar representation, then, for $k\geq 2$, we believe,  it holds  \begin{equation}  \mathcal U( \mathfrak u(n,1))W \cap \mathcal L_{W,\mathfrak{sp}(n, \R)+\mathfrak{sp}(1, \R)} \cap \mathcal V^{(k)} =\{0\} \end{equation}

$\g=\mathfrak{so}(2,2n), 1 \leq k<n, \h=\mathfrak{so}(2,2k)+\mathfrak{so}(2n-2k), \h_0=\mathfrak{so}(2,2n-2k)+\mathfrak{so}(2k)$.

$\Psi_c=\{\delta_i \pm \delta_j :  1 \leq i<j\leq n\}$, $\Psi_n=\{\epsilon \pm \delta_i :  1 \leq i \leq n \}$. $\Psi_n(\h)=\{ \epsilon \pm \delta_i  :  1 \leq i\leq k \}$, $\Psi_n(\h_0)=\{ \epsilon \pm \delta_j, k+1 \leq j \leq n  \}$, $\Psi_c(\l)=\{\delta_i \pm \delta_j :  1 \leq i<j\leq k, \text{or}\, k+1 \leq i<j\leq n  \}$, $\Psi ((\q_\C \cap \k_\C)^+ )= \{\delta_i \pm \delta_j :  1 \leq i \leq k, k+1 \leq j \leq n \}$. One  choice  for $\Sigma$ is the set
$  \{ \epsilon -\delta_r, r\geq k+1 \}$,  because the set  $\{E_{\beta -\gamma} : \beta \in \Psi_n(\h), \gamma \in \Sigma\}= \{E_{\delta_r \pm \delta_i} : 1\leq i \leq k, k+1 \leq r \leq n \}$ spans a  nilpotent Lie algebra, as a direct computation yields.

\smallskip

$\g=\mathfrak{so}(2,2n),   \h=\mathfrak{u}(1,n) , \h_0=\mathfrak{u}(1,n) $.

  $\Psi_n=\{\epsilon \pm \delta_i :  1 \leq i \leq n \}$. $\Psi_n(\h)=\{ \epsilon - \delta_i  :  1 \leq i\leq n \}$, $\Psi_n(\h_0)=\{ \epsilon + \delta_j,  1 \leq j \leq n  \}$, 
   One  choice  for $\Sigma$ is the set
$  \{ \epsilon +\delta_n,   \}$,   as a direct computation yields.

$\g=\mathfrak{so}^\star(2n), 1\leq m <n,  \h=\mathfrak u(m,n-m), \h_0=\mathfrak{so}^\star(2m)+\mathfrak{so}^\star(2n-2m)$.
$\Psi_n(\h)=\{ \epsilon_r +\epsilon_s, 1\leq r \leq m, m+1 \leq s \leq n \}$

$\Psi_n(\h_0)=\{ \epsilon_i +\epsilon_j, 1\leq i \not= j \leq \ m  \}\cup \{ \epsilon_p +\epsilon_q, m+1\leq p \not= q \leq  n  \} $.

It readily follows: $[\p_\h^+, \p_{\mathfrak{so}^\star(2m)}^-]=\sum_{-\epsilon_i+\epsilon_b: 1\leq i\leq m, m+1\leq b \leq n} \C E_{-(\epsilon_i -\epsilon_b)}$.

 $[\p_\h^+, \p_{\mathfrak{so}^\star(2n-2m)}^-]=\sum_{\epsilon_a-\epsilon_p: 1\leq a\leq m, m+1\leq p\leq n} \C E_{(\epsilon_a -\epsilon_p)}$.

Whence, at least, we have two choices for $\Sigma$, either $\{ \epsilon_i +\epsilon_j, 1\leq i,j \leq \ m  \}$ or $\{ \epsilon_p +\epsilon_q, m+1\leq p,q \leq  n  \}$.

We exchange roles $\h:=\h_0$, $\h_0:=\h$, to follow, we    choose $\Sigma=\{\epsilon_m +\epsilon_n\}$, then $\Sigma$  yields   $\{E_{\beta - \gamma}: \beta \in \Psi_n(\mathfrak{so}^\star(2m)+\mathfrak{so}^\star(2n-2m)), \gamma \in \Sigma \}$ spans a nilpotent subalgebra.  In fact,    $\epsilon_i +\epsilon_j-(\epsilon_m +\epsilon_n)$ (resp. $ \epsilon_p +\epsilon_q-(\epsilon_m +\epsilon_n)$) is a root if and only if $i=m $ or $j=m$ (resp. $p=n $ or $p=n$), and the difference is $\epsilon_j-\epsilon_n$ (resp. $\epsilon_q-\epsilon_m$). Since, $1\leq j <m, m+1 \leq p <n$, the   root vectors associated  to differences span a commutative nilpotent algebra.

\subsubsection{}\label{sub:e6su24}$\g=\e_{6(-14)}, \h\equiv \mathfrak{su}(2,4) +\mathfrak{su}(2)$. $\h_0 \equiv \h$.
 Next, we pin down the positive roots for $\g, \h, \h_0$. The Vogan diagram for a holomorphic system for $\mathfrak e_{6(-14)}$ has fundamental roots (as in Bourbaki) $\alpha_1, \dots, \alpha_5, \alpha_6, $  the noncompact simple root is $\alpha_1$, the root $\alpha_2$ is  adjacent to both $\alpha_4$ and to the opposite of the maximal root. Maximal root is $\alpha_1 +2\alpha_3 +3 \alpha_4+2 \alpha_5+\alpha_6 +2\alpha_2$. Then $\Phi(\k,\t)= \{\pm \sum_j a_j \alpha_j \in \Phi(\mathfrak e_6,\mathfrak t) :    a_1=0\}$. We set $\breve{\alpha}=\alpha_3+ 2\alpha_4+2\alpha_5+\alpha_6+\alpha_2$. The simple roots for $\Psi \cap \Phi(\h,\t)$ are: (in the order for the Dynkin diagram), $\breve{\alpha}, \alpha_1, \alpha_3, \alpha_4, \alpha_2; \alpha_6$.

 The roots in $\Psi_n(\h,  \t)$ are:

$
\begin{smallmatrix}
    &   & 0 &   &   \\
  1& 1 & 1 & 1 & 0
\end{smallmatrix}
$
$\begin{smallmatrix}
   &  & 1 &   &  \\
   1 & 1 & 1& 1 & 0
 \end{smallmatrix}
 $
$\begin{smallmatrix}
   &  &  0 &   &  \\
   1 & 1 & 1 & 1 & 1
 \end{smallmatrix}
 $
 $\begin{smallmatrix}
   &  &  1 &   &  \\
   1 & 1 & 1 & 1 & 1
 \end{smallmatrix}
 $
 $\begin{smallmatrix}
   &  &  1 &   &  \\
   1 & 1 & 2 & 1 & 0
 \end{smallmatrix}
 $
 $\begin{smallmatrix}
   &  &  1 &   &  \\
   1 & 1& 2 & 1 & 1
 \end{smallmatrix}
 $
 $\begin{smallmatrix}
   &  &  1 &   &  \\
   1 & 2 & 2 & 1& 0
 \end{smallmatrix}
 $
 $\begin{smallmatrix}
   &  & 1&   &  \\
   1 & 1 & 2 & 1 & 1
 \end{smallmatrix}
 $

   $\Psi_n(\h_0 ,\t)$ consists of the eight roots

$
\begin{smallmatrix}
    &   & 0 &   &   \\
  1& 0 & 0 & 0 & 0
\end{smallmatrix}
$
$\begin{smallmatrix}
   &  & 0 &   &  \\
   1 & 1 & 0 & 0 & 0
 \end{smallmatrix}
 $
$\begin{smallmatrix}
   &  &  0 &   &  \\
   1 & 1 & 1 & 0 & 0
 \end{smallmatrix}
 $
 $\begin{smallmatrix}
   &  &  1 &   &  \\
   1 & 1 & 1 & 0 & 0
 \end{smallmatrix}
 $
 $\begin{smallmatrix}
   &  &  1 &   &  \\
   1 & 1 & 2 & 2 & 1
 \end{smallmatrix}
 $
 $\begin{smallmatrix}
   &  &  1 &   &  \\
   1 & 2 & 2 & 2 & 1
 \end{smallmatrix}
 $
 $\begin{smallmatrix}
   &  &  1 &   &  \\
   1 & 2 & 3 & 2 & 1
 \end{smallmatrix}
 $
 $\begin{smallmatrix}
   &  &  2 &   &  \\
   1 & 2 & 3 & 2 & 1
 \end{smallmatrix}
 $

 The maximal root and the noncompact simple root for $\Psi$ belongs to $\Psi_n(\h_0)$. Thus, at least we have two choices for $\Sigma \subset \Psi_n(\h_0)$.

 Next, we exchange roles and define $\h:=\h_0, \h_0:=\h$, then, it readily follows that $\Sigma =\{
\begin{smallmatrix}
    &   & 0 &   &   \\
  1& 1 & 1 & 1 & 0
\end{smallmatrix}\} $ works. That is, $\{E_{\beta -\gamma} : \beta \in \Psi_n(\h), \gamma \in \Sigma\}$ spans an abelian nilpotent Lie algebra. For further use, we point out that $\k/\l \equiv \mathfrak{so}(10)/\mathfrak{so}(4)+\mathfrak{so}(6)$, whence the representation of $L$ in $\k_\C/\l_\C$ is irreducible. Thus, each $(\q_\C \cap \k_\C)^\pm$ is not an $Ad(L)$-invariant subspaces and $[\p_\h^+,\p_{\h_0}^-]= \q_\C \cap \k_\C$.

\smallskip

 $\g=\mathfrak e_{6(-14)},  \h=\mathfrak{so}^\star(10)+\mathfrak {so}(2), $
     $  \h_0 = \mathfrak {su}(5,1) +\mathfrak{sl}_2(\R)  $. Next, we pin down the positive roots for $\g, \h, \h_0$.

  The Vogan diagram for a holomorphic system for $\mathfrak e_{6(-14)}$ has fundamental roots (as in Bourbaki) $\alpha_1, \dots, \alpha_5, \alpha_6, $  the noncompact simple root is $\alpha_1$, the root $\alpha_2$ is  adjacent to both $\alpha_4$ and to the opposite of the maximal root. Maximal root is $\alpha_1 +2\alpha_3 +3 \alpha_4+2 \alpha_5+\alpha_6 +2\alpha_2$.

\noindent
 Then   $\Psi_n(\h ,\t)=\{  \sum_j a_j \alpha_j \in \Phi(\mathfrak e_6,\mathfrak t) :  a_2 \in \{1,2\}, a_1=1\}$, $ \Psi_n(\h_0,  \t)=\{  \sum_j a_j \alpha_j \in \Phi(\mathfrak e_6,\mathfrak t) :  a_2=0, a_1=1 \}\cup \{\alpha_{max} \}$. The choice of $\Sigma=\{\alpha_{max} \}$    yields $\{E_{\beta -\gamma} : \beta \in \Psi_n(\h), \gamma \in \Sigma\}$ spans a nilpotent Lie algebra. We change role, $\h :=\h_0, \h_0 :=\h
 $. Then $\Sigma:=\{ \begin{smallmatrix}
    &   & 1 &   &   \\
  1& 1 & 1 & 0 & 0
\end{smallmatrix}\} $ is a good choice.

\smallskip

\smallskip

 $\g=\mathfrak e_{7(-25)},  \h=\mathfrak{so}^\star(12)+\mathfrak {su}(2), $
     $  \h_0 = \mathfrak {su}(2,6)   $. Next, we pin down the positive simple roots for $\g, \h, \h_0$.

  The Vogan diagram for a holomorphic system for $\mathfrak e_{7(-25)}$ has fundamental roots (as in Bourbaki) $\alpha_1, \dots, \alpha_5, \alpha_6,  \alpha_7$  the noncompact simple root is $\alpha_1$, the root $\alpha_2$ is  adjacent to  $\alpha_5$, the root $\alpha_7$ is  adjacent to the opposite of the maximal root. The maximal root is $\alpha_1 +2\alpha_3 +3 \alpha_4+4 \alpha_5+3\alpha_6 +2\alpha_7 +2\alpha_2$.

The simple roots for  $    \mathfrak {su}(2,6)   $ are: $ \alpha_{max}^{\e_6} := \begin{smallmatrix}
   & &   & 2 &   &   \\
  0&1& 2& 3 & 2 & 1
\end{smallmatrix} $, $\alpha_1$, $\alpha_3$, $\alpha_4$, $\alpha_5$, $\alpha_6$, $\alpha_7$.
Since, the  root $\alpha_1+\alpha_3+\alpha_4+\alpha_5+\alpha_6+\alpha_7 \in \Psi_n(\mathfrak {su}(2,6))$ and $\alpha_{max}^{\e_6} + (\alpha_1+\alpha_3+\alpha_4+\alpha_5+\alpha_6+\alpha_7)=\alpha_{max}^{\e_7}$, we obtain $\alpha_1, \alpha_{max}^{\e_7} \in \Psi_n(\mathfrak {su}(2,6))$.

The simple roots for  $    \mathfrak{so}^\star(12)+\mathfrak {su}(2)$ are:
  $ \begin{smallmatrix}
   & &   & 1 &   &   \\
  1&1& 1 & 1 & 0 & 0
\end{smallmatrix}  $,   $\alpha_3$, $\alpha_4$, $\alpha_5$, $\alpha_6$, $\alpha_7$; $ \begin{smallmatrix}
   & &   & 2 &   &   \\
  0&1& 2 & 3 & 2& 1
\end{smallmatrix} $, the non compact root $ \begin{smallmatrix}
   & &   & 1 &   &   \\
  1&1& 1 & 1 & 0 & 0
\end{smallmatrix}  $ is adjacent to $\alpha_6$ and orthogonal to the other compact simple roots.
\noindent
For the triple $\g=\mathfrak e_{7(-25)},  \h=\mathfrak{so}^\star(12)+\mathfrak {su}(2), $
     $  \h_0 = \mathfrak {su}(2,6)   $ we choose $\Sigma$ to be $\{\alpha_1\}$.

For the triple $\g=\mathfrak e_{7(-25)},  \h=\mathfrak {su}(2,6) , $
     $  \h_0 = \mathfrak{so}^\star(12)+\mathfrak {su}(2)    $ we choose $\Sigma$ to be the singleton $ \{\begin{smallmatrix}
   & &   & 1 &   &   \\
  1&1& 1 & 1 & 0 & 0
\end{smallmatrix} \} $.

  In both cases, the choice of  $\Sigma$   yields $\{E_{\beta -\gamma} : \beta \in \Psi_n(\h), \gamma \in \Sigma\}$ spans a nilpotent Lie algebra.

\smallskip

 $\g=\mathfrak e_{7(-25)},  \h=\mathfrak{so}(2,10)+\mathfrak {sl}_2(\Bbb R), $
     $  \h_0 = \mathfrak {e}_{6(-14)}+\mathfrak{so}(2)   $. Next, we write the positive simple roots for $\g, \h, \h_0$.

  The Vogan diagram for a holomorphic system for $\mathfrak e_{7(-25)}$ has fundamental roots (as in Bourbaki) $\alpha_1, \dots, \alpha_5, \alpha_6,  \alpha_7$  the noncompact simple root is $\alpha_1$, the root $\alpha_2$ is  adjacent to  $\alpha_5$, the root $\alpha_7$ is  adjacent to the opposite of the maximal root. The maximal root is $\alpha_1 +2\alpha_3 +3 \alpha_4+4 \alpha_5+3\alpha_6 +2\alpha_7 +2\alpha_2$.

The simple roots for  $    \mathfrak {so}(2,10) +\mathfrak{sl}_2(\Bbb R)  $ are:  $\alpha_1$, $\alpha_3$, $\alpha_4$, $\alpha_5$, $\alpha_6$, $\alpha_2; \alpha_{max}$.

The simple roots for  $   \h=\mathfrak{e}_{6(-14)}+\mathfrak{so}(2)  $ are:
  $\{ \gamma :=\begin{smallmatrix}
   & &   & 0 &   &   \\
  1&1& 1 & 1 & 1 & 1
\end{smallmatrix}\} $,   $\alpha_3$, $\alpha_4$, $\alpha_5$, $\alpha_6$, $\alpha_7$.


 Then   $\Psi_n(\h ,\t)=\{  \sum_j a_j \alpha_j \in \Phi(\mathfrak e_6,\mathfrak t) :  a_2 \in \{1,2\}, a_1=1\}$, $ \Psi_n(\h_0,  \t)=\{  \sum_j a_j \alpha_j \in \Phi(\mathfrak e_6,\mathfrak t) :  a_2=0, a_1=1 \}\cup \{\alpha_{max} \}$. The choice of $\Sigma=\{\alpha_{max} \}$    yields $\{E_{\beta -\gamma} : \beta \in \Psi_n(\h), \gamma \in \Sigma\}$ spans a nilpotent Lie algebra. We change role, $\h :=\h_0, \h_0 :=\h
 $. Then $\Sigma:=\{ \begin{smallmatrix}
    &   & 0 &   &   \\
  1& 1 & 1 & 1 & 1
\end{smallmatrix}\} $ is a choice.

\smallskip
To finish the proof, from   \cite{Va} \cite{Kob9} \cite{KOadv} we are left to consider the pairs when $U\not= T$. The algebras we have to consider are: $\g= \mathfrak{su}(n,n), \h=\mathfrak{so}^\star(2n), \h_0=\mathfrak{sp}(n,\mathbb R)$ and $\g=\mathfrak{so}(2,2n), \h=\mathfrak{so}(2,2k+1)+\mathfrak{so}(2n-2k-1), \h_0=\mathfrak{so}(2,2n-2k-1)+\mathfrak{so}(2k+1)$.

\smallskip
Notation for $\g= \mathfrak{su}(n,n), \h=\mathfrak{so}^\star(2n), \h_0=\mathfrak{sp}(n,\mathbb R)$ as in \ref{sub:forb1}. A choice for $\Sigma=\{\frac12(\alpha_{n-1}+\alpha_{n+1})+\alpha_{n}\}$. When we change roles, $\h:=\h_0, \h_0:=\h$, $\Sigma=\{\alpha_{n}\}$ works.

\smallskip
Notation as in \ref{prop:triplein}. $\g=\mathfrak{so}(2,2n), 0 \leq k\leq n-1, \h=\mathfrak{so}(2,2k+1)+\mathfrak{so}(2n-2k-1), \h_0=\mathfrak{so}(2,2n-2k-1)+\mathfrak{so}(2k+1)$.   $\Psi_n=\{\epsilon \pm \delta_i :  1 \leq i \leq n \}$. $\u^\star $ is the span of $\{ \epsilon , \delta_i, i=1\dots k \}\cup \{  \delta_i, i \geq  k+2 \}$.  $\Psi_n(\h)=\{ \epsilon \pm \delta_i  :  1 \leq i\leq k \}\cup \{(\epsilon + \delta_{k+1})_{\vert_\mathfrak u}=\epsilon \}$, $\Psi_n(\h_0)=\{ \epsilon \pm \delta_j, k+2 \leq j \leq n  \}\cup \{(\epsilon+ \delta_{k+1})_{\vert_\mathfrak u} \}$, $\Psi_c(\l)=\{\delta_i \pm \delta_j :  1 \leq i<j\leq k, \text{or}\, k+1 \leq i<j\leq n  \}\cup \{(\pm \delta_i + \delta_{k+1})_{\vert_\u}, (\pm \delta_{k+1} + \delta_{a})_{\vert_\u}, i\leq k, a\geq k+2\}$, $\Psi ((\q_\C \cap \k_\C)^+ )= \{\delta_i \pm \delta_j :  1 \leq i \leq k, k+2 \leq j \leq n \}$. For $k\leq n-2$, one  choice  for $\Sigma$ is the set
$  \{ \epsilon +\delta_r, n\geq r\geq k+2 \}$,  owing to the set  $\{E_{\beta -\gamma} : \beta \in \Psi_n(\h), \gamma \in \Sigma\}=\{ E_{-\delta_r}, E_{-\delta_r\pm \delta_i}, 1\leq i \leq k, k+2 \leq r \leq n \} $    spans a nontrivial  nilpotent Lie algebra, as a direct computation yields.
 For $k=n-1$, we have  $\Psi_n(\h_0)=\{(\epsilon+ \delta_{n})_{\vert_\mathfrak u}=\epsilon \}$
The unique choice for $\Sigma =\{\epsilon \}$, yields $\pm \delta_j \in \{ \beta -\gamma : ...\}$. Thus, the subalgebra spanned by  $\{E_{\beta -\epsilon} : \beta \in \Psi_n(\mathfrak{so}(2,2n-1) \}$ contains a copy of $sl_2$. Actually, we   have
\begin{equation}\label{eq:so2n21}\mathcal V^{(1)} \cap \mathcal U(\mathfrak{so}(2,1)+\mathfrak{so}( 2n-1))W \cap \mathcal L_{W, \mathfrak{so}(2,2n-1)+\mathfrak{so}(1)}  =\{0\}. \end{equation}
In fact, let $E_{-(\epsilon +\delta_n)_{\vert_\u}}^{\h_0} \otimes w$ an element of the intersection, we show $w=0$.  As before, this forces

 $\tau ([E_{(\epsilon +\delta_n)_{\vert_\u}}^{\h},E_{-(\epsilon +\delta_n)_{\vert_\u} }^{\h_0}])(w)=0$,

  $\tau ([E_{(\epsilon \pm \delta_i)}^\h ,E_{-(\epsilon +\delta_n)_{\vert_\u}}^{\h_0}])(w)=0, \text{for}\, 1\leq  i\leq n-1$.

  In this case, the outer automorphism $\sigma$  so that $\sigma(\delta_i)=\delta_i, i=1 \dots n-1, \sigma(\delta_n)=-\delta_n, \sigma(\epsilon)=\epsilon$ has an extension to $\mathfrak{so}(2+2n, \C)$ such that $\sigma(E_\alpha)=E_{\sigma(\alpha)}, \, \forall \alpha \in \Psi$, (for $\alpha$ simple it is Serre's Theorem, for any $\alpha$ Freudenthal ....)for a convenient choice of root vectors $E_\alpha$.  Then,
$E_{(\epsilon \pm \delta_i)}^\h=E_{(\epsilon \pm \delta_i)}, i=1\cdots n-1$, $E_{-(\epsilon +\delta_n)_{\vert_\u}}^{\h_0}=E_{-(\epsilon +\delta_n)}-E_{-(\epsilon -\delta_n)}$, $E_{(\epsilon \pm \delta_n)}^\h=E_{(\epsilon \pm \delta_n)}+ E_{(\epsilon \mp \delta_n)}$, and $[E_{(\epsilon + \delta_n)}+ E_{(\epsilon - \delta_n)}, E_{-(\epsilon +\delta_n)}-E_{-(\epsilon -\delta_n)}]= b(E_{(\epsilon + \delta_n)},  E_{-(\epsilon +\delta_n)})H_{(\epsilon + \delta_n)}-b(E_{(\epsilon - \delta_n)},E_{-(\epsilon - \delta_n)})H_{( -2\delta_n)}$,

now, $b(E_{(\epsilon + \delta_n)},  E_{-(\epsilon +\delta_n)})=b(\sigma(E_{(\epsilon + \delta_n)}), \sigma( E_{-(\epsilon +\delta_n)}))= b(E_{(\epsilon - \delta_n)},E_{-(\epsilon - \delta_n)})$,

hence, $ b(E_{(\epsilon + \delta_n)},  E_{-(\epsilon +\delta_n)})H_{(\epsilon + \delta_n)}-b(E_{(\epsilon - \delta_n)},E_{-(\epsilon - \delta_n)})H_{(\epsilon -\delta_n)}$

 $= b(E_{(\epsilon + \delta_n)},  E_{-(\epsilon +\delta_n)})(H_{(\epsilon + \delta_n)}-H_{(\epsilon -\delta_n)})$
 $=  b(E_{(\epsilon + \delta_n)},  E_{-(\epsilon +\delta_n)}) H_{2\delta_n} $, and

$[E_{(\epsilon \pm \delta_n)}^\h, E_{-(\epsilon +\delta_n)_{\vert_\u}}^{\h_0}]= b(E_{(\epsilon + \delta_n)},  E_{-(\epsilon +\delta_n)}) H_{2\delta_n} $

 $[E_{(\epsilon \pm \delta_i)}^\h, E_{-(\epsilon +\delta_n)_{\vert_\u}}^{\h_0}]=[E_{(\pm\delta_i + \epsilon)} , E_{-(\epsilon +\delta_n)}-E_{-(\epsilon -\delta_n)}]= E_{-\delta_n\pm \delta_i}- E_{\pm\delta_i +\delta_n}$.  we recall $b(E_\alpha, E_{-\alpha})\not= 0$.  Thus, we have to show
$w=0$ is the unique solution to  $\tau(E_{\delta_i-\delta_n}- E_{\delta_i +\delta_n)})(w)=0, i=1 \dots n-1$, $\tau(E_{-\delta_i-\delta_n}- E_{-\delta_i +\delta_n)})(w)=0, i=1 \dots n-1$ and
  $  \tau(H_{2\delta_n})(w)=0 $

  \medskip
We observe for $1\leq i\not= j\leq n-1$, that $[E_{\delta_i-\delta_n}- E_{\delta_i +\delta_n)}, E_{-\delta_j-\delta_n}- E_{-\delta_j +\delta_n)}]=-2E_{\delta_i -\delta_j}$ and $[E_{\delta_i-\delta_n}- E_{\delta_i +\delta_n)}, E_{-\delta_i-\delta_n}- E_{-\delta_i +\delta_n)}]=-2H_{\delta_i }$, hence, $\tau(\mathfrak{so}(2n-1))w=0$, thus, $\tau$ restricted to $\mathfrak{so}(2n-1)$ contains the trivial representation, since $\tau$ is a nonscalar representation and the lowest $SO(2)\times SO(2n)$-type of a Discrete Series representation, this is not possible owing to the Harish-Chandra parameter $c_0\epsilon +c_1 \delta_1+\cdots +c_n \delta_n$ is dominant and regular with respect to the system $\Psi_c \cup -\Psi_n$, $\rho_n =-n\epsilon, \rho_c =\sum_{1\leq j \leq n} (n-j) \delta_j$, hence the highest weight for $(\tau,W)$ is $(c_0-n)\epsilon + \sum_{1\leq j \leq n} (c_j-(n-j)) \delta_j $, the regularity and dominance imply $\vert c_0 -n\vert <c_1, c_i>c_{i+1}, i=1, \cdots n-1$, hence $-n+i+c_i>-n+i+1+c_{i+1}$ the Weyl's Theorem to compute the highest weights of the restriction to $\mathfrak{so}(2n-1)$ forces $-n+i+c_i=-n+i+1+c_{i+1}$ a contradiction. Thus, for a  nonscalar irreducible representation, $(\tau,W)$,  it holds

\label{eq:so2n2n-1}$ \mathcal V^{(1)} \cap \mathcal U(\mathfrak{so}(1,2)+\mathfrak{so}( 2n-1))W \cap \mathcal L_{W, \mathfrak{so}(1,2n-1)+\mathfrak{so}(1)}  =\{0\}$.

Therefore, we have completed the construction or non existence of $\Sigma$ and the verification of Table A and Table B.
The following statement gives a neat answer to the structure of the symmetry breaking operators in a very particular case. This case has been considered by \cite[Theorem 5.3]{KP2}.
\begin{lem}\label{lem:sonn-12} We assume $\g\cong \mathfrak{so}(n,2)$, $\h=\mathfrak{so}(n-1,2)$ and $n\geq 3$. Then,

a) For a scalar irreducible representation $(\tau, W)$ for $K$, we have $ \mathcal L_{W,H}\cap \mathcal U(\h_0)W= W+\mathcal V^{(1)} \cap \mathcal L_{W,H}= W+\mathcal V^{(1)} \cap \mathcal U(\h_0)W$.

b) For a non  scalar irreducible representation $(\tau, W)$ for $K$, we have $ \mathcal L_{W,H}\cap \mathcal U(\h_0)W= W$
\end{lem}
Here, $\h_0=\mathfrak{so}(n-1) +\mathfrak{so}(1,2)$
\begin{proof}The hypothesis $n\geq 3$ implies $\g$ is a simple Lie algebra and $\z_\k=\z_\l$.  In \ref{sub:tech} we have  verified $W\subseteq  \mathcal L_{W,H}\cap \mathcal U(\h_0)W$. We set some notation: $\p_{\h_0}^-=\C Y$, $\sum_{0\leq r \leq R} c_r [Y^r \otimes w_r], w_r \in W $ a generic element of $\mathcal U(\h_0)W$. We note $[Y^r \otimes w_r] \in \mathcal V^{(r)}$. In \ref{sub:triplepairszero} we noted $[[\p_\h^+, \p_{\h_0}^-],\p_{\h_0}^-]\not=\{0\}$. Thus, there is  $ X_0\in \p_{\h}^+ $ with $[[X_0,Y],Y] \not=0$.  a) In Proposition~\ref{sub:7} we have shown for a scalar representation  $ \mathcal L_{W,H}\cap \mathcal U(\h_0)W \supseteq W+\mathcal V^{(1)} \cap \mathcal L_{W,H}= W+\mathcal V^{(1)} \cap \mathcal U(\h_0)W$. We are to show the reverse inclusion. The
reverse inclusion  follows from the statement $L_X^\tau (\sum_{0\leq r \leq R} c_r [Y^r \otimes w_0])=0 \, \forall X \in \p_\h^+$ implies $c_2=\cdots =c_R=0$. Here, $W=\C w_0$. In \ref{eq:formulapiX2} we computed  $L_X^\tau (\sum_{0\leq r \leq R} c_r [Y^r \otimes w_0])=  \sum_{1\leq r \leq R} r c_r  [Y^{r-1} \otimes \tau([X,Y])w_0]+ \sum_{2\leq r \leq R} \binom{r}{2} c_r [Y^{r-2}[[X,Y],Y] \otimes w_0]  $. Now $[X_0,Y]\in [\p_\h^+, \p_{\h_0}^ -] =\q_\C \cap \k_\C \subset \k_{ss}$ ($\g$ is simple!) and $\tau$ is a scalar representation, hence the first summand vanishes.   Also,   \ref{prop:triplein} yields   $0\not= [[X_0,Y],Y] \in \p_\h^-$, finally, the direct sum decomposition $\p_\g^-=\p_{\h_0}^- \oplus \p_{\h}^-$ forces the vectors $Y^{r-2}[[X_0,Y],Y] , r\geq 2$ are linearly independent in $S(\p_\g^-)$. Therefore, the equality  $L_X^\tau (\sum_{0\leq r \leq R} c_r [Y^r \otimes w_0])=0 \, \forall X \in \p_\h^+$ implies $c_2=\cdots =c_R=0$.

 b) \, For a non scalar representation in \ref{eq:so2n2n-1} we have verified for   $n=2k$,  $ \mathcal V^{(1)} \cap \mathcal U(\mathfrak{so}(1,2)+\mathfrak{so}( 2k-1))W \cap \mathcal L_{W, \mathfrak{so}(1,2k-1)+\mathfrak{so}(1)}  =\{0\}$, and in \ref{sub:interequalceronosigma} we have checked for $n=2k+1$ , $\mathcal V^{(1)}\cap \mathcal L_{W,H}^c \cap \mathcal U(\h_0)W=  \mathcal V^{(1)}\cap \mathcal L_{W, \mathfrak{so}(2,2k-1)+\mathfrak{so}(1) }^c \cap \mathcal U(\mathfrak{so}(2,1)+\mathfrak{so}(2k-1))W =\{0\}$. We are to show $\mathcal L_{W, \mathfrak{so}(2,2k-1)+\mathfrak{so}(1) }^c \cap \mathcal U(\mathfrak{so}(2,1)+\mathfrak{so}(2k-1))W =W$. Equivalently.\\ $\sum_{0\leq r \leq R} c_r [Y^r \otimes w_r] \in \mathcal L_{W, \mathfrak{so}(2,2k-1)+\mathfrak{so}(1) }^c \cap \mathcal U(\mathfrak{so}(2,1)+\mathfrak{so}(2k-1))W$ yields $c_1=\cdots =c_R=0$. In this case $L_X^\tau (\sum_{0\leq r \leq R} c_r [Y^r \otimes w_r])= \sum_{1\leq r \leq R} r c_r  [Y^{r-1} \otimes \tau([X,Y])w_r]+ \sum_{2\leq r \leq R} \binom{r}{2} c_r [Y^{r-2}[[X,Y],Y] \otimes w_0]  $. Now, in $S(\p_\g^-)$ the set of vectors $\{Y^s, s\geq 1, Y^{t}[[X_0,Y],Y], t\geq 1\}$ is linearly independent owing to $Y\in \p_{\h_0}^- $ and  $[[X_0,Y],Y]\in \p_\h^-$ (see \ref{prop:triplein} and part a)). Thus $c_1=\cdots =c_R =0$.

\end{proof}

  \subsubsection{The equality  $\mathcal U(\h_0)W \cap \mathcal V^{(2)}
   =\mathcal L_{W,H}^c\cap \mathcal V^{(2)}$. } To follow, we present a result analogues
    to the one in  Proposition~\ref{sub:7}.

   \begin{lem} We assume $H/L\rightarrow G/K$ is a holomorphic immersion. Then,
   $\mathcal U(\h_0)W \cap \mathcal V^{(2)}\subseteq \mathcal L_{W,H}^c$, if and only if, $[\p_\h^+, \p_{\h_0}^-] \subseteq Ker(\tau)$ and
$[[\p_{\h_0}^-, \p_\h^+],\p_{\h_0}^- ]$
 is equal to zero.

\end{lem}
\begin{proof} As before $[D\otimes w]$ is the equivalence class in
$\mathcal U(\g) \otimes_{\mathcal U(\t_\C \oplus \p^+)}   W $ for $D\otimes w$. We fix
$ X\in \p_\h^+$, $ Y_1 , Y_2$ linear independent elements of $\p_{\h_0}^- $, $w \in W$. We recall
$\mathcal L_{W,H}^c=\{v \in \mathcal U(\g) \otimes_{\mathcal U(\t_\C \oplus \p^+)}   W : L_X^\tau (v)=0
 \forall X \in \p_\h^+  \} $.
  We compute
\begin{equation*} \begin{split} L_X^\tau(Y_1Y_2\otimes w)& =[XY_1Y_2 \otimes w] \\ & = [Y_1XY_2 + [X,Y_1]Y_2 \otimes w ]
\\ & = [ Y_1Y_2X\otimes w + Y_1[X,Y_2]\otimes w \\ & \quad \quad +Y_2[X,Y_1]\otimes w
+[[X,Y_1],Y_2] \otimes w ]
\\ & = [Y_1Y_2\otimes \tau(X)w + Y_1 \otimes \tau([X,Y_2])w \\ & \quad \quad + Y_2\otimes  \tau([X,Y_1])w +
[[X,Y_1],Y_2] \otimes w]=\\ & = [ Y_1 \otimes \tau([X,Y_2])w \\ & \quad \quad + Y_2\otimes  \tau([X,Y_1])w +
[[X,Y_1],Y_2] \otimes w]. \end{split}
\end{equation*}
Here, $[Y_1Y_2\otimes \tau(X)w]=0$ owing to $\tau(\p_\g^+
) W=0$.
For the direct implication, we assume $\mathcal U(\h_0)W \cap \mathcal V^{(2)}
   \subseteq \mathcal L_{W,H}^c\cap \mathcal V^{(2)}$ and we are to verify: $\tau([X,Y_1])= \tau([X,Y_2])=0$ and $ [[Y_1,X],Y_2]=0$.
The left hand side is zero due to our hypothesis.   The triple bracket belongs to $\p_\h^-$.
 Now, the first summand of the last equality
is equal to zero
owing to our hypothesis $\p_\h^+ \subset \p_\g^+ $. $ w \in  W$ is arbitrary.
  Thus,    since $\mathcal U(\p^-) \otimes W$ represents the totality of equivalence  classes in
$\mathcal U(\g) \otimes_{\mathcal U(\t_\C \oplus \p^+)}   W $,  we are left
with a sum in $\p^- \otimes W$, where, if $[[Y_1,X],Y_2]$   were nonzero,
$Y_2,Y_1, [[Y_1,X],Y_2]$ would be  linearly
independent, but, then $w=0$ which contradicts our hypothesis.
 Whence,
for all $X,Y_1,Y_2$ we have $\tau([X,Y_1])= \tau([X,Y_2])=0$ and $ [[Y_1,X],Y_2]=0$.
Thus, the direct implication in the Lemma follows. For the converse statement, the hypothesis implies
$L_X^\tau(Y_1Y_2\otimes w)=0$, whence the inclusion $\mathcal U(\h_0)W \cap \mathcal V^{(2)}
   \subseteq \mathcal L_{W,H}^c\cap \mathcal V^{(2)}$ holds, the properties of the map $D$ \cite{OV3}
    concludes the  verification of the Lemma,

     Here, we are using: $Y_1\otimes w_1 +  Y_2\otimes w_2 + Y_3\otimes w_3=0$ and $Y_1,Y_2,Y_3$ are
linearly independent, then $w_1=w_2=w_3 =0$. \end{proof}
As a corollary we obtain:
\begin{cor}Whenever $\tau$ is one dimensional or as in Proposition~\ref{sub:7} and
 $[[\p_\h^+, \p_{\h_0}^- ], \p_{\h_0}^- ]=\{0\}$, then $\mathcal U(\h_0)W
   =\mathcal L_{W,H}^c $.
\end{cor}
 \begin{cor}Whenever $\tau$ is one dimensional or as in Proposition~\ref{sub:7} and
 $[[\p_\h^+, \p_{\h_0}^- ], \p_{\h_0}^- ]\not=\{0\}$, then $\mathcal U(\h_0)W
   \not=\mathcal L_{W,H}^c $.
\end{cor}

\subsubsection{Two  examples of $\mathcal V^{(n)}\cap \mathcal L_{W,H}^c \cap \mathcal U(\h_0)W=0 $ for all $n\geq 2 $}\label{:sub:twoexam} For this two examples we consider the triples a) $\g=\mathfrak{so}(2,2n)$, $\h=\mathfrak{so}(2,2n-1)+\mathfrak{so}(1)$, $\h_0= \mathfrak{so}(2,1)+\mathfrak{so}( 2n-1)$ and b)  $\g=\mathfrak{so}(2,2n+1)$, $\h=\mathfrak{so}(2,2n)+\mathfrak{so}(1)$, $\h_0=\mathfrak{so}(2,1)+\mathfrak{so}(2n)$. To follow, we verify:

For a nonscalar irreducible representation $(\tau, W)$, then,  for $n\geq 1$,   $\mathcal V^{(n)}\cap \mathcal L_{W,H}^c \cap \mathcal U(\h_0)W= \{0\}$.

For a scalar representation $(\tau,W)$, we have $\mathcal V^{(1)}\cap \mathcal L_{W,H}^c = \mathcal V^{(1)}\cap \mathcal U(\h_0)W$ and for $n\geq 2$,   $\mathcal V^{(n)}\cap \mathcal L_{W,H}^c \cap \mathcal U(\h_0)W= \{0\}$.

 The case $n=1$ has been computed in the proof of Proposition~\ref{prop:nu1nonzero}, Proposition~\ref{sub:7}, whence, from now on, $n\geq 2$. In both cases $\h_0 \equiv \mathfrak{su}(1,1)$ plus a compact Lie algebra. In both cases, we fix a linear basis $E,Y,Z$ of the complexification of $\h_0$ so that  $ E \in \p_{\h_0}^+, Y \in \p_{\h_0}^-, Z \in \mathfrak u \cap \h_0$. Then, in both cases $\mathcal V^{(n)}  \cap \mathcal U(\h_0)W=\{ [Y^n \otimes w] : w \in W \}$. We recall the formula \ref{eq:formulapiX}  for the action of $X \in \p_{\h}^+$ on $[Y^n \otimes w]$, namely, $L_X^\tau ([Y^n \otimes w])=[n Y^{n-1} \otimes \tau([X,Y])w +\binom{n}{2}Y^{n-2} [[X,Y],Y]\otimes w]$. Now, for the two triples $(\g,\h,\h_0)$ we have $[\p_\h^+, \p_{\h_0}^-], \p_{\h_0}^-] \not=\{0 \}$. Thus, we may choose $X \in \p_\h^+$ such that $[[X,Y],Y]\not= 0$. Then, the vectors $Y^{n-1} \in S^{n-1}(p_{\h_0}^-) \subset S^{n-1}(p_{\g}^-)   ,  Y^{n-2}[[X,Y],Y] \in S^{n-1}(p_{\h_0}^- \oplus p_{\h}^- ) \subset S^{n-1}(p_{\g}^-) $  are linearly independent (we recall $[[X,Y],Y] \in p_{\h}^-$), hence,  $L_X^\tau ([Y^n \otimes w])=0$ if and only if $w=0$.  Whence, for either a scalar or non scalar representation $(\tau,W)$, we obtain   $\mathcal V^{(n)}\cap \mathcal L_{W,H}^c \cap \mathcal U(\h_0)W=\{0\}$.
 \subsubsection{} A list of the split rank  one  pairs is in \cite[Table 2]{KP2}.
From the   list \ref{sub:triplepairszero}  and for $\g$ simple, we extract that the split rank
one symmetric pairs     satisfying the condition
$[[ \p_{\h_0}^+,\p_\h^- ],\p_{\h_0}^+]=\{0\}$, are:
$(\mathfrak{su}(m,n), \mathfrak s(\mathfrak{u} (m,n-1) + \mathfrak u(1)))$,
$(\mathfrak{so}(2m,2), \mathfrak u(m,1))$,    $(\mathfrak{so}^\star(2n),
\mathfrak{so}(2) +\mathfrak{so}^\star(2n-2)) $.
Whence, a consequence of Proposition~\ref{prop:equal}
is the conclusion of \cite[Theorem 5.3]{KP2} for scalar holomorphic Discrete
 Series representations.    That is,   for the three pairs quoted above, every symmetry breaking operator
is represented by a normal derivative differential operator.

\subsubsection{{\bf Tensor product}} \label{sub:tensorpro}As usual, we consider $G_0,K_0$ so that $G_0/K_0$ is
equivalent to a bounded symmetric domain. Henceforth $G_0$ is a simple Lie group. We set $G := G_0\times G_0, K:=K_0\times K_0$,
   $\theta =\theta_0 \times \theta_0$, $\p =\p_0 \times \p_0$. Let $H:=\{(x,x): x\in G_0 \}$. Then, $H$ is a subgroup of $G$, invariant under the Cartan involution of $G$. Actually, $H$ is the fix point subgroup for the involution $\sigma(x,y)=(y,x)$. A maximal compact subgroup $L$ of $H$ is $L=K\cap H=\{(x,x): x\in K_0 \}$. $\q=\{(x,-x): x\in \g_0 \}$, $\q_\C \cap \k_\C=\{(x,-x): x\in (\k_0)_\C \}$.  The associated pair  to $(G,H)$ is $(G, H_0)$ with $H_0$ the image of the immersion of $G_0$ in $G$ via the map $x \rightarrow (x, \theta_0 x).$   Thus, the pair $(H_0, L)$ is isomorphic to the symmetric pair $(G_0, K_0).$  $\p_{\h_0}=\h_0 \cap \p=\{ (X,-X): X \in \p_0 \}$. Let $T_0 \subset K_0$ be a maximal torus in $K_0$.  A compact Cartan for $G$ is $T:=T_0 \times T_0$.   We denote the holomorphic system of positive roots    for $\Phi(\g_0, \t_0) $ by $\Psi_{\g_0}$. Hence $\Psi := \Psi_{\g_0} \times \{0\} \cup \{0\} \times \Psi_{\g_0}$ is a holomorphic system of positive roots for $\Phi(\g, \t) $. A maximal torus in $L$ is $U=H\cap T= \{(x,x): x\in T_0 \}$. $\Psi_\h =\{(\alpha, 0): \alpha \in \Psi_{\g_0} \}$. $\Psi_{\h_0} =\{(\alpha, 0): \alpha \in \Psi_{\g_0} \}$. Respective, root vectors are:
$Y_{(\alpha, 0)}=(Y_\alpha, Y_\alpha)$, $Y_{(\alpha, 0)}^0=(Y_\alpha, -Y_{\alpha})$.

  We also fix $(L_\cdot^\eta,  V_\eta), (L_\cdot^\nu, V_\nu)$ holomorphic Discrete Series representations for $G_0$.
   $V_{\eta,\nu}:= V_\eta \otimes V_\nu $. Then, $V_{\eta,\nu}$ is a unitary irreducible holomorphic Discrete Series representation for $G$.
 Let $(\tau, W)$ denote the lowest $K$-type for $V_{\eta,\nu}$. Hence, $(\tau, W)$ is equivalent to $(L_\cdot^{\eta \otimes \nu}, V_{\eta,\nu}^{\p_\g^+})$ Under the above setting, a simple fact is,
\begin{prop} \label{prop:nu1fortensor} The equality
 $\mathcal V^{(1)} \cap \mathcal U(\h_0)W =\mathcal V^{(1)} \cap \mathcal L_{W,H}^c$ holds  if and only if $\tau$ is a one dimensional representation and $Ker(\tau)$ contains the subalgebra $(\z_\k)_\C  \cap \q_\C$.
\end{prop}
 \begin{proof}The  proof closely  follows the proof of Proposition~\ref{sub:7}. To begin with, we verify $[\p_\h^+, \p_{\h_0}^-]=\q_\C \cap \k_\C$. \\ For this,  we fix generators $(Y_j, Y_j),  (R_j, -R_j) \,  j=1,\dots $ for $\p_{\h}^+ (resp. \,\p_{\h_0}^-) $. Then, $[\p_\h^+, \p_{\h_0}^-] =linspan_\C \{ ([Y_j, R_s], -[Y_j, R_s]) \} $ and $Y_j \in  \p_{\g_0}^+, R_j \in \p_{\g_0}^-$.

  Now, $\q_\C \cap \k_\C =\{ (X,-X): X\in (\k_0)_\C \}$. Obviously, $[\p_\h^+, \p_{\h_0}^-]\subset \q_\C \cap \k_\C$. For the other inclusion, we recall the equality $(\k_0)_\C =[(\p_{\g_0})_\C, (\p_{\g_0})_\C]=[\p_{\g_0}^+, \p_{\g_0}^-] +[\p_{\g_0}^-, \p_{\g_0}^+]= [\p_{\g_0}^+, \p_{\g_0}^-] $.
Thus, for $X\in (\k_0)_\C $,  we have

  $(X,-X)=(\sum_{j,r} c_{j,s} [Y_j,R_s], - \sum_{j,r}  c_{j,s} [Y_j,R_s] )\in [\p_\h^+, \p_{\h_0}^-] $.

  In particular, $(\z_\k)_\C  \cap \q_\C \subseteq   [\p_\h^+, \p_{\h_0}^-]$ (for $\g_0=\mathfrak{su}(1,1)$ they are equal).  We notice,  the ideal in $\k_\C$ spanned by $\q_\C \cap \k_\C$ is equal to $(\k_{ss})_\C +(\z_\k)_\C  \cap \q_\C $. Next, as in the proof of Proposition~\ref{sub:7} if the equality $\mathcal V^{(1)} \cap \mathcal U(\h_0)W =\mathcal V^{(1)} \cap \mathcal L_{W,H}^c$ holds, then $Ker(\tau)$ contains  $[\p_\h^+, \p_{\h_0}^-] $, hence $Ker(\tau)$  contains $\k_{ss}$ and $(\z_\k)_\C  \cap \q_\C$, the irreducibility of $\tau$ forces $\tau$  is a one dimensional representation. The converse statement repeats the proof of Proposition~\ref{sub:7}.
\end{proof}

\begin{examp} For $G_0=SU(1,1)$, an explicit calculation yields   $\mathcal V^{(1)} \cap \mathcal U(\h_0)W
=\mathcal V^{(1)} \cap \mathcal L_{W,H}^c$  if and only if $\tau =\tau_1 \otimes \tau_1$ for $\tau_1 \in \widehat{K_0}$, and,     for $n\geq 2$,  $\mathcal V^{(n)} \cap \mathcal U(\h_0)W
$ is disjoint to $\mathcal V^{(n)} \cap \mathcal L_{W,H}^c$.
\end{examp}
 \begin{rmk} \label{rmk:tensor2}a) For every $G_0$ and $G=G_0\times G_0$ as above,  we always have $[[\p_{\h_0}^+, \p_{\h}^-],\p_{\h_0}^+]\not=\{0\}$ owing to
  $[[(Y_{\beta}, -Y_{\beta}), (Y_{-\beta}, Y_{-\beta})], (Y_{\beta}, -Y_{\beta})]\not= 0  \, \, \text{for \,every \, root } \beta \in \Psi_{\g_0}^n.$ Thus, even in the scalar case, Proposition~\ref{prop:equal} shows that the equality  $\mathcal L_{W,H}^c = \mathcal U(\h_0)W  $ does not hold.

  b) Under the setting in this subsubsection, the duality Theorem becomes \begin{equation*} Hom_H( V_\phi, V_{\eta,\nu})\equiv Hom_{K_0}(Z_\phi, \oplus_{1\leq j \leq S} \, S(\p_{\h_0}^-)\otimes_{\C} Z_j). \end{equation*}
 \end{rmk}Here, $Z_\phi, Z_\eta,  Z_\nu$  denotes the lowest $K_0$-type of $V_\phi, V_\eta,V_\nu$ and $Z_j$ are the irreducible  factors of tensor product  $Z_\eta \otimes Z_\nu$ restricted to the diagonal subgroup determinate  by $\k_0$. That is,  $res_{K_0}(Z_\eta \otimes Z_\nu) =\oplus_j Z_j$

c) The technique developed in \ref{sub:w0irred}b) to produce first order normal derivative symmetry breaking operators does not work for this case, owing to,   for each $Y \in \p_{\h_0}^- $, the subspace  $\{[X,Y], X \in \p_\h^+ \}$ contains semisimple elements, for example, for $Y=(Y_{-\beta}, -Y_{-\beta})$ the semisimple element $([Y_\beta, Y_{-\beta}], [Y_\beta, -Y_{-\beta}])$ belongs to, hence, in principle,the subspace  $\{[X,Y], X \in \p_\h^+ \}$ spans a solvable Lie algebra.

   \section{Acknowledgements}
Part of the research in this paper was carried out within the online research community on Representation Theory and Noncommutative Geometry sponsored by the American Institute of Mathematics. The  authors would like to thank T. Kobayashi for much insight  and inspiration on the problems considered here. Also, we thank Michel Duflo, Birgit Speh, Yosihiki Oshima and Jan Frahm for conversations on the subject. The second author   thanks  Aarhus University for generous support, its hospitality and  excellent  working conditions during the preparation of this paper.

\providecommand{\MR}{\relax\ifhmode\unskip\space\fi MR }
\providecommand{\MRhref}[2]{%
\href{http://www.ams.org/mathscinet-getitem?mr=#1}{#2}
}
\providecommand{\href}[2]{#2}

\end{document}